\documentclass[11pt,reqno]{amsart} 
\usepackage[margin=1in]{geometry} 
\usepackage{amsmath,amsthm,amsfonts,amssymb,amscd}
\usepackage[nobysame]{amsrefs}
\usepackage{enumitem}

\usepackage{times}
\usepackage{esint,stackrel}
\usepackage{enumitem}

\usepackage{xfrac}

\usepackage{cases}

\newcommand{\bound}[2]{{\mathcal B}(#1;#2)}
 
\usepackage{fancyhdr} 
\def\p{\partial} 
\def\eps{\varepsilon}

\newcommand{\abs}[1]{\left|#1\right|}

\newcommand*{\sgn}{\ensuremath{\mathrm{sgn\,}}}

\newcommand{\TT}{\mathbb T}
\newcommand{\OO}{\mathcal O}

\renewcommand*{\tilde}{\widetilde}
\renewcommand*{\hat}{\widehat}
\renewcommand*{\bar}{\overline}

\renewcommand{\epsilon}{\varepsilon}

\newcommand{\kcal}{S}

\newtheorem{theorem}{Theorem}[section]
\newtheorem{lemma}[theorem]{Lemma}
\newtheorem{proposition}[theorem]{Proposition}

\newtheorem{definition}[theorem]{Definition}

\numberwithin{equation}{section}

\newcommand{\be}{\begin{equation}}
\newcommand{\ee}{\end{equation}}

\newcommand{\R}{\mathbb{R}}
\newcommand{\I}{\mathcal I}

\def\aa{ {\mathsf a} }

\def\cir{ \! \circ \! }

\newcommand{\C}{\mathbb{C}}
\newcommand{\Z}{\mathbb{Z}}

\usepackage[colorlinks=true, pdfstartview=FitV, linkcolor=blue,citecolor=blue, urlcolor=blue]{hyperref}

\title{A characteristics approach to shock formation in 2D Euler with azimuthal symmetry and entropy}

\author{Isaac Neal}
\address{Courant Institute of Mathematical Sciences, New York University, New York, NY 10012.}
\email{\href{in577@cims.nyu.edu}{in577@cims.nyu.edu}}

\author{Steve  Shkoller}
\address{Department of Mathematics, University of California Davis, Davis, CA 95616.}
\email{\href{shkoller@math.ucdavis.edu}{shkoller@math.ucdavis.edu}}

\author{Vlad Vicol}
\address{Courant Institute of Mathematical Sciences, New York University, New York, NY 10012.}
\email{\href{vicol@cims.nyu.edu}{vicol@cims.nyu.edu}}

\begin{document}

\begin{abstract}
We provide a detailed analysis of the shock formation process for the  non-isentropic 2d Euler equations in azimuthal symmetry.
We prove that from an open set of smooth and generic initial data, solutions  of Euler form a first singularity or gradient blow-up or shock.  This
first singularity is termed a  H\"{o}lder  $C^{\frac{1}{3}} $ {\it pre-shock}, and our analysis provides the first detailed description of this cusp solution.
The novelty of this work relative to~\cite{BuDrShVi2022} is that we herein consider a much larger class of initial data,  allow for a non-constant initial entropy, allow for a non-trivial sub-dominant Riemann variable, and introduce a host of new identities to avoid apparent derivative loss due to
entropy gradients.  The method of proof is also new and robust, exploring the transversality of the
three different characteristic families
 to transform space derivatives into time derivatives. Our main result provides a  fractional series expansion of the Euler solution about the
 pre-shock, whose coefficients  are computed from the  initial data. 
\end{abstract}

\maketitle


\allowdisplaybreaks


\section{Introduction}

Investigating shock formation and development is one of the central problems of hyperbolic PDE. Establishing {\em shock formation} (gradient blowup) from smooth initial data, in a {\em constructive} manner, is crucial for analyzing the dynamics of the resulting discontinuous shock waves. A precise description of the solution at the {\em pre-shock} (the spacetime set where smooth solutions first form cusps) is what allows for a full characterization of singularity propagation, especially in multiple space dimensions  (see~\S~\ref{sec:motivation} for details).

This paper establishes shock formation for smooth solutions of the non-isentropic two-dimensional compressible Euler equations in azimuthal symmetry. When compared to~\cite{BuShVi2019a} this work gives a detailed description of the solution near the pre-shock as a fractional power series. This paper also goes beyond~\cite{BuDrShVi2022} by establishing shock formation in the {\em non-isentropic} setting, and  with {\em minimal constraints} imposed on the initial data  (see~\S~\ref{sec:motivation} for details).

Beyond the result itself, we develop a new robust proof strategy for establishing shock formation for a complex system of hyperbolic PDEs with multiple wave speeds. Instead of appealing to modulated self-similar analysis (cf.~\cite{BuShVi2019a,BuDrShVi2022}), we use new variables which satisfy  pointwise and integral identities which accurately capture the compressible Euler dynamics (see~\S~\ref{sec:newideas} for details).

\subsection{The compressible Euler equations}

The Euler equations of gas dynamics consist of the three conservation laws for momentum, mass, and energy, given respectively by
\begin{subequations} 
\label{euler-weak}
\begin{align}
\partial_t (\rho u)  + \operatorname{div} (  \rho u\otimes u + p I ) &=0\,,  \label{ee2}\\
\partial_t \rho + \operatorname{div}  (  \rho u ) &=0\,, \label{ee1}\\
\partial_t E + \operatorname{div}  (  (p+ E) u ) &=0 \label{ee3}\,.
\end{align}
In two space dimensions, the focus of this paper, 
 $u :\mathbb{R}^2  \times \mathbb{R}  \to \mathbb{R}^2  $ denotes the velocity vector field, $\rho: \mathbb{R}^2  \times \mathbb{R}  \to \mathbb{R}  _+$ denotes the
strictly positive density function,    $E: \mathbb{R}^2  \times \mathbb{R}  \to \mathbb{R}$ denotes the total energy function, and 
$p: \mathbb{R}^2  \times \mathbb{R}  \to \mathbb{R}$ denotes the pressure function which is related to $(u,\rho, E)$ by the identity
$p = (\gamma-1)  ( E- \tfrac{1}{2} \rho \abs{u}^2 ) $,
where $\gamma>1$ denotes the adiabatic exponent.
For the analysis of the shock formation process, it is convenient to replaced conservation of energy \eqref{ee3} with transport  of entropy
\begin{align}
\p_t S + u \cdot \nabla S =0 \,.
\label{eq:entropy2}
\end{align}
\end{subequations} 
Here,  $S: \mathbb{R}^2  \times \mathbb{R}  \to \mathbb{R}$ denotes the specific entropy, and the  equation-of-state for pressure is written as
\begin{align}
 p(\rho,\kcal) = \tfrac{1}{\gamma} \rho^\gamma e^{\kcal}\,.
 \label{peos}
\end{align}


In preparation for reducing the equations to a more symmetric form, using Riemnann-type variables, we introduce the adiabatic exponent
$$
\alpha = \tfrac{ \gamma -1}{2}  \,
$$
so that
the (rescaled) sound speed reads
\begin{align} 
\sigma = \tfrac{1}{\alpha }  \sqrt{ \sfrac{\p p}{\p \rho} } =  \tfrac{1}{\alpha }  e^{{\frac{\kcal}{2}} } \rho^ \alpha \,.
\label{sigma1}
\end{align} 
With this notation, the ideal gas equation of state \eqref{peos} becomes 
\begin{align} 
p= \tfrac{\alpha ^2}{\gamma} \rho \sigma^2 \,. \label{p1}
\end{align} 
The Euler equations \eqref{ee2}, \eqref{ee1}, \eqref{eq:entropy2}, as a system for $(u, \sigma , \kcal)$ are then given by
\begin{subequations}
\label{eq:Euler2}
\begin{align}
\p_t u + (u \cdot \nabla) u + \alpha \sigma  \nabla \sigma  &=  \tfrac{\alpha }{2\gamma} \sigma^2  \nabla \kcal \,,  \label{eq:momentum2} \\
\partial_t \sigma + (u \cdot \nabla) \sigma  + \alpha \sigma \operatorname{div}u&=0 \,,  \label{eq:mass2} \\
\partial_t \kcal  +  (u \cdot\nabla) \kcal &=0 \,.  \label{eq:entropy22}
\end{align}
\end{subequations}
We let $\omega=\nabla^\perp\cdot  u$ denote the scalar vorticity, and 
define the {\it specific vorticity}  by $ \zeta = \tfrac{\omega}{\rho} $.  A straightforward computation shows that $\zeta$ is a solution to
\begin{align} 
\p_t \zeta  + (u \cdot \nabla) \zeta
=  \tfrac{\alpha }{\gamma} \tfrac{\sigma}{\rho} \nabla^\perp \sigma \cdot   \nabla \kcal   \,.   \label{specific-vorticity}
\end{align} 
The term $\tfrac{\alpha }{\gamma} \tfrac{\sigma}{\rho} \nabla^\perp \sigma \cdot   \nabla \kcal  $ on the right side of \eqref{specific-vorticity}
can also be written as $\rho^{-3}  \nabla^\perp \rho \cdot  \nabla p$ and is referred to as {\it baroclinic torque}. 

The  goal of this paper is to give a constructive proof of shock formation for \eqref{eq:Euler2}, from smooth initial data, via a method  powerful enough to capture a high-order series expansion of all fields at the preshock, information which is in turn necessary to study the shock development problem.  More precisely, we prove: 

\begin{theorem}[\bf Main result, abbreviated]
\label{thm:main:soft}
From smooth, non-isentropic initial data with azimuthal symmetry lying in an open set\footnote{ See \S~ \ref{sec:assumptions}-\ref{sec:thm} for the details of the pertinent set of initial data.}, there exist smooth solutions to the
2d Euler equations~\eqref{euler-weak} that form a gradient blowup singularity at a computable time $T_*$\footnote{ We abuse notation here, because the time $T_*$ used here differs from the time $T_*$ referenced in the rest of the paper by a constant dependent on $\gamma>1$. See \S~\ref{sec:changevar}.} and spatial location. More specifically, there exists $\xi_* \in \TT$ such that when the 2d Euler equations are expressed in polar coordinates as in \eqref{eq:Euler:polar}, the azimuthal component of the flow $u_\theta$ and the sound speed $\sigma$ form  $C^{0,\frac{1}{3}}$ cusps along the ray $\theta=\xi_*$ at the time of the blowup, and are given by the fractional series expansions  
\begin{align*}
u_\theta(r,\theta,T_*) & = r\big( \mathsf {b}_0 + \mathsf{b}_1 (\theta-\xi_*)^{1/3} + \mathsf{b}_2(\theta-\xi_*)^{2/3} + \OO(\eps^{-1}|\theta-\xi_*|) \hspace{1mm} \big), \\
\sigma(r,\theta, T_*) & = r\big( \mathsf{c}_0 + \mathsf{b}_1 (\theta-\xi_*)^{1/3} + \mathsf{b}_2(\theta-\xi_*)^{2/3} + \OO(\eps^{-1}|\theta-\xi_*|) \hspace{1mm} \big),
\end{align*}
for $\theta$ in a neighborhood of radius $\sim \eps^3$,\footnote{Here $\eps^{-1}$ is a large parameter quantifying the absolute size of  slope of the initial data. See~\S~\ref{sec:assumptions} for details.} while the radial component $u_r$ of the flow, the specific entropy $S$, and the specific vorticity $\zeta$ remain $C^{1,\frac{1}{3}}$, with fractional series expansions
\begin{align*}
u_r(r,\theta,T_*) & = r\big( \aa_0 + \aa_3(\theta-\xi_*) + \aa_4(\theta-\xi_*)^{4/3} + \mathcal O( \eps^{-1/2}|\theta-\xi_*|^{5/3}) \: \big), \\
S(r,\theta, T_*) & = \mathsf{k}_0 + \mathsf{k}_3(\theta-\xi_*) + \mathsf{k}_4 (\theta-\xi_*)^{4/3} + \mathcal O(\varepsilon^{-1} |\theta-\xi_*|^{5/3}), \\
\zeta(r,\theta, T_*) & = \mathsf{v}_0 + \mathsf{v}_3(\theta-\xi_*) + \OO(\eps^{-1}|\theta-\xi_*|^{4/3}).
\end{align*}
Here, the constants $\aa_0, \aa_3, \aa_4, \mathsf{b}_0, \mathsf{b}_1, \mathsf{b}_2, \mathsf{c}_0, \mathsf{k}_0, \mathsf{k}_3,\mathsf{k}_4,$ and $\mathsf{v}_0$ are   $\OO(1)$ while $\mathsf{v}_3$ is $\OO(\eps^{-1})$.\footnote{ See \S~\ref{sec:notation} for the details of our use of $\OO(\cdot)$ and $\sim$.}
\end{theorem}

\subsection{Motivation and prior results}
\label{sec:motivation}

We recall that the classical proofs of finite-time singularity formation for the compressible Euler equations and related hyperbolic systems are not constructive (see e.g.~\cite{lax1964development}, \cite{liu1979development}, \cite{sideris1985formation}). A constructive proof of blowup, and equally importantly, a detailed description of the  solution at the pre-shock, is necessary in order to establish shock development.  This is especially true in multiple space dimensions: while the theory of weak solutions for 1D hyperbolic 
systems is well-developed (see e.g.~\cite{dafermos2005hyperbolic}), many of the techniques used in the 1D theory either do not apply in multiple space dimensions\footnote{ For example, the BV estimates utilized in the classical theory of shocks for 1D hyperbolic systems fails for $d\geq 2$. See \cite{rauch1986bv}. }  or   are not precise enough to be useful for the  shock development problem, which requires bounds on derivatives of the solution.

Departing from the weak solutions perspective, Lebaud \cite{lebaud1994description} established shock formation and development for the one-dimensional $p$-system (a variant on 1D isentropic Euler). These results were expanded upon by Chen and Dong \cite{chen2001formation} and Kong \cite{Kong2002}. Studying shock development in the $p$-system does not per se prove anything about physical solutions of Euler, because physical solutions of Euler that have shocks cannot be isentropic (see \S~2.2 of \cite{BuDrShVi2022} or \S~3 of \cite{BDSV_EMS} for details). Moreover, non-isentropic solutions of Euler are generically not irrotational due to a misalignment of pressure and entropy gradients (see \eqref{specific-vorticity} above, and \S~4 of \cite{BDSV_EMS} for the 3D case), so physical solutions which have shocks are also generically not irrotational. Studying shock development for piecewise isentropic or even piecewise irrotational solutions of Euler is called the {\it{restricted shock development}} problem. For the restricted shock development problem, Christodoulou established shock formation and development for irrotational flows in his landmark books~\cite{christodoulou2007formation},~\cite{christodoulou2019shock}.  Yin~\cite{yin2004formation} wrote  the first  paper establishing shock formation and development for non-isentropic Euler, but confined to spherical symmetry (see also~\cite{christodoulou2016shock}). 
Luk and Speck~\cite{luk2018shock}  proved  shock formation for the 2D isentropic Euler equations in the presence of vorticity, by generalizing  Christodoulou's geometric framework.


A different perspective was taken by Buckmaster, Shkoller, and Vicol~\cite{BuShVi2019a,BuShVi2019b,BuShVi2020} who used modulated self-similar variables to construct the  first gradient singularity  (a {\em point shock}) from generic  smooth initial data. In~\cite{BuShVi2019a} they constructed shocks for 2D isentropic Euler in azimuthal symmetry and characterized the shock profile as an asymptotically self-similar, stable 1D blowup profile. After that, they proved for the first time that the 3D isentropic Euler equations  generically   form  a stable {\it{point shock}}, even in the presence of vorticity~\cite{BuShVi2019b}. The important generalization to the full non-isentropic setting was achieved in~\cite{BuShVi2020},
where it is also shown that irrotational data instantaneously creates vorticity, which remains uniformly bounded at the point shock. Later, Luk and Speck~\cite{luk2021stability} generalized their 2D result  to the full 3D non-isentropic setting. Going beyond the first point-singularity, Abbrescia and Speck~\cite{abbrescia2022emergence} recently tackled the problem of maximal development for non-irrotational, non-isentropic Euler.  Using {\it rough foliations}, they obtained a description of the stable formation of a pre-shock in a subset of spacetime where the normal derivative of the foliation density has a favorable sign. Using a {\it smooth spacetime geometry}, based on the Arbitrary Lagrangian Eulerian  description of fluids, the maximal hyperbolic development of smooth Cauchy data for Euler has been studied in~\cite{SteveVladAwesome}.

Buckmaster, Drivas, Shkoller, and Vicol~\cite{BuDrShVi2022} established for the first time shock developement  in the presence of voriticity, by working in azimuthal symmetry.  By improving upon~\cite{BuShVi2019a}, the solution at the pre-shock is described in~\cite{BuDrShVi2022} by a fractional series,  assuming that the flow is initially isentropic ($k_0 \equiv 0$ in \eqref{xland-k} below) and that the subdominant Riemann variable vanishes ($z_0\equiv 0$ in \eqref{xland-z} below). They then used this detailed description of the solution to establish shock development for 2D Euler  within the class of azimuthal solutions. The paper~\cite{BuDrShVi2022} is the first to also confirm the production of both a discontinuous shock wave and two surfaces of cusp singularities emanating from the pre-shock,  as predicted by Landau and Lifschitz~\cite{landau1987fluid}.


\subsection{New ideas}
\label{sec:newideas}

This paper breaks with~\cite{BuShVi2019a} and~\cite{BuDrShVi2022} by forgoing the use of self-similar variables. Instead, we use only the fine structure of the Euler system written in the characteristic coordinates that correspond to the three different wave-speeds present in the system.  We show that the sound speed remains bounded from below up to the time of the first blowup (see Proposition \ref{prop:small:time}), which means that the three wave speeds remain uniformly transverse to one another up to the blowup time. This transversality allows us to to prove useful integral bounds (see Lemma \ref{lem:integralbound} and \S~\ref{sec:estimates1}) and allows us to exchange space derivatives for time derivatives (see \S~\ref{sec:transfunc}), which can be integrated to obtain identities for the higher order derivatives of our variables. This exchange of space for time derivatives via transversality is the key new idea of this work. 

The implementation of this idea is made possible by using the special {\em differentiated Riemann variables} introduced in~\cite{BuDrShVi2022}. These  new variables, labeled $q^w$ and $q^z$, evolve along the characteristics of the fastest and slowest wave speeds respectively, and they do not experience  derivative loss (see \S~3 of~\cite{BuDrShVi2022} or \S~\ref{sec:new-variables} below). Whereas~\cite{BuDrShVi2022} utilized $q^w$ and $q^z$   for studying shock development, we use $q^w$ and $q^z$ to also establish shock formation in the non-isentropic setting. Using pointwise and integral identities for $q^w$ and $q^z$, we are able to obtain estimates for our variables and their derivatives up to the blowup time without first establishing the uniqueness or location of the blowup label; we instead derive the uniqueness and location of the blowup label as a result of our estimates (see \S~\ref{sec:x_*}).

We note that because we avoid self-similar analysis we are able to place far fewer assumptions on our initial data than in~\cite{BuDrShVi2022}. When compared to ~\cite{BuDrShVi2022}, we also obtain a higher order fractional series expansion of the solution at the time of blowup (see Theorem \ref{thm:w:z:k:a}).

\section{Azimuthal symmetry}

\label{azisec}

\subsection{The Euler equations in polar coordinates and azimuthal symmetry}
\label{sec:changevar}

The 2D Euler equations \eqref{eq:Euler2}  take the following form in polar coordinates for the variables $ ( u_\theta, u_r, \rho, \kcal)$:
\begin{subequations}
\label{eq:Euler:polar}
\begin{align}
\left(\partial_t  + u_r\partial_r + \tfrac{1}{r} u_\theta \partial_{\theta}\right) u_r -\tfrac{1}{r}u_{\theta}^2+ \alpha  \sigma \partial_r  \sigma&= \tfrac{\alpha }{2 \gamma }\sigma^2 \p_r \kcal  \,, \\
\left(\partial_t  + u_r\partial_r +\tfrac{1}{r} u_\theta \partial_{\theta}\right)u_\theta+\tfrac{1}{r}u_r u_\theta + \alpha \tfrac{ \sigma}{r}\partial_\theta\sigma&=
\tfrac{\alpha }{2 \gamma }\tfrac{\sigma^2}{r} \p_\theta \kcal  \,,  \\
\left(\partial_t  + u_r\partial_r + \tfrac{1}{r} u_\theta \partial_{\theta}\right) \sigma + \alpha\sigma\left( \tfrac{1}{r} u_r + \partial_r u_r + \tfrac{1}{r} \p_\theta u_\theta \right)  &=0 \,,\\
\left(\partial_t  + u_r\partial_r + \tfrac{1}{r} u_\theta \partial_{\theta}\right) \kcal &=0.
\end{align}
\end{subequations}
We introduce the new variables\footnote{ Note that our symmetry constraints make $S$ discontinuous at the origin unless $S$ is constant. For this reason, a classical solution of the 2D Euler equations~\eqref{eq:Euler2} is recovered from the azimuthal variables $(a,b,c,k)$ via \eqref{scale0} on the punctured plane. Alternatively, we may restrict the domain of evolution for 2D Euler to an annular domain pushed forward under the flow of $u$ (see~\cite[\S~2.1]{BuShVi2019a}).}
\begin{equation}\label{scale0}
u_\theta(r,\theta,t) = r b( \theta, t) \,, \ \ \ u_r(r,\theta,t) = r a( \theta, t)\, , \  \ \ \sigma(r,\theta,t) =r  c(\theta,t),  \ \ \  \kcal(r,\theta,t) =k(\theta,t) \,.
\end{equation} 
The system \eqref{eq:Euler:polar} then takes the form
\begin{subequations}
\label{eq:Euler:polar3}
\begin{align}
\left(\partial_t  + b\partial_{\theta}\right) a  + a^2-b^2+ \alpha    c^2 &=0 \label{g3_a_evo}\\
\left(\partial_t  + b\partial_{\theta}\right)b + \alpha   c \partial_\theta c +2a b&=\tfrac{\alpha }{2 \gamma } c^2 \p_\theta k \\
\left(\partial_t  + b\partial_{\theta}\right)  c +   \alpha  c  \p_\theta b + \gamma a  c &=0 \,  \label{sigma-eqn} \\
\left(\partial_t  + b\partial_{\theta}\right) k&=0 \, .
\end{align}
\end{subequations}
For simplicity of the presentation, we will set $\gamma = 2$ from here on; note however that all statements in this paper apply {\em mutatis mutandis} to the case of a general $\gamma>1$. The Riemann functions  $w$ and $z$  are defined by
\begin{subequations} 
\label{eq:riemann}
\begin{alignat}{2}
w&= b+   c  \,, \qquad  &&z= b-  c  \,, \\
b&= \tfrac{1}{2} (w+z) \,, \qquad &&  c= \tfrac{1}{2} (w-z) \,.
\end{alignat}   
\end{subequations} 
It is convenient to rescale time, letting $\p_t \mapsto \tfrac{3}{4} \p_{\tilde t}$, and for notational simplicity, we continue to write $t$ for $\tilde t$. 
With this temporal rescaling employed,  the system  \eqref{sigma-eqn} can be equivalently
written as 
\begin{subequations} 
\label{eq:w:z:k:a} 
\begin{align}
\p_t w + \lambda_3 \p_\theta w & = - \tfrac{8}{3}  a w + \tfrac{1}{24}(w-z)^2 \p_\theta k   \,,  \label{xland-w} \\
\p_t z + \lambda_1 \p_\theta z & = - \tfrac{8}{3}  a z + \tfrac{1}{24}(w-z)^2 \p_\theta k   \,,  \label{xland-z} \\
\partial_t k   + \lambda_2 \partial_{\theta} k & = 0  \,, \label{xland-k} \\
\partial_t a   + \lambda_2  \partial_{\theta} a & = - \tfrac43 a^2 + \tfrac{1}{3} (w+ z)^2 - \tfrac{1 }{6} (w- z )^2  \,. \label{xland-a} 
\end{align}\end{subequations} 
where the three wave speeds are given by 
\begin{align}
\lambda_1 &=  \tfrac{1}{3} w + z  \quad \,<\, \quad
\lambda_2 = \tfrac{2}{3} w+  \tfrac 23 z \quad \,<\, \quad
\lambda_3 =  w+  \tfrac{1}{3}  z\,.   \label{eq:wave-speeds}
\end{align} 
We note that \eqref{sigma-eqn} takes the form
\begin{align} 
\partial_t  c  + \lambda_2  \partial_{\theta}  c +\tfrac{1}{2} c \p_\theta \lambda _2  & = -\tfrac{8}{3}a  c \,. \label{xland-sigma}
\end{align} 

Finally, we denote the specific vorticity \eqref{specific-vorticity} in azimuthal symmetry by
\begin{align}
\label{xland-svort:def}
\varpi = 4 (w + z - \p_\theta a) c ^{-2} e^k,
\end{align}
which satisfies the evolution equation
\begin{align} 
\p_t \varpi + \lambda_2\p_\theta \varpi =   \tfrac{8}{3} a\varpi  +  \tfrac{4}{3}  e^k \p_\theta k  \,. \label{xland-svort}
\end{align}

\subsection{Notation}

\label{sec:notation}

In most of what follows, there will be an important parameter $\eps > 0$, and $ a \lesssim b$ will be used to signify that $a \leq Cb$ for some constant $C$ independent of $\eps$ and any variables $x,\theta,$ or $t$. However, the constant can depend on the implicit constants in the assumptions on the initial data in \S~\ref{sec:assumptions} and can depend on our choice of $\gamma>1$ for the pressure law\footnote{We have already chosen to fix $\gamma=2$ for the entirety of this paper, but our result will hold for arbitrary $\gamma>1$, and the value of $\gamma$ will effect the constants.}. We will use the notation $a \sim b$ to express $a \lesssim b \lesssim a$. We will also write
\begin{align*}
f = \OO(g)
\end{align*}
to express that $|f| \lesssim g$ everywhere in the relevant domain. We will express bounds of the type 
\begin{align*}
    f(x,t) & = \begin{cases} \OO(b_1) \hspace{10mm} |x| \leq \eps^2 \\ \OO(b_2) \hspace{10mm} |x| \geq \eps^2 \end{cases} 
\qquad 
\mbox{simply as}
\qquad
    f = \bound{b_1}{b_2}.
\end{align*}
Often below we will have functions $f$ defined on $\TT \times [0,T_*)$ and maps $\Psi: \TT \times [0,T_*) \rightarrow \TT$, and we will use the notation
$$ f\cir \Psi (x,t) : = f(\Psi(x,t), t). $$
When such an inverse exists, we will write $\Psi^{-1}$ to denote the function such that $\Psi^{-1} \cir \Psi (x,t) = \Psi \cir \Psi^{-1} (x,t) = x$ for all $t$. 

While the spatial variable $\theta$ for  \eqref{eq:w:z:k:a} lies in $\mathbb{T}$, and we will often identify $\mathbb{T}$ with the interval $(-\pi, \pi]$.

\subsection{Assumptions on the Initial Data}
\label{sec:assumptions}

Our initial data will be $w_0,z_0, k_0, a_0 \in H^6(\mathbb{T})$, where $z_0, k_0,$ and $a_0$  all satisfy
\begin{equation}
\|\p^j_x k_0\|_{L^\infty} \lesssim \varepsilon^{\gamma_j}, \hspace{6mm} \|\p^j_x a_0\|_{L^\infty} \lesssim \varepsilon^{\alpha_j}, \hspace{6mm} \|\p^j_x z_0\|_{L^\infty} \lesssim \varepsilon^{\beta_j},
\end{equation}
for $j = 0,1,2,3,4,5$, where $\alpha_j, \beta_j, \gamma_j$ are fixed constants satisfying the relations
\begin{itemize}
\item $\alpha_0, \beta_0, \gamma_0 \geq 0$,
\item $\gamma_1 \geq \mu$, $\alpha_1 \geq 0$,
\item $ \gamma_j \geq \mu - j$ for $j = 2,3,4,5$,
\item $ \alpha_j \geq \mu+ 1-j$ for $j=2,3,4,5$,
\item $ \beta_j \geq \mu-j$ for $j=1,2,3,4,5$.
\end{itemize}
Here $\mu > 0$ is a fixed positive constant which is a lower bound on the $\ell^\infty$ distance of our vector of parameters $(\alpha_2, \hdots, \alpha_5, \gamma_1, \hdots, \gamma_5, \beta_1, \hdots, \beta_5)$ from the boundary of the open set defined by the constraints $\beta_1 > -1$, $\gamma_1 > 0$, etc. Additionally, we assume that $w_0$ satisfies
\begin{enumerate}
\item $w_0 \sim 1$,
\item $w'_0(0) : = -\frac{1}{\varepsilon}$ and $|w'_0(x)| < \varepsilon^{-1}$ for all $x \neq 0$,
\item $w'_0(x) \geq -\frac{1}{\varepsilon} + C \varepsilon^{\frac{\mu}{2}-1}$ for all $|x| \geq \varepsilon^{3/2}$, and some constant $C > 0$.
\item $w'''_0(x) \sim  \varepsilon^{-4}$ for all $|x| \leq \varepsilon^{3/2}$,
\item $|\p^4_x w_0(x)| \lesssim \varepsilon^{\mu -5}$ for all $|x| \leq \varepsilon^{2}$,
\item $\|\p^5_xw_0\|_{L^\infty} \lesssim \varepsilon^{-7}$,
\end{enumerate}
and that $z_0$ satisfies
\begin{equation}
\label{ineq:c_0}
\max z_0 < \min w_0.
\end{equation}

Note that an immediate consequence of our assumptions is that $w_0$ must also satisfy
\begin{itemize}
\item $w''_0(0) = 0$,
\item $|w''_0(x)| \lesssim \varepsilon^{-2}$ for $|x| \leq \varepsilon^2$,
\item $\|w''_0\|_{L^\infty} \lesssim \varepsilon^{-\frac{5}{2}}$, 
\item $\|w'''_0\|_{L^\infty} \lesssim \varepsilon^{-4}$,
\item $ \| \p^4_x w_0 \|_{L^\infty} \lesssim \varepsilon^{-\frac{11}{2}}$.
\end{itemize}

The following additional constraints are not at all necessary for proving our theorem, but they do make the formulas of the proof below cleaner
\begin{align}
 \alpha_0 = \beta_0 = \gamma_0 = \alpha_1 = 0, \hspace{4mm} \beta_1 \leq 0, \hspace{4mm} \text{ and } \hspace{4mm} \alpha_j, \beta_j, \gamma_j \leq 1 \hspace{5mm} \forall \: j =0, 1,2,3,4,5. 
\end{align}

Note that the constraints made here on the first five derivatives of $(w_0,z_0,k_0,a_0)$ are much less stringent than those imposed in  ~\cite{BuDrShVi2022}. In ~\cite{BuDrShVi2022}, the authors assume that $k_0$ is constant, $z_0$ is identically 0, and that $w'_0$ and $a_0$ have support with diameter $\OO(\varepsilon^{1/2})$, among other constraints. Here we do away with such unnecessary hypotheses. Additionally the result of this paper applies to a wide range of parameters $(\alpha_j, \beta_j, \gamma_j)$, whereas in ~\cite{BuDrShVi2022} the authors only work with $(\alpha_0,\alpha_1, \alpha_2, \alpha_3, \alpha_4) =(1,0,0,0)$, which is only one point in our admissible range for these parameters.

In what follows, we will parametrize time so that the initial time is always $t = -\varepsilon$. The local well-posedness theory  of \eqref{eq:Euler2} implies that for any $(w_0,z_0,k_0, a_0) \in H^6(\mathbb{T})$ there exists a time $T_* \in (-\varepsilon, +\infty]$ such that there exists a unique $C^1$ solution $(w,z,k,a)$ of \eqref{eq:w:z:k:a} satisfying $(w,z,k,a) \big\vert_{t=-\varepsilon} = (w_0,z_0,k_0,a_0)$. Furthermore, $(w,z,k,a)$ is guaranteed to be in $C^0([-\varepsilon, T_*) ; H^6(\mathbb{T}))\cap C^1([-\eps, T_*); H^5(\TT))$. Additionally, it follows from the standard theory of \eqref{eq:Euler2} that if $T_* < \infty$ then 
\begin{equation}
\label{LCC}
\int^{T_*}_{-\varepsilon} \|\p_\theta w(t)\|_{L^\infty} + \|\p_\theta z(t)\|_{L^\infty} + \|\p_\theta k(t)\|_{L^\infty} + \|\p_\theta a(t)\|_{L^\infty} \: dt = +\infty.
\end{equation}


The inequalities above can be made into open constraints by making them strict inequalities. While the two pointwise constraints that require $w'_0$ to attain its unique global minimum at $x=0$ and $w'_0(0) = -\frac{1}{\varepsilon}$ are not open constraints, for any suitably small perturbation of initial data $(w_0, z_0, k_0, a_0)$ which satisfying all of the above constraints, one can recover the two pointwise constraints by translating in space and rescaling the solution in time. Since the spacial translation and time rescaling can be made sufficientily small, there exists an open set of initial data around the functions $(w_0, z_0, k_0, a_0)$ described above for which the results of Theorem \ref{thm:w:z:k:a} below still hold. Thus, the shock formation we describe is stable.


\subsection{Statement of the main theorem}
\label{sec:thm}

\begin{theorem}[\bf Main theorem]
\label{thm:w:z:k:a}
For $\mu > 0$, $\varepsilon > 0$ sufficiently small, and initial data $(w,z,k,a) \big\vert_{t= -\varepsilon} = (w_0, z_0, k_0, a_0)$ in the open set described in \S~\ref{sec:assumptions}, there exists a blowup time $T_*$ with $|T_*| \lesssim \varepsilon^{1+\mu}$, a unique blowup location $\xi_* \in \mathbb{T}$, and unique $C^1$ solutions $(w,z,k,a)$ to \eqref{eq:w:z:k:a}  on $\TT \times [-\eps, T_*)$ such that $|x_*| \lesssim \varepsilon^{2+\mu}$, 
$$ w(\cdot, T_*) \in C^{0, \frac{1}{3}}(\mathbb{T}), 
\qquad 
z(\cdot, T_*),  k(\cdot, T_*), a(\cdot, T_*), \varpi(\cdot, T_*) \in C^{1,\frac{1}{3}} (\mathbb{T}), $$
where $\varpi$ is the specific vorticity (see \eqref{xland-svort:def}). Furthermore, there exists a unique blowup label $x_* \in (-\pi, \pi]$ such that $$\lim_{t\rightarrow T_*} \eta(x_*,t) = \xi_*$$ where $\eta$ is the 3-characteristic defined in \S~\ref{sec:char}. In a neighborhood $\theta \in \eta([-\varepsilon^2, \varepsilon^2], T_*)$ of radius $\sim \varepsilon^3$ the functions $w(\cdot, T_*), z(\cdot, T_*), k(\cdot, T_*)$, and $a(\cdot, T_*)$ have the following fractional series expansions:

There exists constants $\aa^w_0, \aa^w_1, \aa^w_2$ with
\begin{align*}
|\aa^w_0| \lesssim 1, 
\qquad 
|\aa^w_1| \lesssim 1, 
\qquad 
|\aa^w_2| \lesssim 1,
\end{align*}
such that
\begin{align}
w(\theta, T_*) & = \aa^w_0 + \aa^w_1(\theta-\xi_*)^{1/3} + \aa^w_2(\theta-\xi_*)^{2/3} + \mathcal O(\varepsilon^{-1}|\theta-\xi_*|), \notag \\
 \p_\theta w(\theta, T_*) & = \tfrac{1}{3} \aa^w_1 (\theta-\xi_*)^{-2/3} + \tfrac{2}{3}\aa^w_2(\theta-\xi_*)^{-1/3} + \mathcal O(\varepsilon^{-1}), \notag \\
\p^2_\theta w(\theta, T_*) & = - \tfrac{2}{9} \aa^w_1(\theta-\xi_*)^{-5/3} -\tfrac{2}{9}\aa^w_2(\theta-\xi_*)^{-4/3} + \OO(\eps^{-1}|\theta-\xi_*|^{-1}),   \notag \\
\p^3_\theta w(\theta, T_*)  & = \tfrac{10}{27} \aa^w_1 (\theta-\xi_*)^{-8/3} + \tfrac{8}{27} \aa^w_2 (\theta-\xi_*)^{-7/3} + \OO(\eps^{-1} |\theta-\xi_*|^{-2}).
\end{align}
There exists constants $\aa^z_0,\aa^z_3, \aa^z_4$ with
\begin{align*}
  |\aa^z_0| \lesssim 1, 
  \qquad 
  |\aa^z_3| \lesssim \varepsilon^{\mu-1}, 
  \qquad 
  |\aa^z_4| \lesssim \varepsilon^{\mu-1},
\end{align*}
such that 
\begin{align}
z(\theta, T_*) & = \aa^z_0 + \aa^z_3 (\theta-\xi_*) + \aa^z_4 (\theta-\xi_*)^{4/3}+ \mathcal O(\varepsilon^{\mu-2} |\theta-\xi_*|^{5/3}), \notag \\
 \p_\theta z(\theta, T_*) & = \aa^z_3 + \tfrac{4}{3} \aa^z_4(\theta-\xi_*)^{1/3}+ \mathcal O(\varepsilon^{\mu-2}|\theta-\xi_*|^{2/3}), \notag \\
\p^2_\theta z(\theta, T_*) & = \tfrac{4}{9} \aa^z_4(\theta-\xi_*)^{-2/3} + \OO(\eps^{\mu-2}|\theta-\xi_*|^{-1/3}),  \notag \\
\p^3_\theta z(\theta, T_*)  & = -\tfrac{8}{27}\aa^z_4 (\theta-\xi_*)^{-5/3} + \OO(\eps^{\mu-2}|\theta-\xi_*|^{4/3}).
\end{align}
There exists constants $\aa^k_0, \aa^k_3, \aa^k_4$ with
\begin{align*}
  |\aa^k_0| \lesssim 1, \qquad |\aa^k_3| \lesssim \varepsilon^\mu, \qquad |\aa^k_4| \lesssim \varepsilon^{\gamma_2+1 \wedge \mu},
\end{align*}
 such that
\begin{align}
k(\theta, T_*) & = \aa^k_0 + \aa^k_3(\theta-\xi_*) + \aa^k_4 (\theta-\xi_*)^{4/3} + \mathcal O(\varepsilon^{\gamma_2 \wedge \mu-1} |\theta-\xi_*|^{5/3}), \notag \\
 \p_\theta k(\theta, T_*) & =  \aa^k_3 + \tfrac{4}{3}\aa^k_4(\theta-\xi_*)^{1/3} + \OO(\varepsilon^{\gamma_2\wedge \mu-1} |\theta-\xi_*|^{2/3}), \notag \\
\p^2_\theta k(\theta, T_*) & = \tfrac{4}{9} \aa^k_4(\theta-\xi_*)^{-2/3} + \OO(\eps^{\gamma_2\wedge \mu-1}|\theta-\xi_*|^{-1/3}),  \notag \\
\p^3_\theta k(\theta, T_*)  & = -\tfrac{8}{27}\aa^k_4 (\theta-\xi_*)^{-5/3} + \OO(\eps^{\gamma_2\wedge \mu-1}|\theta-\xi_*|^{4/3}).
\end{align}
There exists constants $\aa^a_0, \aa^a_3, \aa^a_4$ with
\begin{align*}
  |\aa^a_0| \lesssim 1, \qquad |\aa^a_3| \lesssim 1, \qquad |\aa^a_4| \lesssim 1,
\end{align*}
 such that 
\begin{align}
a(\theta, T_*) & = \aa^a_0 + \aa^a_3(\theta-\xi_*) + \aa^a_4(\theta-\xi_*)^{4/3} + \mathcal O( \eps^{-1}|\theta-\xi_*|^{5/3}), \notag \\
 \p_\theta a(\theta, T_*) & =  \aa^a_3 + \tfrac{4}{3}\aa^a_4 (\theta-\xi_*)^{1/3} + \OO(\eps^{-1}|\theta-\xi_*|^{2/3}), \notag \\
\p^2_\theta a(\theta, T_*) & = \tfrac{4}{9} \aa^a_4(\theta-\xi_*)^{-2/3} + \OO(\eps^{-1}|\theta-\xi_*|^{-1/3}),  \notag \\
\p^3_\theta a(\theta, T_*)  & = -\tfrac{8}{27}\aa^a_4 (\theta-\xi_*)^{-5/3} + \OO(\eps^{-1}|\theta-\xi_*|^{4/3}).
\end{align}
There exists constants $\aa^\varpi_0, \aa^\varpi_3$ with
\begin{align*}
  |\aa^\varpi_0| \lesssim 1, \qquad |\aa^\varpi_3| \lesssim \eps^{-1},
\end{align*}
such that 
\begin{align}
\varpi(\theta, T_*) & = \aa^\varpi_0 + \aa^\varpi_3(\theta-\xi_*) + \mathcal O( \eps^{-1}|\theta-\xi_*|^{4/3}), \notag \\
 \p_\theta \varpi(\theta, T_*) & =  \aa^\varpi_3 + \OO(\eps^{-1}|\theta-\xi_*|^{1/3}), \notag \\
\p^2_\theta \varpi(\theta, T_*) & = \mathcal O( \eps^{-1}|\theta-\xi_*|^{-2/3}), \notag \\
\p^3_\theta \varpi(\theta, T_*)  & = \mathcal O( \varepsilon^{-1} |\theta-\xi_*|^{-5/3}). 
\end{align}
Moreover, the $C^5$ regularity away from the pre-shock is characterized by
\begin{align}
\label{regularity}
\max_{n \leq 5} |\p^n_\theta w(\eta(x,t),t)| + |\p^n_\theta z(\eta(x,t),t)| + |\p^n_\theta k(\eta(x,t),t)| + |\p^n_\theta a(\eta(x,t),t)| \notag \\
\lesssim \bound{\varepsilon^{-7} \big[ \tfrac{1}{2\varepsilon} (T_*-t) + c(\varepsilon+t) \varepsilon^{-4} (x-x_*)^2\big]^{-1}}{\varepsilon^{-16}} .
\end{align}

\end{theorem}

Theorem \ref{thm:main:soft} clearly follows from Theorem \ref{thm:w:z:k:a} as an immediate corollary.

\subsection{Outline of the proof of Theorem~\ref{thm:w:z:k:a}}

In this paper, we will show that the classical solution $(w,z,k,a)$ of  \eqref{eq:w:z:k:a}  with the initial data specified in \S~\ref{sec:assumptions} breaks down in finite time, and that this occurs when the flow $\eta$ of the fastest wave speed $\lambda_3$ ceases to be a diffeomorphism. More specifically, the blowup time $T_*$ will be characterized as the first time $t$ when $\min_x \eta_x(x, t) = 0$. We will also establish that there is a unique Lagrangian label $x_*$ for which $\eta_x(x_*,T_*) = 0$, which will imply that $\eta_{xx}$ vanishes at $(x_*,T_*)$ as well. While $w, z,k,a, \p_\theta z, \p_\theta k,$ and $\p_\theta a$ will be shown to remain bounded on $\TT \times [-\eps, T_*]$, $\p_\theta w$ will be shown to go to $-\infty$ at the point $(\xi_*,T_*) : = (\eta(x_*,T_*), T_*)$ and remain smooth elsewhere. The key ingredient for implementing the above described strategy is to show that the functions $w\circ \eta, z\circ \eta, k\circ \eta$, and $a\circ \eta$ remain as smooth as their initial data, {\em uniformly} up to $T_*$. The authors of~\cite{BuDrShVi2022} proved such uniform estimates using self-similar analysis, but only in a special case.\footnote{ The authors of~\cite{BuDrShVi2022} work in the case where $z$ and $k$ are identically zero and many more constraints are placed on $w_0$ and $a_0$. See \S~\ref{sec:assumptions} above for a discussion.} In this paper, we prove obtain uniform $C^5$ estimates for $(w,z,k,a)\circ \eta$ on $\mathbb{T}\times[-\eps,T_*]$, even in the most general setting, not by relying on self-similar variables, but by instead using the transversality of various families of characteristics. This allows us to also consider a much broader class of initial data than previously considered in~\cite{BuDrShVi2022}. Once we have shown that all the variables stay smooth along the $\eta$ characteristic,  we  obtain our functional description of the solution near $(\xi_*,T_*)$ by  inverting the map  $x\mapsto \eta(x,t)$  for $(x,t)$ near the point $(x_*,T_*)$. In light of the constraints $\eta_x(x_*,T_*) = \eta_{xx}(x_*,T_*) = 0$, this amounts to the inversion of what is to leading order a cubic polynomial, resulting in fractional series expansions of $w,z,k,$ and $a$ near $(\xi_*,T_*)$ in terms of powers of $(\theta - \xi_*)^{1/3}$.

This paper is organized as follows:
\begin{enumerate}
\item In \S~\ref{sec:estimates1} we bound $|T_*|$ and prove that $\p_\theta w$ must become infinite at time $T_*$. We use a simple bootstrap argument to get estimates for $w,z,k,a$ and their first derivatives up to time $\varepsilon \wedge T_*$. Using these estimates, we show that $\eta_x$ must have a zero before time $t=\eps$, and conclude that $|T_*| \lesssim \varepsilon^{1+\mu}$. This implies that $\eps\wedge T_*= T_*$ and therefore all of our estimates and identities hold up to time $T_*$. The fact that $\p_\theta  w$ must blow up then follows immediately from the fact that $\p_\theta z, \p_\theta k,$ and $\p_\theta a$ remain bounded up to time $T_*$ (see \eqref{LCC} above).

\item Next we show that $w \circ \eta, z\circ \eta, k\circ\eta,$ and $a\circ \eta$ remain smooth up to time $T_*$. To do this, we first establish crucial identities in \S~\ref{sec:transfunc} which result from the fact that the waves speeds are uniformly transverse to one another. Then in \S~\ref{sec:2DE} -~\ref{sec:5DE} we prove pointwise bounds on $z,k,a$ and their derivatives in terms of $w$ and its derivatives by analyzing how our new variables evolve along the multiple wave speeds. This allows us to conclude in \S~\ref{sec:3char} that $w,z,k,$ and $a$ all remain smooth along $\eta$.

\item Lastly, we establish that the singularity occurs at a unique point $(\xi_*,T_*) \in \TT \times [-\varepsilon, T_*]$ and we invert $\eta$ near this point to obtain fractional series expansions for $w,z,k,$ and $a$. We do this by establishing in \S~\ref{sec:3char} that there is a unique point $(x_*,T_*) \in \TT \times [-\varepsilon, T_*]$ where $\eta_x$ vanishes and that $\eta_{xx}(x_*,T_*) = 0$ as well. Since $\eta(x,T_*) = \xi_* + \eta_{xxx}(x_*, T_*) (x-x_*)^3 + \OO(|x-x_*|^4)$ near $(x_*,T_*)$, it follows (see \S~\ref{sec:expansion}) that $(x-x_*) \sim (\theta-\xi_*)^{1/3}$ for small $|x-x_*|$ at time $T_*$, and the Taylor series expansions of the smooth functions $w\cir \eta (\cdot, T_*), z\circ \eta(\cdot, T_*), k\circ \eta (\cdot, T_*)$, and $a\cir \eta (\cdot,T_*)$  near $x_*$ become fractional series expansions of $w(\cdot,T_*),z(\cdot,T_*),k(\cdot,T_*),$ and $a(\cdot,T_*)$ near $\xi_*$.
\end{enumerate}


\section{Preliminaries}
\label{sec:prelim}

\subsection{The characteristics}
\label{sec:char}

Let $\varphi > 0$ and let $\Psi$ be the flow of $\lambda : = (1-\varphi) w + (\tfrac{1}{3} + \varphi) z$.
\begin{align}
\Psi_x & = e^{\int^t_{-\varepsilon} \p_\theta \lambda \circ \Psi}. \\
\p_t c + \lambda \p_\theta c & = -(\varphi \p_\theta w + (\tfrac{2}{3}-\varphi) \p_\theta z + \tfrac{8}{3} a)c. \notag
\end{align}
If $c> 0$ everywhere, this tells us that
\begin{align}
 -\tfrac{1}{\varphi} \p_t( \log c\cir\Psi) & = (\p_\theta w + (\tfrac{2}{3}\tfrac{1}{\varphi}-1)\p_\theta z + \tfrac{8}{3} \tfrac{1}{\varphi} a) \cir \Psi. \notag \\
\implies \p_\theta \lambda \cir \Psi & = - \tfrac{1-\varphi}{\varphi} \p_t(\log c\cir\Psi) + ( (2-\tfrac{2}{3}\tfrac{1}{\varphi})) \p_\theta z - \tfrac{8}{3} \tfrac{1-\varphi}{\varphi}  a) \cir\Psi. \notag \\
\implies \Psi_x & = \big( \tfrac{c_0}{c\circ\Psi}\big)^{\frac{1-\varphi}{\varphi}} e^{ \int^t_{-\varepsilon} (2-\frac{2}{3}\frac{1}{\varphi})) \p_\theta z - \tfrac{8}{3} \frac{1-\varphi}{\varphi}  a) \circ\Psi} . \label{eq:flow_x}
\end{align}
If $c \sim 1$ and $\p_\theta z, a$ are bounded, then this lets us conclude that $\Psi_x \sim 1$. We will prove in the next section that $c\sim 1$ and that $\p_\theta z, a$ are indeed bounded on $\mathbb{T} \times [-\varepsilon, T_*)$, so everything that follows is relevant.

In the case where $\varphi = \frac{2}{3}$, we have $\lambda = \lambda_1$, the first wave speed. Let $\psi$ denote the corresponding flow, the so-called {\bf{1-characteristic}}. Its first derivative satisfies 
\begin{align}
\label{psi_x:formula}
\psi_x = \bigg( \frac{c_0}{c\cir \psi} \bigg)^{\frac{1}{2}} e^{\int^t_{-\varepsilon} \! \! (\p_\theta z-\frac{4}{3} a) \circ \psi},
\end{align}
while its second derivative obeys
\begin{align}
\label{psi_xx:formula}
\psi_{xx} & = \psi_x\bigg( \tfrac{1}{2} \frac{c_0'}{c_0} + \int^t_{-\varepsilon} \! \! \psi_x( \p_\theta^2 z - \frac{4}{3} \p_\theta a) \cir \psi \bigg) - \tfrac{1}{2} \psi^2_x \frac{\p_\theta c\cir \psi}{c\cir \psi} \notag \\
& =: \psi_x \Psi - \tfrac{1}{2} \psi^2_x \frac{\p_\theta c\cir \psi}{c\cir \psi}.  \\
& =: \psi^2_x( Q_1- \tfrac{1}{2} c^{-1} \p_\theta c)\cir\psi.
\end{align}

%
%
%


When $\varphi = \frac{1}{3}$, we have $\lambda = \lambda_2$ and the corresponding flow is the {\bf{2-characteristic}}, $\phi$. The first derivative of $\phi$ satisfies
\begin{align}
\label{phi_x:formula}
\phi_x = \bigg( \frac{c_0}{c\cir \phi} \bigg)^2 e^{-\frac{16}{3}\int^t_{-\varepsilon} \! \! a \circ \phi} 
\end{align}
while its second derivative obeys
\begin{align}
\label{phi_xx:formula}
\phi_{xx} & = \phi_x\bigg( 2\frac{c_0'}{c_0} - \tfrac{16}{3} \int^t_{-\varepsilon} \! \! \p_x(a\cir \phi) \bigg) - 2\phi^2_x \frac{\p_\theta c\cir \phi}{c\cir \phi} \notag \\
& =: \phi_x \Phi - 2\phi^2_x \frac{\p_\theta c\cir \phi}{c\cir \phi} \\
& =: \phi^2_x (Q_2- 2c^{-1} \p_\theta c)\cir\phi \label{phi_xx:formula:2}
\,.
\end{align}

%
%
%
%
%

When $\varphi = 0$, we have $\lambda = \lambda_3$ and the corresponding flow is the {\bf{3-characteristics}}, $\eta$. Note that our analysis for $\varphi>0$ breaks down for $\eta$, but also that $w$ is essentially transported along $\eta$.

\subsection{$q^w$ and $q^z$}
\label{sec:new-variables}

Our system \eqref{eq:w:z:k:a} can be written as
\begin{equation}
\label{eq:vec:x}
\p_t \vec{x} + A \p_\theta \vec{x} = \vec{b}
\end{equation}
where
$$ \vec{x} : = \begin{bmatrix} w \\ z \\ k \\ a \end{bmatrix}, \qquad A : = \begin{pmatrix} \lambda_3 & 0 & - \frac{1}{6} c^2 & 0 \\ 0 & \lambda_1 & - \frac{1}{6} c^2 & 0 \\ 0 & 0 & \lambda_2 & 0 \\ 0 & 0 & 0 & \lambda_2 \end{pmatrix}, \qquad \vec{b} : = \begin{bmatrix} -\frac{8}{3} aw \\ -\frac{8}{3} az \\ 0 \\ - \frac{4}{3}a^2 + \frac{1}{3}(w+z)^2 - \frac{1}{6}(w-z)^2 \end{bmatrix}. $$
Taking $\p_\theta$ of \eqref{eq:vec:x} and diagonalizing $A$ gives us
$$ \p_t \vec{y} + D \vec{y} = Q(\vec{x}, \vec{y}) $$
where $D = \text{diag}(\lambda_3, \lambda_1, \lambda_2, \lambda_2)$, $\vec{y} : = (\p_\theta w - \frac{1}{4} c\p_\theta k, \p_\theta z + \frac{1}{4} c\p_\theta k , \p_\theta k, \p_\theta a)$, and $Q: \R^{8} \rightarrow \R^4$ is a third order polynomial. This motivates the introduction of the following variables:
\begin{equation}
q^w : = \p_\theta w - \frac{1}{4} c\p_\theta k \hspace{4mm} \text{ and } \hspace{4mm} q^z : = \p_\theta z + \frac{1}{4}c\p_\theta k .
\end{equation}
On can check using the identities in \S~\ref{sec:id} that
\begin{align}
\p_t\big(  q^w\cir \eta \eta_x) & = ( -\frac{8}{3} a + \frac{1}{12} c \p_\theta k) \cir \eta \big( q^w\cir \eta \eta_x \big) + \frac{1}{12} (c\p_\theta k) \cir \eta \big( q^z\cir\eta \eta_x\big) - \frac{8}{3} \p_x( a\cir \eta) w\cir \eta. \\
\p_t\big(  q^z\cir \psi \psi_x) & = ( -\frac{8}{3} a - \frac{1}{12} c \p_\theta k) \cir \psi \big( q^z\cir \psi \psi_x \big) - \frac{1}{12} (c\p_\theta k) \cir \psi \big( q^w\cir\psi \psi_x\big) - \frac{8}{3} \p_x( a\cir \psi) z\cir \psi. 
\end{align}
If we define
\begin{equation}
\label{eq:I_t}
I_t(x) := e^{\frac{1}{8} k\circ \eta - \frac{8}{3} \int^t_{-\varepsilon} a\circ\eta }
\end{equation}
then our equation for $\p_t(q^w\cir\eta \eta_x)$ gives us the Duhamel formula
\begin{align}
\label{q^w:equation}
\eta_x q^w\cir\eta & = I_t \bigg[ (w'_0 - \tfrac{1}{4} c_0 k'_0) e^{-\tfrac{1}{8} k_0} + \tfrac{1}{12} \int^t_{-\varepsilon} \! \! I^{-1}_\tau \eta_x (c \p_\theta k q^z)\cir\eta \: d\tau - \tfrac{8}{3} \int^t_{-\varepsilon} \! \! I^{-1}_\tau w\cir \eta \p_x(a\cir \eta) \: d\tau \bigg].
\end{align}
It follows immediately from the definitions of $\lambda_3$ and $q^w$ that
\begin{align}
\label{eta_x:equation}
\eta_x  = 1 + \int^t_{-\varepsilon} \eta_x\p_\theta \lambda_3 \cir\eta \: d\tau = 1 + \int^t_{-\varepsilon} \: \eta_x q^w\cir\eta \: d\tau + \tfrac{1}{4} \int^t_{-\varepsilon} \! \! \p_x(k\cir\eta) (c\cir\eta) \: d\tau + \tfrac{1}{3} \int^t_{-\varepsilon} \! \! \p_x(z\cir\eta) \: d\tau.
\end{align}
Identity \eqref{q^z:duhamel} will be used in \S~\ref{sec:z_x}, and \eqref{q^w:equation} and \eqref{eta_x:equation}  will be used in \S~\ref{sec:intbd}, \ref{sec:z_x}, \ref{sec:T_*}, \ref{sec:3char}. Similarly, $q^z\cir\psi \psi_x$ satisfies the Duhamel formula
\begin{align}
\label{q^z:duhamel}
q^z\cir\psi \psi_x & = (z'_0 + \frac{1}{4}c_0k'_0)  e^{-\int^t_{-\varepsilon} (\frac{8}{3} a + \frac{1}{12}c\p_\theta k)\circ\psi}  -\tfrac{1}{12}\int^t_{-\varepsilon} e^{-\int^t_{\tau} (\frac{8}{3} a + \frac{1}{12}c\p_\theta k)\circ\psi}\psi_x(c\p_\theta k q^w)\cir\psi  \notag \\
& - \tfrac{8}{3} \int^t_{-\varepsilon} e^{-\int^t_{\tau} (\frac{8}{3} a + \frac{1}{12}c\p_\theta k)\circ\psi} \psi_x(\p_\theta a z )\cir\psi \: d\tau.
\end{align}

\subsection{Integral bounds}
\label{sec:intbd}

Let $\varphi > 0$ and let $\Psi$ be the flow of $\lambda : = (1-\varphi) w + (\tfrac{1}{3} + \varphi) z$.

\begin{lemma}
\label{lem:integralbound}
Suppose that $T \in [-\varepsilon, \varepsilon \wedge T^*]$ and that for all $(\theta, t) \in \mathbb{T} \times [-\varepsilon, T]$ we have
$$
w \sim 1\,,
\qquad
c\sim 1\,,
\qquad
|k|, |a|  \lesssim 1\,,
\qquad
|\p_\theta z| \lesssim \varepsilon^{\beta_1}\,,
\qquad
|\p_\theta k| \lesssim \varepsilon^{\gamma_1}\,,
\qquad
|\p_\theta a| \lesssim 1\,.
$$
%
Then, for all $\varphi > 0$, we have
\begin{equation}
\label{bd:int}
\int^t_{-\varepsilon} \! \! |\p_\theta w\cir\Psi (x, \tau)| \: d\tau \lesssim \frac{1}{\varphi} \bigg( \frac{\varepsilon+t}{\varepsilon}\bigg)
\end{equation}
with a constant uniform in $\varphi > 0, (x,t) \in \mathbb{T} \times [-\varepsilon, T].$ 
\end{lemma}

\begin{proof}[Proof of Lemma ~\ref{lem:integralbound}]
Fix $\varphi >0$ and define
$$ g(x,t) : = \eta^{-1}(\Psi(x,t),t). $$
We compute that
\begin{align}
\label{g_t}
\p_t g(x,t) & = \p_t \eta^{-1}(\Psi(x,t),t) + \p_t \eta^{-1}( \Psi(x,t),t) \Psi_t(x,t) \notag \\
& = \frac{ - \eta_t(g(x,t),t) + \Psi_t(x,t)}{ \eta_x(g(x,t),t)} \notag \\
& = \frac{-\lambda_3 \cir \eta( (g(x,t),t) + \lambda \cir \Psi(x,t)}{ \eta_x(g(x,t),t)} \notag \\
& = \bigg( \frac{-\lambda_3 \cir \eta + \lambda \cir\eta}{\eta_x} \bigg) ( g(x,t), t) \notag \\
& = - 2\varphi \frac{c\cir\eta}{\eta_x} ( g(x,t),t).
\end{align}
Note that $\p_t g(x,t) < 0$ everywhere. We also know that
\begin{align*}
 \Psi(x,t) & = x + \int^t_{-\varepsilon} \! \! \lambda \cir \Psi(x,\tau) \: d\tau, \hspace{5mm} \text{ and } \\
\Psi(x,t) & = \eta(g(x,t),t) = g(x,t) + \int^t_{-\varepsilon} \! \! \lambda_3 \cir \eta (g(x,t),\tau) \: d\tau, \\ 
\text{so } \hspace{4mm} x-g(x,t) & = \int^t_{-\varepsilon} \! \! \lambda_3 \cir \eta(g(x,t), \tau) - \lambda \cir\Psi(x,\tau) \: d\tau.
\end{align*}

Our hypotheses allow us to conclude (see \eqref{I_t:bound} below) that
\begin{align*}
I_t(x) e^{-\frac{1}{8}k_0(x)} & = 1 + \mathcal O(\varepsilon +t)
\end{align*}
for times $t \in [-\varepsilon, T]$. Using our hypotheses, along with this equation and \eqref{q^w:equation}, we conclude that that for all $(x,t) \in \mathbb{T} \times [-\varepsilon, T]$, we have
\begin{align*}
\sup_{[-\varepsilon, t]} |\eta_x q^w\cir\eta| & \leq \varepsilon^{-1}(1 + \mathcal O(\varepsilon^{\gamma_1})) + \mathcal O(\varepsilon^{0 \wedge \beta_1+\gamma_1})(\varepsilon+t) \sup_{[-\varepsilon, t]} \eta_x.
\end{align*}
Plugging this into \eqref{eta_x:equation} and using the fact that $T \leq \varepsilon$ gives us
\begin{align*}
\sup_{[-\varepsilon,t]} \eta_x & \leq 1 + (\varepsilon+t)\varepsilon^{-1} (1 + \mathcal O(\varepsilon^{\gamma_1})) + \mathcal O(\varepsilon^{\beta_1}) (\varepsilon+t) \sup_{[-\varepsilon, t]} \eta_x \\
& \leq  1 + 2 (1 + \mathcal O(\varepsilon^{\gamma_1})) + \mathcal O(\varepsilon^{\beta_1+1}) \sup_{[-\varepsilon, t]} \eta_x. \notag \\
\implies \sup_{[-\varepsilon,t]} \eta_x & \leq \frac{ 1 + 2 (1 + \mathcal O(\varepsilon^{\mu}))}{1-\mathcal O(\varepsilon^\mu)} \leq 4.
\end{align*}
The last inequality is true for $\varepsilon> 0$ taken to be small enough, since $\mu> 0$. Plugging this into  \eqref{q^w:equation} and letting $\varepsilon$ be sufficiently small gives us
\begin{align}
\eta_x |q^w\cir\eta| & \leq \varepsilon^{-1} ( 1 + \mathcal O(\varepsilon^{\mu}) ) \notag \\
\implies  q^w\cir \eta & = \frac{w'_0 I_t e^{-\frac{1}{8}k_0} + \mathcal O(\varepsilon^{\gamma_1})}{\eta_x}.
\end{align}
It follows that
$$ q^w\cir \Psi(x,t) = -\tfrac{1}{\varphi} \p_t g(x,t) \frac{w'_0(g(x,t)) + \mathcal O(1)}{ 2 c\cir\Psi(x,t)}. $$
Since $\p_t g < 0$, it follows that
$$ |q^w\cir \Psi(x,t)| \lesssim -\tfrac{1}{\varphi} \p_t g(x,t) \tfrac{1}{\varepsilon}. $$
So
$$ \int^t_{-\varepsilon} \! \! |q^w\cir\Psi(x,\tau)| \: d\tau \lesssim \frac{x-g(x,t)}{ \varphi \varepsilon} \lesssim \frac{1}{\varphi} \bigg( \frac{\varepsilon+t}{\varepsilon}\bigg). $$
Our result follows immediately from this inequality and our hypotheses.
\end{proof}

 
 \section{Initial Estimates}
\label{sec:estimates1}

\subsection{Zeroth Order Estimates}
\label{sec:0DE}

\begin{proposition}
\label{prop:small:time}
For $\varepsilon$ small enough the following estimates hold for all $t \in [-\varepsilon, \varepsilon \wedge T_*]$:
\begin{align*}
&w \sim 1\,,
\qquad
c \sim 1\,,
\qquad
\phi_x \sim 1\,,\\
&
\|\partial_\theta k\|_{L^\infty} \lesssim \|k'_0\|_{L^\infty}\,,
\qquad
\|a\|_{L^\infty} \leq \|a_0\|_{L^\infty} +  \mathcal O(\varepsilon)\,,
\qquad
\|z\|_{L^\infty} \leq \|z_0\|_{L^\infty} + \mathcal O(\varepsilon)\,\,.
\end{align*}
\end{proposition}

\begin{proof}[Proof of Proposition~\ref{prop:small:time}]This follows from an easy bootstrap argument. Let $t \leq \varepsilon \wedge T_*$. If we assume all of the listed bounds hold up to time $t$ for some constants, then it follows that \eqref{phi_x:formula} holds up to time $t$. $k$ satisfies $k\cir\phi = k_0$ and 
\begin{equation}
\label{k_x:formula}
\phi_x \p_\theta k\cir \phi = k'_0.
\end{equation}
Additionally, \eqref{eq:w:z:k:a} gives us Duhamel formulas for $w\cir\eta, z\cir \psi$, and $a\cir\phi$. Using these Duhamel formulas along with \eqref{ineq:c_0},\eqref{phi_x:formula}, \eqref{k_x:formula}, and the fact that $t \leq \varepsilon$ it is straightforward to improve our bounds for all times before $t$, provided the constants we assumed in our bootstrap hypothesis are appropriate and $\varepsilon$ is small enough.
\end{proof}

Using these estimates, it is easy to show that for $\varepsilon > 0$ sufficiently small we get
\begin{equation}
\label{I_t:bound}
\big\vert I_te^{-\frac{1}{8}k_0} -1\big\vert   \leq \mathcal O(\varepsilon + t) \hspace{10mm} \forall \: x \in \mathbb{T}, -\varepsilon \leq t \leq \varepsilon\wedge T_*.
\end{equation}

\subsection{ $\partial_\theta a$ bounds}
\label{sec:a_x}

Using \eqref{phi_x:formula} and \eqref{k_x:formula} we have
\begin{align}
\label{vorticity_t:identity}
\p_t(\varpi \cir\phi) & = \frac{8}{3} (a\varpi)\cir\phi + \frac{4}{3} e^{k_0} \p_\theta k \cir\phi  \\
& = \frac{8}{3} (a\varpi)\cir\phi + \frac{4}{3} e^{k_0}k'_0 \phi^{-1}_x \notag \\
& = \frac{8}{3} (a\varpi)\cir\phi + \frac{4}{3} \frac{k'_0}{c^2_0}e^{k_0} \I_t^2 (c\circ \phi)^2, \notag
\end{align}
where
\begin{equation}
\I_t : = e^{\frac{8}{3}\int^t_{-\varepsilon} a\circ\phi}.
\end{equation}
Therefore,
\begin{align}
\label{vorticity:duhamel}
 \varpi\cir\phi & = \varpi_0 \mathcal I_t + \frac{4}{3} c^{-2}_0 k'_0e^{k_0} \I_t \int^t_{-\varepsilon} \I_\tau (c^2 \circ \phi) \: d\tau.
\end{align}
Note that
\begin{equation}
\label{phi_x:formula:2}
\phi_x = c_0^2  I_t^{-2} c^{-2}\cir\phi \mathcal.
\end{equation}
This relation will be useful for estimating the higher derivatives of $a$.

Since 
\begin{equation}
\label{vorticity_0:formula}
\varpi_0 = 4c^{-2}_0(w_0 + z_0 - a'_0) e^{k_0},
\end{equation}
our assumptions on our initial data let us conclude that $|\varpi_0| \lesssim 1$, and therefore for all $(\theta,t) \in \mathbb{T} \times [-\varepsilon, T]$ we have
\begin{equation}
|\varpi| \lesssim 1.
\end{equation}
Since
\begin{equation}
\label{a_x:formula}
\p_\theta a = w+ z-\frac{1}{4}c^2 \varpi e^{-k},
\end{equation}
 it follows that 
 \begin{equation}
 \label{a_x:bound}
 |\p_\theta a| \lesssim 1
 \end{equation}
 for all times $t \in [-\varepsilon, \varepsilon \wedge T_*]$.

Using \eqref{a_x:bound}, and the bounds on the initial data we conclude that 
\begin{equation}
\label{Phi:bound}
|\Phi| \lesssim \varepsilon^{-1},
\end{equation}
for all times $t \in [-\varepsilon, \varepsilon \wedge T_*]$.

\subsection{ $\partial_\theta z$ bounds}
\label{sec:z_x}

\begin{proposition}
\label{prop:z_x}
For all $(x,t) \in \mathbb{T} \times [-\varepsilon, \varepsilon \wedge T_*]$ we have
\begin{align*}
 \eta_x |q^w\cir\eta| \leq 2\varepsilon^{-1}\,,
 \qquad
 \eta_x \leq 4\,,
 \qquad
 |q^z\cir\psi | \lesssim \|z'_0\|_{L^\infty}\,,
 \qquad
 \psi_x \sim 1\,.
\end{align*}
\end{proposition}

\begin{proof}[Proof of ~\ref{prop:z_x}]
We will use a bootstrap argument. Let $T \in [-\varepsilon, \varepsilon \wedge T_*)$ and let our bootstrap assumption be that
\begin{align*}
|q^z| & \leq C\|z_0'\|_{L^\infty}
\end{align*}
for all $(\theta,t) \in \mathbb{T} \times [-\varepsilon, T]$ and a constant $C$ to be determined. Since we are assuming $|\p_\theta z| \lesssim \varepsilon^{\beta_1}$ for times $t \in [-\varepsilon, T]$, it follows from \eqref{psi_x:formula} and our estimates from \S~\ref{sec:0DE} that $\psi_x \sim 1$ with constants independent of $C$ for times $t \in [-\varepsilon, T]$, provided that $\varepsilon$ is small enough relative to $C$.

Using our bootstrap assumption, along with the estimates from \S~\ref{sec:0DE}, \ref{sec:a_x}, we can conclude (see Lemma \ref{lem:integralbound} and its proof) that for all $(x,t) \in \mathbb{T} \times [-\varepsilon, T]$, we have
\begin{align*}
\eta_x & \leq 4, \\
\eta_x |q^w\cir\eta| & \leq 2\varepsilon^{-1}, \\
\int^t_{-\varepsilon} |q^w\cir\psi(x,\tau)| \: d\tau  & \lesssim 1.
\end{align*}
Using this last estimate along with the estimates from \S~\ref{sec:0DE}, \ref{sec:a_x}, it follows from \eqref{q^z:duhamel} and the fact that $\psi_x \sim 1$ that
$$ |q^z\cir\psi \psi_x - z'_0 | \lesssim \varepsilon^{\gamma_1} $$
for times $t \in [-\varepsilon, T]$. It follows that
$$ |q^z\cir\psi| \leq \|\psi^{-1}_x \|_{L^\infty(\mathbb{T} \times [-\varepsilon, T])} \big( \|z'_0\|_{L^\infty} + \mathcal O(\varepsilon^{\gamma_1}) \: \big) $$
for $t \in [-\varepsilon, T]$. Since $\gamma_1 > 0 \geq \beta_1$, it follows that if we let $\varepsilon$ become small enough we get
$$ |q^z| \leq 2\|\psi^{-1}_x \|_{L^\infty(\mathbb{T} \times [-\varepsilon, T])}  \|z'_0\|_{L^\infty} $$
for all $t \in [-\varepsilon, T]$. If $C$ is chosen large enough and $\varepsilon$ is chosen small enough, this improves upon our second bootstrap assumption.
\end{proof}

It follows as an immediate corollary of this proposition that
\begin{align}
\label{z_x:bound}
|\p_\theta z| & \lesssim \varepsilon^{\beta_1}, \\
\eta_x & \lesssim 1, \label{eta_x:bound} \\
\psi_x & \sim 1, \label{psi_x:bound} 
\end{align}
for all times $t \in [-\varepsilon, \varepsilon \wedge T_*]$.

\subsection{Bounding $|T_*|$}
\label{sec:T_*}

Now our estimates will let us conclude that $\eta_x$ behaves roughly the same as it would if $w$ were the solution of Burger's equation with initial data $w_0$ and $\eta$ were the flow of $w$. Using Proposition ~\ref{prop:small:time}, \eqref{I_t:bound}, \eqref{a_x:bound}, \eqref{z_x:bound}, and \eqref{eta_x:bound} in equation \eqref{q^w:equation} gives us
\begin{equation*}
\eta_x q^w\cir\eta = \big( - \tfrac{1}{\varepsilon} + (w'_0 + \tfrac{1}{\varepsilon}) \big)I_t e^{-\frac{1}{8}{k_0}} + \mathcal O(\varepsilon^{\gamma_1})
\end{equation*}
for all times $t \in [-\varepsilon, \varepsilon \wedge T_*]$. Plugging this into \eqref{eta_x:equation} and using the same bounds produces
\begin{equation}
\label{eq:eta_x}
\eta_x = 1 + \big( - \tfrac{1}{\varepsilon} + (w'_0 + \tfrac{1}{\varepsilon}) \big)\int^t_{-\varepsilon} \! \! I_\tau e^{-\frac{1}{8}{k_0}} \: d\tau + \mathcal O(\varepsilon^{\beta_1+1})
\end{equation}
for $t \in [-\varepsilon, \varepsilon \wedge T_*]$. Evaluating \eqref{eq:eta_x} at $x=0$ and using \eqref{I_t:bound} gives
\begin{align}
\eta_x(0,t) & = 1- (\varepsilon+t)\varepsilon^{-1} + \mathcal O( \varepsilon^{\mu}) \notag \\
& = - \tfrac{t}{\varepsilon} + \mathcal O( \varepsilon^{\mu}). \notag 
\end{align}
Since this is true for all $t \in [-\varepsilon, \varepsilon \wedge T_*]$, it follows that we must have $T_* \lesssim \varepsilon^{1+\mu}$ if $\varepsilon$ is chosen small enough. Therefore, $T_* = \varepsilon \wedge T_*$, and everything we have proven for $t \in [-\varepsilon, \varepsilon \wedge T_*]$ is true for $t \in [-\varepsilon, T_*]$.

We can also prove a lower bound on $T_*$. Since $w'_0(x) + \tfrac{1}{\varepsilon} \geq 0$ for all $x$, it follows from \eqref{eq:eta_x} and \eqref{I_t:bound} that
\begin{align*}
\eta_x & \geq - \tfrac{t}{\varepsilon} + \mathcal O(\varepsilon^{\mu})
\end{align*}
everywhere. Therefore, $|T_*| \lesssim \varepsilon^{1+\mu}$, else $\p_\theta w, \p_\theta z, \p_\theta k$ and $\p_\theta a$ would all stay bounded up to $T_*$.

We can also get a lower bound for $\eta_x$ away from 0. Indeed, since $w'_0(x) +\frac{1}{\varepsilon} \geq  C \varepsilon^{\frac{\mu}{2}-1}$ for $|x| \geq \varepsilon^{3/2}$,  we have
\begin{align}
\label{eta_x:lowerbound}
\eta_x & \geq -\tfrac{t}{\varepsilon} + C\varepsilon^{\frac{\mu}{2}-1} (\varepsilon + t) + \mathcal O(\varepsilon^\mu) \notag \\
& = (T_*-t) [ \tfrac{t}{\varepsilon} - C\varepsilon^{\frac{\mu}{2}-1}] + C\varepsilon^{\frac{\mu}{2}} + \mathcal O(\varepsilon^\mu) \notag \\
& \geq C\varepsilon^{\frac{\mu}{2}} + \mathcal O(\varepsilon^\mu) \notag \\
& \gtrsim \varepsilon^{\frac{\mu}{2}}.
\end{align}

Using Lemma ~\ref{lem:integralbound} and the estimates proven in this section, we can now conclude that the bound \eqref{bd:int} holds for all $\varphi > 0, (x,t) \in \mathbb{T} \times [-\varepsilon, T_*].$ This fact will be used so frequently in the rest of the paper that we will not bother to cite it.


\section{Transversality}
\label{sec:transfunc}

Let $\varphi > 0$ and let $\Psi$ be the flow of $\lambda : = (1-\varphi) w + (\tfrac{1}{3} + \varphi) z$.
\begin{align*}
-\tfrac{1}{\varphi} \Psi_x \p_t (\p_\theta c\cir\Psi) & = \tfrac{1}{\varphi} \Psi_{xt} \p_\theta c\cir\Psi - \tfrac{1}{\varphi}\p_{tx}( c\cir\Psi) \\
& = \tfrac{1}{\varphi} \Psi_{xt} \p_\theta c\cir\Psi + \Psi_x (\p_\theta w + (\tfrac{2}{3}\tfrac{1}{\varphi}-1)\p_\theta z + \tfrac{8}{3} \tfrac{1}{\varphi} a) \cir \Psi (\p_\theta c\cir\Psi) \\
& + \Psi_x(c\p_\theta^2 w + (\tfrac{2}{3}\tfrac{1}{\varphi}-1)c\p_\theta^2 z + \tfrac{8}{3} \tfrac{1}{\varphi} \p_\theta ac) \cir \Psi. \\
\implies - \tfrac{1}{\varphi} \p_t \big( \Psi_x ( c^{-1} \p_\theta c)\cir\Psi \big) & = - \tfrac{1}{\varphi} \Psi_{tx} (c^{-1} \p_\theta c)\cir \Psi - \Psi_x (c^{-2} \p_\theta c)\cir \Psi \big( -\tfrac{1}{\varphi} \p_t (c\cir\Psi) \big) \\
& + (c^{-1}\cir\Psi) \big( - \tfrac{1}{\varphi} \Psi_x \p_t( \p_\theta c\cir\Psi) \big) \\
& = \Psi_x(\p_\theta^2 w + (\tfrac{2}{3}\tfrac{1}{\varphi}-1)\p_\theta^2 z + \tfrac{8}{3} \tfrac{1}{\varphi} \p_\theta a) \cir \Psi.
\end{align*}
Therefore, if $h: \mathbb{T} \times [-\varepsilon, T_*) \rightarrow \R$ is any differentiable function, we have
\begin{align*}
\Psi_x(h\p^2_\theta w)\cir\Psi & = -\tfrac{1}{\varphi} \p_t\big( \Psi_x (c^{-1} \p_\theta c h) \cir \Psi\big) - \Psi_x( (\tfrac{2}{3}\tfrac{1}{\varphi}-1)h\p_\theta^2 z + \tfrac{8}{3} \tfrac{1}{\varphi} \p_\theta ah) \cir \Psi \notag\\
&\qquad - \Psi_x(c^{-1} \p_\theta c)\cir\Psi\big( -\tfrac{1}{\varphi} \p_t( h\cir\Psi) \big).
\end{align*}

This gives us the following equation:
\begin{align}
\label{trans:equation}
\p_x\big( (h\p_\theta w)\cir\Psi \big) & =  -\tfrac{1}{\varphi} \p_t\big( \Psi_x (c^{-1} \p_\theta c h) \cir \Psi\big) \notag \\
& - \Psi_x( (\tfrac{2}{3}\tfrac{1}{\varphi}-1)h\p_\theta^2 z + \tfrac{8}{3} \tfrac{1}{\varphi} \p_\theta ah) \cir \Psi \notag \\
& + \Psi_x \big[ ( \p_\theta h \p_\theta w) \cir \Psi -(c^{-1} \p_\theta c)\cir\Psi\big( -\tfrac{1}{\varphi} \p_t( h\cir\Psi) \big) \big].
\end{align}

The last term in this expression motivates the following definition: 

\begin{definition}[Transversality]
A differentiable function $h: \mathbb{T} \times [-\varepsilon, T_*) \rightarrow \R$ is {\bf{transversal}} (or {\bf{1-transversal}}) if it is bounded and there exists a constant $\varphi > 0$ and bounded functions $A,B, C$ such that
$$ \begin{cases} \p_\theta h = A c^{-1} \p_\theta c + B \\ \p_t h + \lambda  \p_\theta h = - \varphi A \p_\theta w -\varphi C \end{cases} $$
Here $\lambda = (1-\varphi) w + (\tfrac{1}{3} + \varphi) z$, as in the above discussion. If in addition $A,B,C$ are themselves transversal functions, we say that $h$ is {\bf{2-transversal}}. We recursively define $h$ to be {\bf{$n$-transversal}} if $A,B,$ and $C$ are $(n-1)$-transversal.
\end{definition}

A few remarks about transversal functions:
\begin{itemize}
\item If $h$ satisfies the transversality condition for one $\varphi_0 > 0$, then it satisfies the transversality condition for all $\varphi > 0$. If indeed, if we have
$$ \begin{cases} \p_\theta h = A c^{-1} \p_\theta c + B \\ \p_t h + \lambda_0  \p_\theta h = - \varphi_0 A \p_\theta w -\varphi_0 C \end{cases} $$
for some $\varphi_0 > 0$ then for any other $\varphi>0$ we have
$$ \begin{cases} \p_\theta h = A c^{-1} \p_\theta c + B \\ \p_t h + \lambda  \p_\theta h = - \varphi A \p_\theta w + (\varphi-\varphi_0) A \p_\theta z + 2(\varphi_0-\varphi) cB -\varphi_0 C  \end{cases} .$$
Since $\p_\theta z$ is bounded, $h$ still satisfies the transversality condition for $\varphi$, albeit with a different choice of bounded function $C$. So the notion of a transversal function is independent of our choice of $\varphi > 0$.
\item Note that while being transversal does not depend on the choice of $\varphi$ (as the previous bullet illustrated), and $A$ and $B$ are independent of $\varphi$, the function $C$ changes based on $\varphi$.
\item If $h$ is a bounded function with bounded derivatives, then $h$ is trivially transversal, with $A = 0, B = \p_\theta h$ and $C = -\tfrac{1}{\varphi} (\p_t h + \lambda \p_\theta h)$.
\item If functions $h_1, h_2$ are $n$-transversal, then $h_1+h_2$ is $n$-transversal. Indeed, we have
$$ \begin{cases} \p_\theta(h_1+h_2) = (A_1+A_2) c^{-1} \p_\theta c +B_1+B_2 \\ (\p_t + \lambda \p_\theta ) (h_1+h_2) = -\varphi (A_1+A_2) \p_\theta w -\varphi(C_1+C_2) \end{cases} . $$
\item If functions $h_1, h_2$ are $n$-transversal, then their product is $n$-transversal. Indeed, we have
$$ \begin{cases} \p_\theta( h_1h_2) = (A_1h_2+A_2h_1) c^{-1} \p_\theta c + B_1h_2 + B_2h_1 \\ (\p_t +\lambda  \p_\theta)( h_1h_2) = - \varphi (A_1h_2+A_2h_1) \p_\theta w -\varphi( C_1h_2+ C_2h_1) \end{cases}. $$
\item If $h$ is $n$-transversal and $h \sim 1$ then $h^{-1}$ is also $n$-transversal. Indeed,
$$ \begin{cases} \p_\theta (h^{-1}) = -h^{-2}A c^{-1} \p_\theta c - h^{-2}B \\ \p_t (h^{-1}) + \lambda  \p_\theta (h^{-1}) = - \varphi (-h^{-2}A) \p_\theta w -\varphi(- h^{-2}C) \end{cases} .$$
\item If $F: \R \rightarrow \R$ is smooth and $h$ is $n$-transversal, then $F \cir h$ is $n$-transversal. Indeed, we have
$$ \begin{cases} \p_\theta (F\cir h) = (A F'\cir h) c^{-1} \p_\theta c + B F'\cir h \\ (\p_t + \lambda \p_\theta)(F\cir h) = - \varphi (A F'\cir h) \p_\theta w - \varphi C F'\cir h \end{cases}. $$
This rule will be especially useful for $F(x) = e^x$.
\item $c$ is transversal with $A = c, B = 0,$ and $C = 4ac$ when $\varphi = \frac{2}{3}$. It follows inductively that if $a$ is $n$-transversal, then $c$ is $(n+1)$-transversal. At this point, we already know that $a$ is at least 1-transversal because it is uniformly $C^1$, so $c$ is currently proven to be at least 2-transversal. $c \sim 1$, so $c^{-1}$ is also 2-transversal. The fact that both $c$ and $c^{-1}$ are transversal was the main ingredient used in the computation of \eqref{trans:equation}.
\end{itemize}

The following lemma will be used in \S~\ref{sec:z_xxx}, \S~\ref{sec:z_xxxx}, and \S~\ref{sec:z_xxxxx}.

\begin{lemma}[Identities for transversal functions along 1-characteristics]
\label{lem:trans:1.1}
If $h: \mathbb{T} \times [-\varepsilon, T_*) \rightarrow \R$ is transversal with
$$ \begin{cases} \p_\theta h = A c^{-1} \p_\theta c + B \\ -\tfrac{3}{2}\p_t(h\cir\psi) = (A\p_\theta w + C)\cir\psi \end{cases} $$
then we have
\begin{equation}
\label{eq:trans:1.1.1}
\p_x\big( (h\p_\theta w)\cir\psi\big)  = -\tfrac{3}{2} \p_t \big(\psi_x ( c^{-1} \p_\theta c h)\cir\psi\big) + \psi_x\big( [B-\tfrac{1}{2}c^{-1} C] \p_\theta w\big) \cir\psi + \psi_x ( \tfrac{1}{2} c^{-1} C \p_\theta z- 4 \p_\theta a h)\cir\psi.
\end{equation}
and
\begin{align}
\label{eq:trans:1.1.2}
\p_x\big( \psi_x(h\p_\theta w)\cir\psi\big) &  = -\tfrac{3}{2} \p_t \big(\psi_x^2 ( c^{-1} \p_\theta c h)\cir\psi\big) + \psi^2_x\big( [Qh + B-\tfrac{1}{2}c^{-1} C+ \tfrac{3}{4} c^{-1} h \p_\theta z] \p_\theta w\big) \cir\psi  \notag \\
& + \psi^2_x ( \tfrac{1}{2} c^{-1} C \p_\theta z- 4 \p_\theta a h - \tfrac{3}{4} c^{-1} h \p_\theta z^2)\cir\psi.
\end{align}
From theses two equations we get the bounds
\begin{equation}
\label{bd:trans:1.1.1}
\big\vert \p_x\big( (h\p_\theta w)\cir\psi\big)  + \tfrac{3}{2} \p_t \big(\psi_x ( c^{-1} \p_\theta c h)\cir\psi\big) \big\vert  \lesssim \|h\|_{L^\infty} + \varepsilon^{\beta_1} \|C\|_{L^\infty} + ( \|B\|_{L^\infty} + \|C\|_{L^\infty}) |\p_\theta w\cir\psi|.
\end{equation}
and
\begin{align}
\label{bd:trans:1.1.2}
\big\vert \p_x\big( \psi_x(h\p_\theta w)\cir\psi\big)  + \tfrac{3}{2} \p_t \big(\psi_x^2 ( c^{-1} \p_\theta c h)\cir\psi\big) \big\vert  & \lesssim \varepsilon^{2\beta_1} \|h\|_{L^\infty} + \varepsilon^{\beta_1} \|C\|_{L^\infty}  \notag \\
& + ( \varepsilon^{-1} \|h\|_{L^\infty} + \|B\|_{L^\infty} + \|C\|_{L^\infty}) |\p_\theta w\cir\psi|.
\end{align}
\end{lemma}

\begin{proof}[Proof of Lemma~\ref{lem:trans:1.1}]
\eqref{eq:trans:1.1.1} follows immediately from \eqref{trans:equation}. To prove \eqref{eq:trans:1.1.2}, 
\begin{align*}
 \p_x(\psi(h\p_\theta w)\cir\psi) & = \psi_{xx} (h\p_\theta w)\cir\psi + \psi_x \p_x\big( (h\p_\theta w)\cir\psi\big)  \\
 & =  \psi^2_x ( [Q - \tfrac{1}{2} c^{-1} \p_\theta c] h \p_\theta w)\cir\psi - \p_t \big( \psi^2_x (c^{-1} \p_\theta c h)\cir\psi\big) + \tfrac{3}{2} \psi_{xt} \psi_x (c^{-1} \p_\theta c h)\cir\psi  \\
 & + \psi_x\big( [B-\tfrac{1}{2}c^{-1} C] \p_\theta w\big) \cir\psi + \psi_x ( \tfrac{1}{2} c^{-1} C \p_\theta z- 4 \p_\theta a h)\cir\psi \\
 & = -\tfrac{3}{2} \p_t \big(\psi_x^2 ( c^{-1} \p_\theta c h)\cir\psi\big) + \psi^2_x\big( [Qh + B-\tfrac{1}{2}c^{-1} C+ \tfrac{3}{4} c^{-1} h \p_\theta z] \p_\theta w\big) \cir\psi  \notag \\
& + \psi^2_x ( \tfrac{1}{2} c^{-1} C \p_\theta z- 4 \p_\theta a h - \tfrac{3}{4} c^{-1} h \p_\theta z^2)\cir\psi.
\end{align*}
The inequalities follow immediately from the equations and the first order estimates.
\end{proof}

The following lemma will be used in \S~\ref{sec:a_xxx}, \S ~ \ref{sec:a_xxxx}, and \S~\ref{sec:a_xxxxx}.

\begin{lemma}[Identities for transversal functions along 2-characteristics]
\label{lem:trans:2.1}
If $h: \mathbb{T} \times [-\varepsilon, T_*) \rightarrow \R$ is a differentiable function satisfying the transversality condition
$$ \begin{cases} \p_\theta h = A c^{-1} \p_\theta c + B \\ -3\p_t(h\cir\phi) = (A\p_\theta w + C)\cir\phi \end{cases} $$
then we have
\begin{equation}
\label{eq:trans:2.1.1}
\p_x \big( (h \p_\theta w) \cir\phi) = -3\p_t \big( \phi_x (c^{-1} \p_\theta c h)\cir\phi\big) + \phi_x (B \p_\theta w - C c^{-1} \p_\theta c - h\p^2_\theta z - 8 \p_\theta a h) \cir\phi.
\end{equation}
and
\begin{align}
\label{eq:trans:2.1.2}
\p_x\big(\phi^2_x(h\p_\theta w)\cir\phi\big) & = -3\p_t \big( \phi^3_x ( c^{-1} \p_\theta c \p_\theta w)\cir\phi \big) \notag \\
& + \phi^3_x(B\p_\theta w - Cc^{-1} \p_\theta c + 4 c^{-1} \p_\theta c h \p_\theta z - h \p^2_\theta z - 8 \p_\theta a h)\cir\phi \notag \\
& + 2\phi^2_x \Phi (h\p_\theta w)\cir\phi.
\end{align}
\end{lemma}

\begin{proof}[Proof of Lemma~\ref{lem:trans:2.1}]
\eqref{eq:trans:2.1.1} follows immediately from \eqref{trans:equation}. The proof of \eqref{eq:trans:2.1.2} is an easy computation using \eqref{trans:equation} and  \eqref{phi_xx:formula}.
\end{proof}

The following lemma will first be used in \S~\ref{sec:z_xxxx}, so there is no circularity in its proof. See \S~\ref{sec:k_xx} for the definition of $\mathcal E$ and \S~\ref{sec:z_xx} for the definition of $f$.

\begin{lemma}[Identities for 2-transversal functions along 1-characteristics]
\label{lem:trans:1.2}
If $h: \mathbb{T} \times [-\varepsilon, T_*) \rightarrow \R$ is 2-transversal with
\begin{align*}
\p_\theta h & = A c^{-1} \p_\theta c + B \\ 
-\tfrac{3}{2}\p_t(h\cir\psi) & = (A\p_\theta w + C)\cir\psi \\
\p_\theta A & = A_A c^{-1} \p_\theta c + B_A \\ 
-\tfrac{3}{2}\p_t(A\cir\psi) & = (A_A\p_\theta w + C_A)\cir\psi \\
\p_\theta B & = A_B c^{-1} \p_\theta c + B_B \\ 
-\tfrac{3}{2}\p_t(B\cir\psi) & = (A_B\p_\theta w + C_B)\cir\psi \\
\p_\theta C & = A_C c^{-1} \p_\theta c + B_C \\ 
-\tfrac{3}{2}\p_t(C\cir\psi) & = (A_C\p_\theta w + C_C)\cir\psi
\end{align*}
then we have
\begin{align}
\label{eq:trans:1.2.1}
\p^2_x\big( (h\p_\theta w)\cir\psi\big) & = - \tfrac{3}{2} \p_t\big( \psi^2_x( c^{-1} \p^2_\theta c h+  \big[-\tfrac{3}{2} h + A \big] c^{-2} \p_\theta c^2 + \big[ Q_1 + 2B - \tfrac{1}{2}c^{-1}C\big] c^{-1} \p_\theta c)\cir\psi \big) \notag \\
& + \psi^2_x\big( \big[ BQ_1 +B_B -\tfrac{1}{2} c^{-1} B_C \big] \p_\theta w \big) \cir\psi \notag \\
& - \psi^2_x \big( \big[ Q_1C + C_B -\tfrac{1}{2} c^{-1} C_C + 2  a c^{-1} C + \tfrac{1}{2} c^{-1} \p_\theta z C - 2 \p_\theta a h + 4 \p_\theta a A \big] c^{-1} \p_\theta c\big) \cir\psi \notag \\
& - \psi^2_x\big( \big[ \tfrac{1}{4} c^{-1} \p_\theta z C - \tfrac{1}{2} A_C c^{-1} \p_\theta z + \tfrac{1}{4} C \p_\theta k\big] c^{-1} \p_\theta c \big) \cir\psi \notag \\
& + \psi^2_x(\tfrac{1}{2} f - 4 \p_\theta a h Q_1 - 4 \p^2_\theta a h - 4 \p_\theta a B - \tfrac{1}{8} c \mathcal E + \tfrac{1}{2} B_C c^{-1} \p_\theta z)\cir\psi
\end{align}
and
\begin{align}
\label{eq:trans:1.2.2}
\p^2_x(h\cir\psi) & = -\tfrac{3}{4}\p_t \big( \psi^2_x(Ac^{-2} \p_\theta c)\cir\psi \big) \notag \\
& + \psi^2_x \big( [\tfrac{1}{2}Ac^{-1} Q_1 + \tfrac{1}{2} B_A c^{-1} - \tfrac{1}{4} C_A c^{-2} + Aac^{-2} + \tfrac{3}{8} Ac^{-2} \p_\theta z] \p_\theta w\big)\cir\psi \notag \\
& + \psi^2_x\big( [\tfrac{1}{4} Ac^{-1}\p_\theta z- \tfrac{1}{2} B+ A_B - \tfrac{1}{2} A_A c^{-1} \p_\theta + \tfrac{1}{4} Ac^{-1} \p_\theta z + \tfrac{1}{4} A\p_\theta k] c^{-1} \p_\theta c\big) \cir\psi \notag \\
& + \psi^2_x( \tfrac{1}{4} C_A c^{-2} \p_\theta z - A a c^{-2} \p_\theta z -2 \p_\theta a A c^{-1} - \tfrac{3}{8} A c^{-2} \p_\theta z^2) \cir\psi \notag \\
& + \psi^2_x( BQ_1 - \tfrac{1}{2} Ac^{-1} Q_1 \p_\theta z + B_B - \tfrac{1}{2} B_A c^{-1} \p_\theta z - \tfrac{1}{2} Ac^{-1} f + \tfrac{1}{8} A \mathcal E) \cir\psi.
\end{align}
\end{lemma}

\begin{proof}[Proof of Lemma~\ref{lem:trans:1.2}]
Taking $\p_x$ of \eqref{eq:trans:1.1.1} gives us
\begin{align*}
\p^2_x\big( (h\p_\theta w)\cir\psi\big) & = - \tfrac{3}{2} \p^2_{tx}\big( \psi_x(c^{-1} \p_\theta c)\cir\psi\big) + \p_x\big( \psi_x([B-\tfrac{1}{2}c^{-1}C]\p_\theta w)\cir\psi \big) \notag \\
& + \p_x\big( \psi_x (\tfrac{1}{2} c^{-1}C \p_\theta z - 4 \p_\theta a h)\cir\psi\big).
\end{align*}
If we define $\tilde{h} : = B-\tfrac{1}{2} c^{-1}C$, then the rules for transversal functions tell us that
\begin{align*}
B_{\tilde{h}} & = B_B- \tfrac{1}{2}c^{-1} B_C \\
C_{\tilde{h}} & = C_B - \tfrac{1}{2} c^{-1} C_C + 2 ac^{-1} C.
\end{align*}
Applying \eqref{eq:trans:1.1.2} to $\tilde{h}$ and simplifying gives us \eqref{eq:trans:1.2.1}.

For the next identity, we see that
\begin{align*}
\p^2_x(h\cir\psi) & = \p_x\big( \psi_x(\tfrac{1}{2} Ac^{-1} \p_\theta w)\cir\psi \big) + \p_x\big( \psi_x (B-\tfrac{1}{2} Ac^{-1} \p_\theta z) \cir\psi \big).
\end{align*}
Applying \eqref{eq:trans:1.1.2} to the function $\tfrac{1}{2} Ac^{-1}$ and simplifying gives us \eqref{eq:trans:1.2.2}.
\end{proof}

The following lemma will be used in \S~\ref{sec:4DE} and \S~\ref{sec:5DE}.

\begin{lemma}[Classes of transversal functions]
\label{lem:trans:3}
Let $\varphi> 0$ and let $\Psi$ be the flow of $\lambda = (1-\varphi)w + (\frac{1}{3} + \varphi) z$.
Then
\begin{enumerate}
\item If $h$ is a transversal function and $H$ is defined by 
$$ H\cir\Psi(x,t) : = \int^t_{-\varepsilon} \! \! h\cir\Psi (x,\tau) \: d\tau $$
then $H$ is transversal.
\item $\Psi_x \cir\Psi^{-1}$ is a transversal function.
\item If $h$ is a transversal function and $K$ is defined by
$$ K\cir\Psi(x,t) : = \int^t_{-\varepsilon} (h\p_\theta w)\cir\Psi(x,\tau) \: d\tau $$
then $K$ is transversal.
\item If $h$ is a 2-transversal function and $H$ is defined by 
$$ H\cir\Psi(x,t) : = \int^t_{-\varepsilon} \! \! h\cir\Psi (x,\tau) \: d\tau $$
then $H$ is 2-transversal.
\item $\Psi_x \cir \Psi^{-1}$ is a 2-transversal function.
\item  If $h$ is a 2-transversal function and $K$ is defined by
$$ K\cir\Psi(x,t) : = \int^t_{-\varepsilon} (h\p_\theta w)\cir\Psi(x,\tau) \: d\tau $$
then $K$ is 2-transversal.
\end{enumerate}

\end{lemma}

\begin{proof}[Proof of Lemma~\ref{lem:trans:3}]
In this proof, $h$ satisfies
$$ \begin{cases} \p_\theta h = A c^{-1} \p_\theta c + B \\ -\tfrac{1}{\varphi} \p_t(h\cir\Psi) = (A \p_\theta w + C) \cir\Psi \end{cases}. $$

{\it{(i)}} Since
\begin{equation}
\label{eq:trans:3.1}
 \begin{cases} \p_\theta H = \bigg(\Psi_x^{-1} \int^t_{-\varepsilon}  (A c^{-1} \p_\theta c + B) \cir\Psi \: d\tau\bigg) \cir \Psi^{-1} \\ -\tfrac{1}{\varphi} \p_t(H\cir\Psi) = -\tfrac{1}{\varphi} h\cir\Psi \end{cases} 
\end{equation}
it follows from \eqref{bd:int} and the fact that $\Psi_x \sim 1$ that $H$ is transversal.

{\it{(ii)}} We know that $a$ is transversal, and it will be proven in \S~\ref{sec:z_xx} that $\p_\theta z$ is transversal. Therefore, part (i) applies to the function
$$ H\cir\Psi = \int^t_{-\varepsilon} \! \! (2-\frac{2}{3}\frac{1}{\varphi})) \p_\theta z - \tfrac{8}{3} \frac{1-\varphi}{\varphi}  a) \cir\Psi. $$
Since $F(x) = e^x$ is smooth, it follows that
$$ e^{\int^t_{-\varepsilon} (2-\frac{2}{3}\frac{1}{\varphi})) \p_\theta z - \tfrac{8}{3} \frac{1-\varphi}{\varphi}  a) \circ\Psi} \cir\Psi^{-1} $$
is transversal. We already know that $c$ and $c^{-1}$ are both 2-transversal, so it now follows from \eqref{eq:flow_x} that $\Psi_x\cir\Psi^{-1}$ is transversal.

{\it{(iii)}} Using \eqref{trans:equation} tells us that
\begin{align}
\label{eq:trans:3.2}
\p_\theta K\cir\psi & = - \tfrac{1}{\varphi} (h c^{-1} \p_\theta c)\cir\Psi + \frac{1}{\varphi} \Psi_x^{-1} c^{-1}_0 c'_0 h_0  \notag \\
& - \Psi^{-1}_x \int^t_{-\varepsilon} \! \!  \Psi_x ( (\tfrac{2}{3}\tfrac{1}{\varphi}-1) h \p^2_\theta z + \tfrac{8}{3} \tfrac{1}{\varphi} \p_\theta a h) \cir \Psi \: d\tau + \Psi^{-1}_x \int^t_{-\varepsilon} \! \! \Psi_x (B \p_\theta w - c^{-1} \p_\theta c C) \cir\Psi \: d\tau, \notag \\
- \tfrac{1}{\varphi} \p_t(K\cir\Psi) & = - \tfrac{1}{\varphi} (h \p_\theta w)\cir\Psi.
\end{align}
It now follows from \eqref{z_xx:bound} that $K$ is transversal.

{\it{(iv)}} This follows immediately from applying {\it{(ii)}} and {\it{(iii)}} to \eqref{eq:trans:3.1}.

{\it{(v)}} It will be proven in \S~\ref{sec:a_xx} that $\p_\theta a$ is transversal, from which it will follow that $a$ is 2-transversal, and it will be proven in \S~\ref{sec:z_xxx} that $\p_\theta z$ is 2-transversal. Since $F(x) = e^x$ is smooth, it follows from {\it{(iv)}} that
$$ e^{\int^t_{-\varepsilon} (2-\frac{2}{3}\frac{1}{\varphi})) \p_\theta z - \tfrac{8}{3} \frac{1-\varphi}{\varphi}  a) \circ\Psi} \cir\Psi^{-1} $$
is 2-transversal. Since $c^{-1}$ is 2-transversal it now follows from \eqref{eq:flow_x} that $\Psi_x\cir\Psi^{-1}$ is 2-transversal.

{\it{(vi)}} We prove in \S~\ref{sec:a_xx}, \ref{sec:z_xx}, \ref{sec:k_xxx}, \ref{sec:z_xxx} that $\p_\theta a, \p_\theta z, \mathcal E,$ and $f$ are all transversal, so our result follows from applying {\it{(i)}}, {\it{(ii)}}, and {\it{(iii)}} to \eqref{eq:trans:3.2}.

\end{proof}



\section{Second Derivative Estimates}
\label{sec:2DE}

\subsection{$\p_\theta^2 k$ bounds}
\label{sec:k_xx}

Differentiating \eqref{k_x:formula} and plugging in \eqref{phi_xx:formula} gives us
\begin{align}
\label{k_xx:formula}
\phi^2_x\p_\theta^2 k \cir \phi  & = k''_0 - \phi_{xx} \p_\theta k \cir \phi \notag \\
& = k''_0 - \phi_x \Phi \p_\theta k \cir \phi + 2\phi^2_x \frac{\p_\theta c \cir \phi \p_\theta k \cir \phi}{c\cir \phi} \notag \\
& = k''_0 - \Phi k'_0 + 2\phi^2_x \frac{\p_\theta c \cir \phi \p_\theta k \cir \phi}{c\cir \phi}.
\end{align}

If we define
\begin{equation}
\label{E:definition}
\mathcal E : = \p_\theta^2k- 2c^{-1} \p_\theta c \p_\theta k, 
\end{equation}
then it follows that $\mathcal E = [ \phi^{-2}_x (k''_0 - \Phi k'_0) ] \cir\phi^{-1}$ and from \eqref{Phi:bound} we conclude that
\begin{align}
\label{E:bound}
|\mathcal E| & \lesssim \varepsilon^{\gamma_2\wedge \gamma_1-1}. \\
\implies |\p^2_\theta k| & \lesssim  \varepsilon^{\gamma_2 \wedge \gamma_1-1} + \varepsilon^{\gamma_1} |\p_\theta w|. \label{k_xx:bound}
\end{align} 
With this notation, we can write
\begin{align}
\label{k_x:trans}
\p^2_\theta k  & = (2\p_\theta k) c^{-1} \p_\theta c + \mathcal E\notag \\
-\tfrac{3}{2} \p_t(\p_\theta k\cir\psi) & = (2\p_\theta k) \p_\theta w + c\mathcal E,
\end{align}
so  $\p_\theta k$ is transversal. The fact that $\p_\theta k$ is transversal will be used through \S~\ref{sec:3DE} and \S~\ref{sec:2DE}.

\subsection{$\p_\theta^2 a$ bounds}
\label{sec:a_xx}

Using \eqref{phi_x:formula} and \eqref{c_t:identity} we have
\begin{align}
\label{a:bound:identity1}
\I_t \p_x(c^2\cir\phi) & = 2\phi_x \I_t \big[ -\frac{3}{2} \p_t(c\cir\phi) - (4ac + c\p_\theta z)\cir\phi \big] \notag \\
& = c^2_0 \I^{-1}_t \big[ 3 \p_t (c^{-1}\cir \phi) - (8 c^{-1} a + 2c^{-1} \p_\theta z)\cir \phi \big] \notag \\
& = c^2_0 \big[ 3 \p_t\big( \I^{-1}_t c^{-1}\cir \phi \big) - 2 \I^{-1}_t (c^{-1} \p_\theta z)\circ \phi \big].
\end{align}
Therefore, differentiating \eqref{vorticity:duhamel} gives us
\begin{align}
\label{vorticity:duhamel:p_x}
\p_x(\varpi\cir\phi) & = \p_x\big( \varpi_0 \I_t \big) - 4 \frac{k'_0 e^{k_0}}{c_0} \I_t  + 3\frac{k'_0 e^{k_0}}{c\cir \phi} - \tfrac{8}{3}k'_0e^{k_0} \I_t \int^t_{-\varepsilon} \! \! \I^{-1}_\tau (c^{-1} \p_\theta z)\cir \phi \: d\tau \notag \\
& + \tfrac{4}{3}\p_x\big( \frac{k'_0 e^{k_0}}{c^2_0} \I_t \big) \int^t_{-\varepsilon} \! \! \I_\tau (c^2\cir \phi) \: d\tau + \tfrac{4}{3}\frac{k'_0 e^{k_0}}{c^2_0} \I_t \int^t_{-\varepsilon} \! \! \p_x\big( \I_\tau \big) (c^2\cir \phi) \: d\tau.
\end{align}

It is easy to check that
\begin{equation}
\label{if_x:a:bound}
|\I_t'| \lesssim \varepsilon.
\end{equation}
It follows from this bound and \eqref{vorticity:duhamel:p_x} that
\begin{align}
 |\p_\theta \varpi\cir\phi| & \lesssim |\varpi'_0| + \varepsilon^{\gamma_2+1 \wedge \gamma_1}. \notag
\end{align}
By differentiating the equation \eqref{vorticity_0:formula} and using our assumptions on the initial data, we conclude that $|\varpi'_0| \lesssim \varepsilon^{-1}$. Therefore,
\begin{equation}
\label{vorticity_x:bound}
|\p_\theta \varpi| \lesssim \varepsilon^{-1}.
\end{equation}
Differentiating \eqref{a_x:formula} in space and using our first derivative estimates along with \eqref{vorticity_x:bound} gives us
\begin{align}
\label{a_x:trans}
\p^2_\theta a & = 2[\p_\theta a -c -2z] c^{-1} \p_\theta c + 2\p_\theta z-\tfrac{1}{4} c^2 \p_\theta \varpi e^{-k} + \tfrac{1}{4} c^2 \p_\theta k \varpi e^{-k} \notag \\ 
-\tfrac{3}{2} \p_t (\p_\theta a \cir\psi) & =  \big( 2[\p_\theta a -c -2z] \p_\theta w + 4a \p_\theta a - \tfrac{1}{4} c^3 \p_\theta \varpi e^{-k}+ \tfrac{1}{4} c^3 \p_\theta k \varpi e^{-k}\big) \cir\psi.
\end{align}
So \eqref{vorticity_x:bound} lets us conclude that $\p_\theta a$ is transversal, which will be used in \S~\ref{sec:3DE} and \S~\ref{sec:4DE}. This equation for $\p^2_\theta a$ and our estimate \eqref{vorticity_x:bound} also lets us conclude that
\begin{align}
\label{a_xx:bound}
|\p^2_\theta a| & \lesssim \varepsilon^{-1} + |\p_\theta w|.
\end{align}

It now follows from \eqref{a_xx:bound} that 
\begin{equation}
\label{Phi_x:bound}
|\p_x\Phi| \lesssim \varepsilon^{-2} + |w''_0|.
\end{equation}

\subsection{$\p_\theta^2 z$ bounds}
\label{sec:z_xx}

Let's introduce the new variable
\begin{equation}
\label{f:def}
f : = \partial^2_\theta z - \tfrac{1}{2} c^{-1} \p_\theta c \p_\theta z + \tfrac{1}{4}c \p^2_\theta k.
\end{equation} 
Using the identities from \S~\ref{sec:id} along with \eqref{E:definition} gives us
\begin{align}
\label{f_t:formula}
\p_t(f\cir\psi) & : = \big( \p_t(\log c\cir\psi) + \tfrac{1}{2} \p_t(k\cir\psi) - 3\p_\theta z\cir\psi \big) (f\cir\psi) \notag \\
& + ( \tfrac{1}{12} c^2 \p_\theta k \p^2_\theta k - \tfrac{1}{2} c^{-1} \p_\theta c \p_\theta z^2) \cir\psi - \tfrac{1}{8} \mathcal E ( \tfrac{1}{3} c\p_\theta w - c\p_\theta z)\cir\psi \notag \\
& - \tfrac{8}{3} ( \p^2_\theta a z + \tfrac{3}{2} \p_\theta a \p_\theta z - \tfrac{1}{2} \p_\theta a c^{-1} \p_\theta c z) \cir\psi.
\end{align}
If we define
\begin{equation}
J_t : = e^{\frac{1}{2}k\circ\psi  - 3 \int^t_{-\varepsilon} \p_\theta z\circ\psi},
\end{equation}
then \eqref{f_t:formula} gives us the Duhamel formula
\begin{align}
\label{f:duhamel}
f\cir\psi & = c^{-1}_0 (z''_0 - \tfrac{1}{2}c^{-1}_0 c'_0 z'_0 + \tfrac{1}{2} c_0 k''_0)e^{-\frac{1}{2}k_0} J_t c\cir\psi + \tfrac{1}{12} c\cir \psi J_t \int^t_{-\varepsilon} \! \! J^{-1}_\tau ( c \p_\theta k \p^2_\theta k) \cir \psi \: d\tau \notag \\
& - \tfrac{1}{2}c\cir \psi J_t \int^t_{-\varepsilon} \! \! J^{-1}_\tau (c^{-2} \p_\theta c \p_\theta z^2) \cir \psi \: d\tau - \tfrac{1}{8} c\cir \psi J_t \int^t_{-\varepsilon} \! \! J^{-1}_\tau \mathcal E ( \tfrac{1}{3} \p_\theta w - \p_\theta z) \cir\psi \: d\tau \notag \\
& - \tfrac{8}{3} c\cir \psi J_t \int^t_{-\varepsilon} \! \! J^{-1}_\tau ( \p^2_\theta a c^{-1} z + \tfrac{3}{2} \p_\theta a c^{-1} \p_\theta z - \tfrac{1}{2} \p_\theta a c^{-2} \p_\theta c z) \cir \psi \: d\tau.
\end{align} 
It now follows from \eqref{k_xx:bound}, \eqref{E:bound}, \eqref{a_xx:bound} that
\begin{equation}
\label{f:bound}
|f|  \lesssim \varepsilon^{\beta_2 \wedge \gamma_2 \wedge \beta_1-1}.
\end{equation}

It follows immediately from this bound and \eqref{k_xx:bound} that
\begin{equation}
\label{z_xx:bound}
|\p^2_\theta z | \lesssim \varepsilon^{\beta_2 \wedge \gamma_2 \wedge \beta_1-1} + \varepsilon^{\beta_1} |\p_\theta w|.
\end{equation}
This bound tells us that
\begin{equation}
\label{Psi:bound}
|\Psi| \lesssim \varepsilon^{-1}.
\end{equation}

We can also conclude that $\p_\theta z$ is transversal. Indeed
\begin{align}
\label{z_x:trans}
\p^2_\theta z & = \big[ \tfrac{1}{2} \p_\theta z - \tfrac{1}{2} c\p_\theta k \big] c^{-1} \p_\theta c + f-\tfrac{1}{4} c\mathcal E \notag \\
-\tfrac{3}{2}\p_t(\p_\theta z\cir\psi) & = \big( \big[ \tfrac{1}{2} \p_\theta z - \tfrac{1}{2} c\p_\theta k \big]\p_\theta w + \tfrac{1}{2} c\p_\theta k \p_\theta z - \tfrac{1}{4} c^2 \mathcal E + 4\p_\theta a z + a \p_\theta z \big) \cir\psi.
\end{align}
The fact that $\p_\theta z$ is transversal will be used in \S~\ref{sec:3DE} and \S~\ref{sec:4DE}.


\section{Third derivative estimates}
\label{sec:3DE}

\subsection{$\p_\theta^3k$ bounds}
\label{sec:k_xxx}

Using the fact that $\mathcal E = [ \phi^{-2}_x (k''_0 - \Phi k'_0) ] \cir\phi^{-1}$, we can compute that
\begin{equation}
\p_\theta \mathcal E = [ \phi^{-3}_x(k_0'''-3\Phi k''_0 + (2\Phi^2-\p_x \Phi)k'_0) ] \cir \phi^{-1}  + 4 \mathcal E c^{-1}\p_\theta c
\end{equation}
Define
\begin{equation}
\tilde{\mathcal E} : = [ \phi^{-3}_x(k_0'''-3\Phi k''_0 + (2\Phi^2-\p_x \Phi)k'_0) ] \cir \phi^{-1}.
\end{equation}
We know from \eqref{Phi:bound} and \eqref{Phi_x:bound} that
\begin{equation}
\label{E:bound:2}
|\tilde{\mathcal E}| \lesssim  \varepsilon^{\gamma_1} |w''_0\cir\phi^{-1}| + \varepsilon^{\gamma_3 \wedge \gamma_2 -1 \wedge \gamma_1-2}.
\end{equation}
Since
\begin{align}
\p_\theta  \mathcal E & =   4 \mathcal E c^{-1} \p_\theta c + \tilde{\mathcal E} \notag \\
-\tfrac{3}{2} \p_t (\mathcal E\cir\psi) & = (4 \mathcal E \p_\theta w -8 \p_\theta a \p_\theta k + \tilde{\mathcal E} c)\cir\psi.
\end{align}
it follows that $\mathcal E$ is transversal and therefore $\p_\theta k$ is 2-transversal. 

Taking $\p_\theta$ of \eqref{E:definition} gives us
\begin{align}
\label{k_xxx:formula}
\p^3_\theta k & = \tilde{\mathcal E}  + 6 \mathcal E c^{-1} \p_\theta c + 2 c^{-2} \p_\theta c^2 \p_\theta k + 2c^{-1}\p^2_\theta c\p_\theta k .
\end{align}
It follows that
\begin{align}
\label{k_xxx:bound}
|\p^3_\theta k| & \lesssim \varepsilon^{\gamma_3\wedge \gamma_2-1\wedge \gamma_1-2} + \varepsilon^{\gamma_2 \wedge \gamma_1-1} |\p_\theta w| \notag \\
& + \varepsilon^{\gamma_1} |w''_0\cir\phi^{-1}| + \varepsilon^{\gamma_1} |\p_\theta w|^2 + \varepsilon^{\gamma_1} |\p^2_\theta w|.
\end{align}

\subsection{ $\p_\theta^3 a$ bounds} 
\label{sec:a_xxx}

Taking $\p_x$ of \eqref{vorticity:duhamel:p_x} gives us
\begin{align*}
\p^2_x(\varpi\cir\phi) & =  \p^2_x\big( \varpi_0 \I_t \big) -4 \p_x\big( \tfrac{k'_0 e^{k_0}}{c_0} \I_t \big) + 4\frac{\p_x(k'_0 e^{k_0})}{c\cir\phi} - 4 k'_0 e^{k_0} \phi_x \frac{\p_\theta c\cir\phi}{(c\cir\phi)^2}  \\
&  - \tfrac{8}{3}\p_x(k'_0 e^{k_0} \I_t) \int^t_{-\varepsilon} \! \! \I^{-1}_\tau (c^{-1} \p_\theta z)\cir \phi \: d\tau - \tfrac{8}{3}k'_0 e^{k_0} \I_t \int^t_{-\varepsilon} \! \! \p_x\big( \I^{-1}_\tau (c^{-1} \p_\theta z) \cir \phi \big) \: d\tau \\
& + \tfrac{4}{3} \p^2_x\big( \frac{k'_0 e^{k_0}}{c^2_0} \I_t \big) \int^t_{-\varepsilon} \! \! \I_\tau (c^2\cir\phi) \: d\tau + \tfrac{8}{3}\p_x\big( \frac{k'_0 e^{k_0}}{c^2_0} \I_t \big) \int^t_{-\varepsilon} \! \! \p_x\big( \I_\tau\big) (c^2\cir\phi) \: d\tau \\
& + \tfrac{4}{3}\p_x\big( \frac{k'_0e^{k_0}}{c^2_0} \I_t \big) \int^t_{-\varepsilon} \! \! \I_\tau \p_x(c^2\cir\phi) \: d\tau + \tfrac{4}{3}\frac{k'_0 e^{k_0}}{c^2_0} \I_t \int^t_{-\varepsilon} \! \! \p_x(\I_\tau) \p_x(c^2\cir\phi) \: d\tau \\
& + \tfrac{4}{3}\frac{k'_0 e^{k_0}}{c^2_0} \I_t \int^t_{-\varepsilon} \! \! \p^2_x(\I_\tau) (c^2\cir\phi) \: d\tau.
\end{align*}
It is easy to use \eqref{phi_xx:formula} and \eqref{a_xx:bound} to get
$|\I_t''|, |\p^2_x \I_{t,\tau}| \lesssim 1.$
Using \eqref{a:bound:identity1} we have
\begin{align}
\label{a:bound:identity2}
\p_x(\I_t) \p_x(c^2\cir\phi) & = \frac{8}{3}\bigg( \int^t_{-\varepsilon} \! \! \p_x(a \cir\phi) \bigg) \I_t \p_x(c^2\cir\phi) \\
& = c^2_0 \bigg[ 8 \p_t\bigg( \bigg( \int^t_{-\varepsilon} \! \! \p_x(a\cir\phi) \bigg) \I^{-1}_t c^{-1}\cir\phi \bigg) - 8 \p_x(a\cir\phi) \I^{-1}_t c^{-1}\cir\phi + 2 \p_x(\I_t^{-1}) (c^{-1}\p_\theta z)\cir\phi \bigg].
\end{align}
Using \eqref{a:bound:identity1} and \eqref{a:bound:identity2} we have
\begin{align*}
\p_x\big( \frac{k'_0e^{k_0}}{c^2_0} \I_t \big) \int^t_{-\varepsilon} \! \! \I_\tau \p_x(c^2\cir\phi) \: d\tau & = -3 \frac{\p_x(k'_0 e^{k_0} \I_t)}{c_0} + 6 \frac{c'_0 k'_0 e^{k_0}}{c^2_0} \I_t \\
& + \bigg[ 3\p_x(k'_0 e^{k_0}) +8 k'_0 e^{k_0} \int^t_{-\varepsilon} \! \! \p_x(a\cir\phi) \: d\tau  -6 \tfrac{c'_0 k'_0 e^{k_0}}{c_0} \bigg] c^{-1}\cir\phi \\
&  - 2\p_x(k'_0 e^{k_0} \I_t) \int^t_{-\varepsilon} \! \! \I^{-1}_\tau (c^{-1} \p_\theta z)\cir \phi \: d\tau  + 4\frac{c'_0 k'_0 e^{k_0}}{c_0} \I_t \int^t_{-\varepsilon} \! \! \I^{-1}_\tau (c^{-1} \p_\theta z)\cir \phi \: d\tau. \\
 \frac{k'_0 e^{k_0}}{c^2_0} \I_t \int^t_{-\varepsilon} \! \! \p_x(\I_\tau) \p_x(c^2\cir\phi) \: d\tau & = 8 k'_0 e^{k_0} \bigg( \int^t_{-\varepsilon} \! \! \p_x(a\cir\phi) \bigg) c^{-1}\cir \phi + 2k'_0 e^{k_0} \I_t \int^t_{-\varepsilon} \! \! \p_x(\I^{-1}_t) (c^{-1} \p_\theta z)\cir\phi \: d\tau \\
 & - 8 c^2_0 k'_0 e^{k_0} \I_t \int^t_{-\varepsilon} \! \! \I^{-3}_\tau (\p_\theta a c^{-3})\cir\phi \: d\tau.
 \end{align*}
Therefore, we have
\begin{align}
\label{vorticity:duhamel:p_xx}
\p^2_x(\varpi\cir\phi) & = \p^2_x\big(\varpi_0 \I_t\big) - 8 \frac{\p_x( k'_0 e^{k_0}\I_t)}{c_0} + 12 \frac{c'_0 k'_0 e^{k_0}}{c^2_0} \I_t  \notag \\
& +\bigg[ 8\p_x(k'_0 e^{k_0}) +\tfrac{64}{3} k'_0 e^{k_0} \int^t_{-\varepsilon} \! \! \p_x(a\cir\phi) \: d\tau  -8 \tfrac{c'_0 k'_0 e^{k_0}}{c_0} \bigg] c^{-1}\cir\phi  - 4k'_0 e^{k_0} \phi_x (c^{-2} \p_\theta c)\cir\phi \notag \\
& - \tfrac{16}{3}\p_x(k'_0 e^{k_0} \I_t) \int^t_{-\varepsilon} \! \! \I^{-1}_\tau (c^{-1} \p_\theta z)\cir \phi \: d\tau  + \tfrac{16}{3}\frac{c'_0 k'_0 e^{k_0}}{c_0} \I_t \int^t_{-\varepsilon} \! \! \I^{-1}_\tau (c^{-1} \p_\theta z)\cir \phi \: d\tau \notag \\
& + \tfrac{4}{3}\p^2_x\big( \frac{k'_0 e^{k_0}}{c^2_0} \I_t \big) \int^t_{-\varepsilon} \! \! \I_\tau (c^2\cir\phi) \: d\tau + \tfrac{8}{3}\p_x\big( \frac{k'_0 e^{k_0}}{c^2_0} \I_t \big) \int^t_{-\varepsilon} \! \! \p_x\big( \I_\tau\big) (c^2\cir\phi) \: d\tau \notag \\
& + \tfrac{4}{3}\frac{k'_0 e^{k_0}}{c^2_0} \I_t \int^t_{-\varepsilon} \! \! \p^2_x(\I_\tau) (c^2\cir\phi) \: d\tau  \notag \\
& +\tfrac{2}{3} c^2_0 k'_0 e^{k_0} \I_t \int^t_{-\varepsilon} \! \!  \I^{-3}_\tau \big(\big[ c^{-4} \p_\theta z + c^{-3} \p_\theta k \big] \p_\theta w\big) \cir\phi \: d\tau - \tfrac{2}{3} c^2_0 k'_0 e^{k_0} \I_t \int^t_{-\varepsilon} \! \! \I^{-3}_\tau (c^{-4} \p_\theta z^2 + c^{-3} \p_\theta k \p_\theta z) \cir\phi \: d\tau  \notag \\
&  + c^2_0 k'_0 e^{k_0} \I_t \int^t_{-\varepsilon} \! \! \I^{-3}_\tau ( \tfrac{2}{3} c^{-2}\mathcal E - \tfrac{8}{3} c^{-3} f - \tfrac{32}{3} \p_\theta a c^{-3}) \cir\phi \: d\tau .
\end{align}
Rearranging this and using $k\cir \phi = k_0$, \eqref{k_x:formula}, \eqref{phi_x:formula:2}, and \eqref{phi_xx:formula:2}, we have
\begin{align}
\label{vorticity:duhamel:p_xx:2}
\p^2_\theta \varpi  & = A c^{-1} \p_\theta c + B
\end{align}
where
\begin{align}
A & : = 2 \p_\theta \varpi - 4 c^{-1}e^{k} \p_\theta k, \label{def:A} \\
B\cir\phi & : = -(Q_2c)\cir\phi \notag \\
& + \big[c^{-4}_0 \mathcal I^4_t \p^2_x\big(\varpi_0 \I_t\big) - 8  c^{-5}_0 \mathcal I^4_t\p_x( k'_0 e^{k_0}\I_t)  + 12 c_0^{-6} c'_0 k'_0 e^{k_0} \I_t^5\big] c^4\cir\phi \notag \\
& +\bigg[ 8\p_x(k'_0 e^{k_0}) +\tfrac{64}{3} k'_0 e^{k_0} \int^t_{-\varepsilon} \! \! \p_x(a\cir\phi) \: d\tau  -8 \tfrac{c'_0 k'_0 e^{k_0}}{c_0} \bigg]c_0^{-4} \mathcal I^4_t c^{3}\cir\phi  \notag \\
& - \tfrac{16}{3}c_0^{-4} \mathcal I_t^4\p_x(k'_0 e^{k_0} \I_t)  \int^t_{-\varepsilon} \! \! \I^{-1}_\tau (c^{-1} \p_\theta z)\cir \phi \: d\tau c^4\cir\phi  + \tfrac{16}{3}c_0^{-5}c'_0 k'_0 e^{k_0} \I_t^5  \int^t_{-\varepsilon} \! \! \I^{-1}_\tau (c^{-1} \p_\theta z)\cir \phi \: d\tau c^4\cir\phi \notag \\
& + \tfrac{4}{3}c_0^{-4} \mathcal I_t^4\p^2_x\big( \frac{k'_0 e^{k_0}}{c^2_0} \I_t \big) \int^t_{-\varepsilon} \! \! \I_\tau (c^2\cir\phi) \: d\tau c^4\cir\phi + \tfrac{8}{3}c_0^{-4} \mathcal I_t^4\p_x\big( \frac{k'_0 e^{k_0}}{c^2_0} \I_t \big) \int^t_{-\varepsilon} \! \! \p_x\big( \I_\tau\big) (c^2\cir\phi) \: d\tau c^4\cir\phi\notag \\
& + \tfrac{4}{3}c_0^{-6}k'_0 e^{k_0}\I_t^5 \int^t_{-\varepsilon} \! \! \p^2_x(\I_\tau) (c^2\cir\phi) \: d\tau c^4\cir\phi  \notag \\
& +\tfrac{2}{3} c^{-2}_0 k'_0 e^{k_0} \I_t^5 \bigg[ \int^t_{-\varepsilon} \! \!  \I^{-3}_\tau \big(\big[ c^{-4} \p_\theta z + c^{-3} \p_\theta k \big] \p_\theta w\big) \cir\phi \: d\tau -  \int^t_{-\varepsilon} \! \! \I^{-3}_\tau (c^{-4} \p_\theta z^2 + c^{-3} \p_\theta k \p_\theta z) \cir\phi \: d\tau \bigg] c^4\cir\phi  \notag \\
&  + c^{-2}_0 k'_0 e^{k_0} \I_t^5 \int^t_{-\varepsilon} \! \! \I^{-3}_\tau ( \tfrac{2}{3} c^{-2}\mathcal E - \tfrac{8}{3} c^{-3} f - \tfrac{32}{3} \p_\theta a c^{-3}) \cir\phi \: d\tau c^4\cir\phi . \label{def:B}
\end{align}
Taking two derivatives of \eqref{vorticity_0:formula} we get the bound
\begin{equation}
\label{vorticity_0xx:bound}
|\varpi''_0| \lesssim \varepsilon^{-2} + |w''_0|.
\end{equation}
It follows that
\begin{align}
\label{B:bound}
|B| & \lesssim \varepsilon^{-2}+ |w''_0\cir\phi^{-1}|  . 
\end{align}
We therefore conclude that
\begin{align}
\label{vorticity_xx:bound}
|\p^2_\theta \varpi| & \lesssim \varepsilon^{-2}+ |w''_0\cir\phi^{-1}| + \varepsilon^{-1}|\p_\theta w|.
\end{align}

Differentiating \eqref{vorticity_t:identity} and using \eqref{k_x:formula} allows us to compute that
\begin{align}
-3\p_t(\p_\theta \varpi \cir \phi) & = ( A \p_\theta w + C)\cir\phi  
\end{align}
where
\begin{align}
\label{def:C}
C & : =  4 e^k c^{-1} \p_\theta k \p_\theta z + 2 \p_\theta \varpi  \p_\theta z\notag \\
& - 8 a \p_\theta \varpi + 8 \p_\theta a \varpi + 4 e^k \mathcal E + 4 e^k \p_\theta k^2.
\end{align}
This, along with \eqref{vorticity_xx:bound}, implies that $\p_\theta \varpi$ is transversal. Therefore (see \eqref{a_x:trans}) $\p_\theta a$ is 2-transversal.

Now taking $\p_\theta$ or \eqref{a_x:trans} and using \eqref{k_xx:bound}, \eqref{z_xx:bound}, and \eqref{vorticity_xx:bound} gives us
\begin{align}
\label{a_xxx:bound}
\p^3_\theta a & = 2[\p_\theta a -c -2z] c^{-1} \p^2_\theta c + \mathcal O \big( \varepsilon^{-2} + |w''_0\cir\phi^{-1}| + \varepsilon^{-1} |\p_\theta w| + |\p_\theta w|^2 \big).
\end{align}

Now, one can compute that
\begin{align}
\label{a_xx:phidentity}
\p^2_x(a\cir\phi) & = \phi^2_x ( -[1+ 2c^{-1}z] \p_\theta w + [3+2c^{-1}z] \p_\theta z)\cir\phi \notag \\
& + \phi^2_x ( \tfrac{1}{4} c^2 \p_\theta k \varpi e^{-k} - \tfrac{1}{4} c^2 \p_\theta \varpi e^{-k} + \p_\theta a Q_2)\cir\phi.  
\end{align}
It now follows from \eqref{Phi:bound}, \eqref{eq:trans:2.1.2}, \eqref{k_xx:bound}, \eqref{z_xx:bound},  and \eqref{vorticity_xx:bound} that
\begin{align}
\label{eq:a_xxx:int}
\int^t_{-\varepsilon} \! \! \p^3_x(a\cir\phi) \: d\tau & = \phi^3_x ( [3 + 6c^{-1} z] c^{-1}\p_\theta c)\cir\phi + \mathcal O\big( \varepsilon^{-1} + \varepsilon |w''_0| \big).
\end{align}
Therefore,
\begin{equation}
\label{Phi_xx:bound}
\Phi_{xx} = 2\frac{c^2_0 c_0'''-3 c_0c_0'c_0''+ 2(c'_0)^3}{c_0^3} -\phi^3_x ( [16 + 32c^{-1} z] c^{-1}\p_\theta c)\cir\phi + \mathcal O\big( \varepsilon^{-1} + \varepsilon |w''_0| \big).
\end{equation}
These equations will be used in \S~\ref{sec:k_xxxx}, \S~\ref{sec:a_xxxx}, \S~\ref{sec:k_xxxxx}, \S~\ref{sec:a_xxxxx}. \footnote{For the fifth order estimates, one actually has to write out the full formula for \eqref{eq:a_xxx:int} and \eqref{Phi_xx:bound} and work with it. We will omit such straightforward but space-consuming details.}

\subsection{$\p_\theta^3 z$ bounds}
\label{sec:z_xxx}

We know that 
\begin{align*}
\psi_x (\p_\theta f\cir\psi) & = \bigg( \tfrac{1}{2} \psi_x \p_\theta k\cir\psi + \psi_x (c^{-1} \p_\theta c)\cir\psi - 3 \int^t_{-\varepsilon} \! \! \psi_x \p^2_\theta z\cir\psi \: d\tau \bigg) f\cir \psi + J_t c\cir\psi \p_x \bigg( J^{-1}_t (c^{-1}f)\cir\psi \bigg). \notag
\end{align*}
Recall from our Duhamel formula for $f$ that
\begin{align}
\label{f:duhamel:2}
&J^{-1}_t (c^{-1}f)\cir\psi \notag\\
& = c^{-1}_0 (z''_0 - \tfrac{1}{2}c^{-1}_0 c'_0 z'_0  + \tfrac{1}{2} c_0 k''_0)e^{-\frac{1}{2}k_0} \notag \\
& + \int^t_{-\varepsilon} \! \!J^{-1}_t (\big[ \tfrac{1}{12} \p_\theta k^2 - \tfrac{1}{4} c^{-2} \p_\theta z^2 - \tfrac{1}{24} \mathcal E - 2 \p_\theta a c^{-2} z + \tfrac{8}{3} c^{-1} z + \tfrac{16}{3} c^{-2} z^2 \big] \p_\theta w ) \cir\psi \: d\tau \notag \\
& + \int^t_{-\varepsilon} \! \!  J^{-1}_t ( \tfrac{1}{12} c \p_\theta k \mathcal E - \tfrac{1}{12} \p_\theta k^2 \p_\theta z + \tfrac{1}{4} c^{-2} \p_\theta z^3 + \tfrac{1}{8} \mathcal E \p_\theta z) \cir\psi \: d\tau \notag \\
& + \int^t_{-\varepsilon} \! \! J^{-1}_t (2 \p_\theta a c^{-2} \p_\theta z z - 8 c^{-1} \p_\theta z z - \tfrac{16}{3} c^{-2} \p_\theta z z^2 - 4 \p_\theta a c^{-1} \p_\theta z + \tfrac{2}{3} c\p_\theta \varpi e^{-k} z - \tfrac{2}{3} c\p_\theta k \varpi e^{-k} z) \cir\psi.
\end{align}

Taking $\p_x$ of \eqref{f:duhamel:2} and using \eqref{bd:trans:1.1.1} and \eqref{E:bound:2} we have
\begin{align}
J_t c\cir\psi \p_x \bigg( J^{-1}_t (c^{-1}f)\cir\psi \bigg) & = \psi_x( -\tfrac{1}{8} \p_\theta c\p_\theta k^2 + \tfrac{3}{8} c^{-2}\p_\theta c \p_\theta z^2 + \tfrac{1}{16}\mathcal E \p_\theta c)\cir\psi \notag \\
& + \psi_x( 3 \p_\theta ac^{-2} \p_\theta c z - 4 c^{-1} \p_\theta c z - 8 c^{-2} \p_\theta c z^2 ) \cir\psi  \notag \\
& + \mathcal O(\varepsilon^{\beta_3\wedge \gamma_3 \wedge \beta_2 -1 \wedge \gamma_2 -1 \wedge \beta_1 - 2} + \varepsilon^{\beta_1}|w''_0| + \varepsilon^{\beta_1+\gamma_1+1} \|w''_0\|_{L^\infty} ). 
\end{align}
Therefore
\begin{align}
\label{f_x:formula}
\p_\theta f & =  c^{-1} \p_\theta c f -\tfrac{1}{8} \p_\theta c\p_\theta k^2 + \tfrac{3}{8} c^{-2}\p_\theta c \p_\theta z^2 + \tfrac{1}{16}\mathcal E \p_\theta c \notag \\
& +  3 \p_\theta ac^{-2} \p_\theta c z - 4 c^{-1} \p_\theta c z - 8 c^{-2} \p_\theta c z^2  \notag \\
& + \mathcal O(\varepsilon^{\beta_3\wedge \gamma_3 \wedge \beta_2 -1 \wedge \gamma_2 -1 \wedge \beta_1 - 2} + \varepsilon^{\beta_1}|w''_0\cir\psi^{-1}| + \varepsilon^{\beta_1+\gamma_1+1} \|w''_0\|_{L^\infty} ).
\end{align}
Taking $\p_\theta$ of \eqref{z_x:trans} and using these bounds, we conclude
\begin{align}
\p^3_\theta z &= [\tfrac{1}{2} \p_\theta z-\tfrac{1}{2} c\p_\theta k] c^{-1} \p^2_\theta c \notag\\
&\qquad + \mathcal O \big(  \varepsilon^{\beta_3 \wedge \gamma_3 \wedge \beta_2 -1 \wedge \gamma_2 -1 \wedge \beta_1-2}  + \varepsilon^{\beta_2\wedge \gamma_2 \wedge \beta_1-1} |\p_\theta w| + \varepsilon^{\beta_1} |w''_0 \cir\psi^{-1}| + \varepsilon^{\beta_1} |\p_\theta w|^2 \big). 
\label{z_xxx:bound}
\end{align}
Since \eqref{f_t:formula} can be rewritten as
\begin{align*}
-\tfrac{3}{2}\p_t(f\cir\psi) & = ([f - \tfrac{1}{8} c \p_\theta k^2 + \tfrac{3}{8} c^{-1} \p_\theta z^2 + \tfrac{1}{16} c \mathcal E + 3 \p_\theta a c^{-1} z - 4z - 8 c^{-1} z^2] \p_\theta w)\cir\psi \notag \\
& +  ( (4a + \tfrac{1}{2} c\p_\theta k + \tfrac{9}{2} \p_\theta z)f + \tfrac{1}{8} c \p_\theta k^2 \p_\theta z - \tfrac{1}{8} c^2 \p_\theta k \mathcal E - \tfrac{3}{8} c^{-1} \p_\theta z^3)\cir\psi \notag \\
& + (-\tfrac{3}{16} c \p_\theta z \mathcal E - (3 \p_\theta a c^{-1} z -4z - 8 c^{-1} z^2) \p_\theta z + 6 \p_\theta a \p_\theta z) \cir\psi \notag \\
& + ( 8z \p_\theta z + c^2 \p_\theta \varpi e^{-k} z + c^2 \p_\theta \varpi e^{-k} z) \cir\psi. 
\end{align*}
we can also conclude that $f$ is transversal, and therefore $\p_\theta z$ is 2-transversal.

\section{Fourth derivative estimates}
\label{sec:4DE}

\subsection{$\p^4_\theta k$ bounds}
\label{sec:k_xxxx}

Taking $\p_x$ of $\tilde{\mathcal E}\cir\phi$ and using \eqref{k_x:formula} and \eqref{Phi_xx:bound} gives us
\begin{align}
\p_x( \tilde{\mathcal E}\cir\phi) & =  \phi^{-3}_x \big( \p^4_x k_0 - 6\Phi k'''_0 - (4\p_x\Phi - 11\Phi^2)k''_0 - (\p^2_x\Phi - 7 \p_x\Phi \Phi + 6\Phi^3) k'_0 \big)  + 6 \phi_x(\tilde{\mathcal E} c^{-1} \p_\theta c)\cir\phi \notag \\
& = \phi^{-3}_x \big( \p^4_x k_0 - 6\Phi k'''_0 - (4\p_x\Phi - 11\Phi^2)k''_0 - ( 2\tfrac{c^2_0 c_0'''-3 c_0c_0'c_0''+ 2(c'_0)^3}{c_0^3}- 7 \p_x\Phi \Phi + 6\Phi^3) k'_0 \big) \notag \\
& + \phi_x \big( \big[\p_\theta k [ 16+ 32 c^{-1}z] + 6 \tilde{\mathcal E} \big] c^{-1} \p_\theta c\big)\cir\phi + \mathcal O\big( \varepsilon^{\gamma_1-1} + \varepsilon^{\gamma_1+1}|w''_0|\big). \notag \\
\implies \p_\theta \tilde{\mathcal E} & = \big[\phi^{-4}_x \big( \p^4_x k_0 - 6\Phi k'''_0 - (4\p_x\Phi - 11\Phi^2)k''_0 - ( 2\tfrac{c^2_0 c_0'''-3 c_0c_0'c_0''+ 2(c'_0)^3}{c_0^3}- 7 \p_x\Phi \Phi + 6\Phi^3) k'_0 \big)\big] \cir\phi^{-1} \notag \\
& + \big[\p_\theta k [ 16+ 32 c^{-1}z] + 6 \tilde{\mathcal E} \big] c^{-1} \p_\theta c + \mathcal O\big( \varepsilon^{\gamma_1-1} + \varepsilon^{\gamma_1+1}|w''_0|\big). \notag 
\end{align}

Define
\begin{equation}
\label{E:definition:3}
\hat{\mathcal E}  = \p_\theta \tilde{\mathcal E} - \big[\p_\theta k [ 16+ 32 c^{-1}z] + 6 \tilde{\mathcal E} \big] c^{-1} \p_\theta c.
\end{equation}
Then \eqref{Phi:bound} and \eqref{Phi_x:bound} tell us that 
\begin{align*}
|\hat{\mathcal E}| & \lesssim \varepsilon^{\gamma_4\wedge \gamma_3-1\wedge \gamma_2-2 \wedge \gamma_1-3} + \varepsilon^{\gamma_2 \wedge \gamma_1-1} |w''_0\cir\phi^{-1}| + \varepsilon^{\gamma_1} |w'''_0\cir\phi^{-1}|.
\end{align*}

Using \eqref{a_xx:phidentity} one can compute that
\begin{align}
-3\p_t(\tilde{\mathcal E} \cir\phi) & = (6 \tilde{\mathcal E} \p_\theta w + 6 \tilde{\mathcal E} \p_\theta z) \cir\phi - 16 \phi^{-3}_x k''_0 \p_x(a\cir\phi) \notag \\
& + 16 \phi^{-3}_x k'_0 \big(4 \Phi \p_x(a\cir\phi) - \p^2_x (a\cir\phi)  \big) \notag \\
& = \big( \big[\p_\theta k [ 16+ 32 c^{-1}z] + 6 \tilde{\mathcal E} \big] \p_\theta w + 6 \tilde{\mathcal E} \p_\theta z\big)\cir\phi + \mathcal O\big( \varepsilon^{\gamma_2 \wedge \gamma_1-1} \big).
\end{align}
so $\tilde{\mathcal E}$ is transversal, and therefore $\p_\theta k$ is 3-transversal. This will be used in \S~\ref{sec:5DE}.

Taking $\p_\theta$ of \eqref{k_xxx:formula} gives us
\begin{align}
\label{k_xxxx:identity}
\p^4_\theta k & = \hat{\mathcal E} + \big[ 16 +32 c^{-1}z\big] c^{-1} \p_\theta c \p_\theta k + 12 \tilde{\mathcal E} c^{-1} \p_\theta c \notag \\
& + 20 \mathcal E c^{-2} \p_\theta c^2 + 8 \mathcal E c^{-1} \p_\theta^2 c + 6 c^{-2} \p_\theta c \p^2_\theta c \p_\theta k + 2 c^{-1} \p^3_\theta c \p_\theta k.
\end{align}
Note that the terms of order $|\p_\theta w|^3$ happen to cancel when this computation is done.

It follows from \eqref{E:bound}, \eqref{z_xx:bound}, \eqref{E:bound:2},  and \eqref{z_xxx:bound} that
\begin{align}
\label{k_xxxx:bound}
|\p^4_\theta k| & \lesssim  \varepsilon^{\gamma_4 \wedge \gamma_3-1 \wedge \gamma_2 -1 \wedge \gamma_1-3}  + \varepsilon^{\gamma_3 \wedge \gamma_2-1 \wedge \gamma_1-2} |\p_\theta w| \notag \\
& + \varepsilon^{\gamma_2 \wedge \gamma_1-1} \big( |w''_0\cir\phi^{-1}| + |\p_\theta w|^2 + |\p^2_\theta w| \big) \notag \\
& + \varepsilon^{\gamma_1}\big( |w'''_0\cir\phi^{-1}| + |w''_0\cir\phi^{-1}| |\p_\theta w|+ |\p_\theta w| |\p^2_\theta w| + |\p^3_\theta w| \big).
\end{align}

\subsection{$\p^4_\theta a$ bounds}
\label{sec:a_xxxx}

At this point, we can apply Lemma \ref{lem:trans:3} to conclude that the variable $Q_2$ defined in \eqref{phi_xx:formula:2} is transversal. In fact, Lemma \ref{lem:trans:3} allows us to conclude that $Q_2$ is 2-transversal, but we will not need to use that until \S~\ref{sec:5DE}.

Recall from \S~\ref{sec:a_xxx} that 
$$ \begin{cases} \p^2_\theta \varpi = A c^{-1} \p_\theta c + B \\ -3\p_t( \p_\theta \varpi \cir\phi) = (A \p_\theta w + C)\cir\phi \end{cases}  $$
where $A,B, C$ are defined by \eqref{def:A}, \eqref{def:B}, and \eqref{def:C}. Since $c, \p_\theta k, \p_\theta a, \p_\theta z, \p_\theta \varpi, \mathcal E$ are all transversal, it is immediate that $A$ and $C$ are transversal. We now also know that $f$ and $Q_2$ are transversal, so Lemma \ref{lem:trans:3} lets us conclude that $B$ is transversal. So $\p_\theta \varpi$ is 2-transversal. This fact will be utilized in \S~\ref{sec:z_xxxx}.

It is immediate from \eqref{def:A} that
\begin{align*}
\p_\theta^3 \varpi & = \p_\theta A c^{-1} \p_\theta c - A c^{-2} \p_\theta c^2 + A c^{-1} \p^2_\theta c + \p_\theta B.
\end{align*}
Since
\begin{align*}
\p_\theta A = [2A - 4c^{-1} e^k \p_\theta k] c^{-1} \p_\theta c + 2B- 4c^{-1} e^k \p_\theta k^2 - 4 c^{-1} e^k \mathcal E,
\end{align*}
we conclude that
\begin{align}
\label{eq:vorticity_xxx}
\p^3_\theta \varpi & = Ac^{-1} \p^2_\theta c + [A-4c^{-1} e^k \p_\theta k] c^{-2} \p_\theta c^2 + [2B-4c^{-1} e^k \p_\theta k^2 -4 c^{-1} e^k \mathcal E] c^{-1} \p_\theta c + \p_\theta B
\end{align}
So to estimate $\p^3_\theta \varpi$, all that remains is to estimate $\p_\theta B$.

We know from \eqref{eq:a_xxx:int} that
\begin{align}
\p^3_x \I_t & = \phi^3_x ( [3 + 6c^{-1} z] c^{-1}\p_\theta c)\cir\phi + \mathcal O\big( \varepsilon^{-1} + \varepsilon |w''_0| \big).
\end{align}
We know from \eqref{eq:trans:2.1.1} that 
\begin{equation*}
\p_x \bigg( \int^t_{-\varepsilon} \! \!  \I^{-3}_\tau \big(\big[ c^{-4} \p_\theta z + c^{-3} \p_\theta k \big] \p_\theta w\big) \cir\phi \: d\tau\bigg) = -3 \phi_x( [c^{-4} \p_\theta z + c^{-3} \p_\theta k] c^{-1} \p_\theta c) \cir\phi + \mathcal O( \varepsilon^{\beta_2+1 \wedge \gamma_2+1 \wedge \beta_1} ).
\end{equation*}
It is straightforward to compute that
\begin{align*}
\p_\theta Q_2 = 2Q_2 c^{-1} \p_\theta c + \mathcal O( \varepsilon^{-2} +  |w''_0\cir\phi^{-1}|).
\end{align*}
Taking $\p^3_x$ of \eqref{vorticity_0:formula} produces
\begin{align}
\label{vorticity_0xxx:bound}
|\varpi'''_0| & \lesssim \varepsilon^{-3} + \varepsilon^{-1} |w''_0| + |w'''_0|.
\end{align}
Therefore, taking $\p_x$ of \eqref{def:B} and using also \eqref{E:bound:2} and \eqref{f_x:formula}, we conclude that
\begin{align}
\label{B_x:bound}
|\p_\theta B| & \lesssim \varepsilon^{-3} + \varepsilon^{-1} |w''_0\cir\phi^{-1}| + |w'''_0\cir\phi^{-1}| + \big( \varepsilon^{-2} + |w''_0\cir\phi^{-1}| \big) |\p_\theta w|.
\end{align}
Therefore,
\begin{align}
\label{vorticity_xxx:bound}
|\p^3_\theta \varpi| & \lesssim \varepsilon^{-3} + \varepsilon^{-1} |w''_0\cir\phi^{-1}| + |w'''_0\cir\phi^{-1}| \notag \\
& + \big( \varepsilon^{-2} + |w''_0\cir\phi^{-1}| \big) |\p_\theta w| + \varepsilon^{-1}|\p_\theta w|^2 + \varepsilon^{-1} |\p^2_\theta w|.
\end{align}
Now, taking $\p^3_\theta$ of \eqref{a_x:formula} and using \eqref{k_xx:bound}, \eqref{vorticity_x:bound}, \eqref{k_xxx:bound},\eqref{vorticity_xx:bound},\eqref{z_xxx:bound}, and \eqref{vorticity_xxx:bound} gives us
\begin{align}
\label{a_xxxx:bound}
|\p^4_\theta a| & \lesssim \varepsilon^{-3} + \varepsilon^{-1} |w''_0\cir\phi^{-1}| + |w'''_0\cir\phi^{-1}| \notag \\
& + \big( \varepsilon^{-2} + |w''_0\cir\phi^{-1}| \big) |\p_\theta w| + \varepsilon^{-1}|\p_\theta w|^2 + \varepsilon^{-1} |\p^2_\theta w| \notag \\
& + |\p_\theta w| |\p^2_\theta w| + |\p^3_\theta w|.
\end{align}

\subsection{$\p^4_\theta z$ bounds}
\label{sec:z_xxxx}

Abusing notation, introduce a function $J$ defined so that
$$ J\cir\psi(x,t) = J_t(x). $$ 
Lemma \ref{lem:trans:3} implies that $J$ is 2-transversal. Using the new function $J$,  we can rewrite \eqref{f:duhamel:2} as
\begin{align*}
(J^{-1}c^{-1}f)\cir\psi & = c^{-1}_0 (z''_0 - \tfrac{1}{2}c^{-1}_0 c'_0 z'_0  + \tfrac{1}{2} c_0 k''_0)e^{-\frac{1}{2}k_0}   + \int^t_{-\varepsilon} \! \! (Jh_1\p_\theta w)\cir\psi \: d\tau + \int^t_{-\varepsilon} \! \! (Jh_2)\cir\psi \: d\tau
\end{align*} 
Given everything that has been proven up to this point $h_1$ and $h_2$ are both 2-transversal. It follows from Lemma \ref{lem:trans:3} that $J^{-1} c^{-1} f$ is 2-transversal, and since $c$ and $J$ are both 2-transversal it thus follows that $f$ is 2-transversal. This will be utilized in \S~\ref{sec:5DE}. 

Using Lemma \ref{lem:trans:1.2} on the functions $h_1$ and $h_2$,  along with estimates from the previous sections, we conclude that
\begin{align*}
\bigg\vert \p^2_x \bigg( \int^t_{-\varepsilon} \! \! (Jh_1\p_\theta w)\cir\psi \: d\tau + \int^t_{-\varepsilon} \! \! (Jh_2)\cir\psi \: d\tau \bigg) \bigg\vert & \lesssim \varepsilon^{\mu-4} + \varepsilon^{\mu-3} |\p_\theta w\cir\psi| + \varepsilon^{\mu-2} |\p_\theta w\cir\psi|^2 + \varepsilon^{\mu-2} |\p^2_\theta w\cir\psi|.
\end{align*}
It therefore follows that
\begin{align*}
\bigg\vert \p^2_x\bigg( (J^{-1}c^{-1}f)\cir\psi \bigg) \bigg\vert &  \lesssim \varepsilon^{\mu-5} + \varepsilon^{\mu-3} |\p_\theta w\cir\psi| + \varepsilon^{\mu-2} |\p_\theta w\cir\psi|^2 + \varepsilon^{\mu-2} |\p^2_\theta w\cir\psi|.
\end{align*}
We conclude that
\begin{align*}
|\p^2_\theta f| &  \lesssim \varepsilon^{\mu-5} + \varepsilon^{\mu-\frac{7}{2}} |\p_\theta w| + \varepsilon^{\mu-2} |\p_\theta w|^2 + \varepsilon^{\mu-2} |\p^2_\theta w|.
\end{align*}
It now follows that
\begin{align}
\label{z_xxxx:bound}
|\p^4_\theta z| & \lesssim \varepsilon^{\mu-5} + \varepsilon^{\mu-\frac{7}{2}} |\p_\theta w| + \varepsilon^{\mu-2} |\p_\theta w|^2 + \varepsilon^{\mu-2} |\p^2_\theta w| \notag \\
& + \varepsilon^{\beta_1} \big( |\p_\theta w|^3 + |\p_\theta w| |\p^2_\theta w| + |\p^3_\theta w|\big).
\end{align} 


\section{Fifth Order Estimates}
\label{sec:5DE}

\subsection{$\p^5_\theta k$ bounds}
\label{sec:k_xxxxx}

We already know (see \S~\ref{sec:k_xxxx}) that $\p_\theta k$ is 3-transversal, and we will not need to show that $\p_\theta k$ is 4-transversal. $\p_\theta k$ is 4-transversal, but it doesn't matter for our purposes. One can easily get the bound
\begin{align*}
|\p_\theta \hat{\mathcal E}| & \lesssim \varepsilon^{\gamma_5 \wedge \gamma_4-1 \wedge \gamma_3-2 \wedge \gamma_2-3\wedge \gamma_1-4}  +\varepsilon^{\gamma_3 \wedge \gamma_2-1 \wedge \gamma_1-2}|w''_0\cir\phi^{-1}| + \varepsilon^{\gamma_2 \wedge \gamma_1-1} |w'''_0\cir\phi^{-1}| \notag \\
& + \varepsilon^{\gamma_1} (|w''_0\cir\phi^{-1}|^2 + |\p^4_xw_0\cir\phi^{-1}|)    + ( \varepsilon^{\gamma_4 \wedge \gamma_3-1 \wedge \gamma_2-2 \wedge \gamma_1-3} + \varepsilon^{\gamma_2\wedge \gamma_1-1} |w''_0\cir\phi^{-1}| + \varepsilon^{\gamma_1} |w'''_0\cir\phi^{-1}| ) |\p_\theta w| \notag \\
\end{align*}
It now follows from taking $\p_\theta$ of \eqref{k_xxxx:identity} that
\begin{align}
\p^5_\theta k & = 2c^{-1} \p^4_\theta c \p_\theta k + \mathcal O\big( \varepsilon^{\gamma_5 \wedge \gamma_4-1 \wedge \gamma_3 -\frac{5}{2} \wedge \gamma_2-4 \wedge \gamma_1-\frac{11}{2}} + \varepsilon^{\gamma_4 \wedge\gamma_3-1 \wedge \gamma_2-\frac{5}{2} \wedge \gamma_1-4} |\p_\theta w| \notag \\
& \hspace{35mm} + \varepsilon^{\gamma_3 \wedge \gamma_2-1 \wedge \gamma_1 - \frac{5}{2}} ( |\p_\theta w|^2 + |\p_\theta^2 w|) \notag \\
& \hspace{35mm} + \varepsilon^{\gamma_2\wedge \gamma_1-1} (|\p_\theta w|^3 + |\p_\theta w| |\p_\theta^2 w| + |\p^3_\theta w|) \notag \\
& \hspace{35mm} + \varepsilon^{\gamma_1} ( |\p_\theta w|^2 |\p^2_\theta w| + |\p^2_\theta w|^2 + |\p_\theta w| |\p^3_\theta w|) \hspace{3mm} \big) .
\end{align}

\subsection{$\p^5_\theta a$ bounds}
\label{sec:a_xxxxx}

Since $c, \p_\theta k, \p_\theta a, \p_\theta z, \p_\theta \varpi,$ and $\mathcal E$ are all 2-transversal, it follows immediately that $A$ and $C$ are 2-transversal. $f$ and $Q_2$ are also 2-transversal, so it follows from Lemma ~\ref{lem:trans:3} that $B$ is 2-transversal. Therefore $\p_\theta \varpi$ is 3-transversal. This will be used in \S~\ref{sec:z_xxxxx}.

Taking $\p_\theta$ of \eqref{eq:vorticity_xxx} and using the bounds we already have on $A,B,C, \p_\theta A$, and $\p_\theta B$ gives us
\begin{align*}
|\p^4_\theta \varpi| & \lesssim |\p^2_\theta B| + \varepsilon^{-4} |\p_\theta w| + \varepsilon^{-5/2}(|\p_\theta w|^2 + |\p^2_\theta w| ) + \varepsilon^{-1}( |\p_\theta w|^3 + |\p_\theta w| |\p^2_\theta w| + |\p^3_\theta w|).
\end{align*}
$\p^2_\theta B$ can be bounded in a manner similar to the way $\p_\theta B$ was bounded. One simply needs to use a lemma similar to Lemma~\ref{lem:trans:1.2} but for the 2-characteristics, which is very straightforward to prove at this point. Then, since $\p^2_\theta Q$ and $\p^4_x \I_t$ can be explicitly computed and bounded \footnote{One must write out the full equation for \eqref{eq:a_xxx:int} in order to do this, which is arduous but straightforward.}, one can bound $\p^2_\theta B$ and conclude that
\begin{align*}
|\p^4_\theta \varpi| & \lesssim \varepsilon^{-11/2} + \varepsilon^{-4}|\p_\theta w| + \varepsilon^{-5/2}(|\p_\theta w|^2 + |\p^2_\theta w|) + \varepsilon^{-1}( |\p_\theta w|^3 + |\p_\theta w| |\p_\theta^2 w| + |\p^3_\theta w|).
\end{align*}
From here, taking $\p^3_\theta$ of \eqref{a_x:trans} gives
\begin{align}
\p^5_\theta a & = 2[\p_\theta a-c-2z]c^{-1} \p^4_\theta c + \mathcal O\big( \varepsilon^{-11/2} + \varepsilon^{-4}|\p_\theta w| + \varepsilon^{-5/2}(|\p_\theta w|^2 + |\p^2_\theta w|) \notag \\
& \hspace{50mm} + \varepsilon^{-1}( |\p_\theta w|^3 + |\p_\theta w| |\p_\theta^2 w| + |\p^3_\theta w|) \notag \\
& \hspace{50mm} + |\p_\theta w| |\p^3_\theta w| + |\p_\theta w|^2 |\p^2_\theta w| + |\p^2_\theta w|^2 \hspace{3mm} \big).
\end{align}

\subsection{$\p^5_\theta z$ bounds}
\label{sec:z_xxxxx}

One can use Lemma ~\ref{lem:trans:1.1}, Lemma ~\ref{lem:trans:1.2}, and Lemma ~\ref{lem:trans:3} to derive a lemma for 3-transversal functions analogous to Lemma ~\ref{lem:trans:1.2}. Bounding $\p^5_\theta z$ now follows in a manner completely analogous to \S~\ref{sec:z_xxxx}. One obtains
\begin{align}
\p^5_\theta z & = [\tfrac{1}{2} \p_\theta z - \tfrac{1}{2} c\p_\theta k] c^{-1} \p^4_\theta c +\mathcal O\big( \varepsilon^{\mu-13/2} + \varepsilon^{\mu-5}|\p_\theta w| + \varepsilon^{\mu-7/2}(|\p_\theta w|^2 + |\p^2_\theta w|) \notag \\
& \hspace{50mm} + \varepsilon^{\mu-2}( |\p_\theta w|^3 + |\p_\theta w| |\p_\theta^2 w| + |\p^3_\theta w|) \notag \\
& \hspace{50mm} \varepsilon^{\beta_1}( |\p_\theta w|^4 + |\p_\theta w| |\p^3_\theta w| + |\p_\theta w|^2 |\p^2_\theta w| + |\p^2_\theta w|^2) \hspace{3mm} \big).
\end{align}


\section{Estimates along $\eta$}
\label{sec:3char}

\subsection{Second derivative estimates $\eta$}
\label{sec:eta_xx}

It follows from the first derivative estimates that 
\begin{equation*}
|\p_xI_t| \lesssim \varepsilon^{\gamma_1}
\end{equation*}
where $I_t$ is the integrating factor in \eqref{eq:I_t}. It follows from the second derivative estimates that
\begin{align*}
\eta_x |\p^2_\theta k\cir \eta| & \lesssim \varepsilon^{\gamma_2 \wedge \gamma_1-1} \\
\eta_x |\p^2_\theta a\cir\eta| & \lesssim \varepsilon^{-1} \\
\eta_x |\p^2_\theta z\cir\eta| & \lesssim \varepsilon^{\beta_2 \wedge \gamma_2 \wedge \beta_1-1}
\end{align*}
Taking $\p_x$ of \eqref{q^w:equation} and using these bounds, we get
\begin{align*}
|\p_x\big( \eta_x q^w\cir\eta)| & \leq |w''_0| (1+\mathcal O( \varepsilon)) + \mathcal O(\varepsilon^{\gamma_2\wedge\gamma_1-1}) + \mathcal O(\varepsilon^{0 \wedge \beta_1+\gamma_1}) (\varepsilon + t) \sup_{[-\varepsilon, t]} \eta_{xx}
\end{align*}
Taking $\p_x$ of \eqref{eta_x:equation} and plugging in this estimate gives us
\begin{align}
\label{eta_xx:bound:1}
\sup_{[-\varepsilon, t]} |\eta_{xx}| & \leq (\varepsilon + t) |w''_0| (1 + \mathcal O(\varepsilon)) + \mathcal O(\varepsilon^{\beta_2\wedge \gamma_2 \wedge \beta_1-1}) )(\varepsilon+t) + \mathcal O(\varepsilon^{\beta_1}) (\varepsilon+t) \sup_{[-\varepsilon, t]} |\eta_{xx}|. \notag \\
\implies \sup_{[-\varepsilon, t]} |\eta_{xx}| & \leq \frac{ (\varepsilon + t)\big[ |w''_0| (1+\mathcal O(\varepsilon)) + \mathcal O(\varepsilon^{\beta_2 \wedge \gamma_2 \wedge \beta_1-1})\big]}{1-\mathcal O(\varepsilon^\mu)} \notag \\
& \lesssim (\varepsilon + t) (|w''_0| + \varepsilon^{\beta_2\wedge \gamma_2 \wedge \beta_1-1}).
\end{align}
It follows that
\begin{equation}
\label{eta_xx:bound:2}
 |\eta_{xx}(x,t)| \leq  \bound{\varepsilon^{-2}(\varepsilon+t)} {\varepsilon^{-5/2}(\eps+t)} .
\end{equation}
Plugging \eqref{eta_xx:bound:1} into our bound for $|\p_x(\eta_x q^w\cir\eta)|$ gives us
\begin{align}
|\p_x\big( \eta_x q^w\cir\eta)| & \leq |w''_0| (1+\mathcal O( \varepsilon)) + \mathcal O(\varepsilon^{\gamma_2\wedge\gamma_1-1})
\end{align}

Using the $\eta_{xx}$ bound along with our second derivative bounds, we get
\begin{align}
\label{eta:k_xx:bound}
|\p^2_x(k\cir\eta)| & \lesssim \varepsilon^{\gamma_2 \wedge \gamma_1-1} + \varepsilon^{\gamma_1}(\varepsilon + t) |w''_0| \\
|\p^2_x(a\cir\eta)| & \lesssim \varepsilon^{-1} + (\varepsilon + t) |w''_0| \\
|\p^2_x(z\cir\eta)| & \lesssim \varepsilon^{\beta_2 \wedge \gamma_2 \wedge \beta_1-1} + \varepsilon^{\beta_1} (\varepsilon + t) |w''_0| \\
|\p^2_x(\varpi\cir\eta)| & \lesssim \eps^{-\frac{5}{2}} + \eps^{-1}(\eps+t)|w''_0|.
\end{align}

Since
\begin{equation*}
\p^2_xI_t = \bigg( \tfrac{1}{8} \p^2_x(k\cir\eta) - \tfrac{8}{3} \int^t_{-\varepsilon} \! \! \p^2_x(a\cir\eta) \: d\tau \bigg) I_t + \bigg( \tfrac{1}{8} \p_x(k\cir\eta) - \tfrac{8}{3} \int^t_{-\varepsilon} \! \! \p_x(a\cir\eta) \: d\tau \bigg)^2 I_t,
\end{equation*}
it now follows that
\begin{align*}
|\p^2_x I_t| & \lesssim \varepsilon^{\gamma_2 \wedge \gamma_1-1} + \varepsilon^{\gamma_1}(\varepsilon+t)|w''_0|.
\end{align*}

Lastly, since
$$ \eta_xq^w\cir\eta = \p_x(w\cir\eta) - \tfrac{1}{4} c\cir\eta \p_x(k\cir\eta), $$
we know that
\begin{align*}
\p^2_x(w\cir\eta) & = \p_x(\eta_x q^w\cir\eta) + \tfrac{1}{4} \p_x(c\cir\eta) \p_x(k\cir\eta) + \tfrac{1}{4} c\cir\eta \p^2_x(k\cir\eta),
\end{align*}
and therefore \eqref{eta_xx:bound:1} and \eqref{eta:k_xx:bound} imply that
\begin{align}
|\p^2_x(w\cir\eta)| & \lesssim |w''_0| + \varepsilon^{\gamma_2 \wedge \gamma_1-1}.
\end{align}
Since
\begin{align*}
\eta^2_x \p^2_\theta w\cir\eta & = \p^2_x(w\cir\eta) - \eta_{xx} \p_\theta w\cir\eta,
\end{align*}
it follows that
\begin{align}
\label{eta_x:w_xx:bound}
\eta^2_x |\p^2_\theta w\cir\eta| & \lesssim |w''_0| + \varepsilon^{\gamma_2 \wedge \gamma_1-1}+  \varepsilon^{-1} \tfrac{|\eta_{xx}|}{\eta_x}. \\
\implies \eta^3_x |\p^2_\theta w\cir\eta| & \lesssim |w''_0| + \varepsilon^{\beta_2\wedge \gamma_2 \wedge \beta_1-1} . 
\end{align}

\subsection{Third derivative estimates along $\eta$}

Using the third derivative estimates and \eqref{eta_x:w_xx:bound} gives us
\begin{align*}
\eta^2_x|\p^3_\theta k\cir\eta| &  \lesssim \varepsilon^{\gamma_3 \wedge \gamma_2-1 \wedge \gamma_1-2} +\varepsilon^{\gamma_1} |w''_0| + \varepsilon^{\gamma_1} |w''_0 \cir\phi^{-1} \cir\eta| + \varepsilon^{\gamma_1-1} \tfrac{|\eta_{xx}|}{\eta_x} \\
\eta^2_x |\p^3_\theta a\cir\eta| & \lesssim \varepsilon^{-2} + |w''_0| + |w''_0\cir\phi^{-1}\cir\eta| + \varepsilon^{-1} \tfrac{|\eta_{xx}|}{\eta_x}. \\
\eta^2_x |\p^3_\theta z\cir\eta| & \lesssim \varepsilon^{\beta_3 \wedge \gamma_3 \wedge \beta_2-1 \wedge \gamma_2-1 \wedge \beta_1-2} + \varepsilon^{\beta_1} |w''_0| + \varepsilon^{\beta_1} |w''_0\cir\psi^{-1} \cir\eta| + \varepsilon^{\beta_1-1} \tfrac{|\eta_{xx}|}{\eta_x} \\
\eta^2_x |\p^3_\theta \varpi \cir \eta| & \lesssim \eps^{-3} + \eps^{-1}|w''_0| + \eps^{-1} |w''_0 \cir \phi^{-1} \cir \eta| + |w'''_0 \cir \phi^{-1} \cir \eta| + \eps^{-2} \tfrac{|\eta_{xx}|}{\eta_x}.
\end{align*}
These estimates will be useful in \S~\ref{sec:eta:4} and \S~\ref{sec:eta:5}.

Multiplying the above bounds by $\eta_x$ gives us
\begin{align*}
\eta^3_x|\p^3_\theta k\cir\eta| & \lesssim \varepsilon^{\gamma_3 \wedge \gamma_2 -1 \wedge \gamma_1-2} + \varepsilon^{\gamma_1} |w''_0| + \varepsilon^{\gamma_1} |w''_0 \cir\phi^{-1} \cir\eta| \\
\eta^3_x |\p^3_\theta a\cir\eta| & \lesssim \varepsilon^{-2} + |w''_0| + |w''_0 \cir\phi^{-1} \cir \eta| \\
\eta^3_x |\p^3_\theta z\cir\eta| & \lesssim \varepsilon^{\beta_3 \wedge \gamma_3 \wedge \beta_2-1\wedge \gamma_2-1 \wedge \beta_1-2} + \varepsilon^{\beta_1} |w''_0| + \varepsilon^{\beta_1} |w''_0\cir\psi^{-1} \cir\eta| \\
\eta^3_x |\p^3_\theta \varpi\cir\eta| & \lesssim \eps^{-3} + \eps^{-1}|w''_0| + \eps^{-1}|w''_0\cir\phi^{-1}\cir\eta| + |w'''_0 \cir \phi^{-1} \cir \eta|.
\end{align*}

Using this we compute that
\begin{align*}
\bigg\vert \p^2_x \big( \eta_x q^w\cir\eta\big) - w'''_0 e^{-\frac{1}{8}k_0} I_t \bigg\vert & \lesssim \varepsilon^{\gamma_1} |w''_0| + \varepsilon^{\gamma_3\wedge \gamma_2-1 \wedge \mu -2} + \varepsilon^{1 \wedge \beta_1+\gamma_1+1} ( |w''_0 \cir\psi^{-1} \cir \eta| + |w''_0 \cir \phi^{-1} \cir \eta| )  \notag \\
& + \sup_{[-\varepsilon, t]} |\eta_{xx}| \varepsilon^{0 \wedge \beta_2 +\gamma_1 + 1 \wedge \gamma_2 + \gamma_1 + 1 \wedge \beta_1 + \gamma_1} + \varepsilon^{0\wedge \beta_1+\gamma_1} (\varepsilon + t) \sup_{[-\varepsilon, t]} |\eta_{xxx}| \notag \\
& \lesssim \varepsilon^{\gamma_1} |w''_0| + \varepsilon^{-2 \wedge \gamma_3 \wedge \beta_2 + \gamma_1-1 \wedge \gamma_2-1 \wedge \beta_1+\gamma_1-2} + \varepsilon^{0\wedge \beta_1+\gamma_1} (\varepsilon + t) \sup_{[-\varepsilon, t]} |\eta_{xxx}|. \notag 
\end{align*}
This is true for all $x \in\mathbb{T}, t \in [-\varepsilon, T_*)$.

Taking $\p^2_x$ of \eqref{eta_x:equation} and using this bound tells us that
\begin{align*}
\sup_{[-\varepsilon, t]} |\eta_{xxx}| & \lesssim (\varepsilon+t) \big( |w'''_0| + \varepsilon^{\mu-4} + \varepsilon^{\beta_1} \sup_{[-\varepsilon,t]} |\eta_{xxx}| \big). \notag \\
\implies  |\eta_{xxx}| & \lesssim (\varepsilon + t) \big( |w'''_0| + \varepsilon^{\mu-4} \big) 
\end{align*}
everywhere. Using this bound, we conclude that
\begin{align*}
\bigg\vert \eta_{xxx} - w'''_0 \int^t_{-\varepsilon} \! \! e^{-\frac{1}{8} k_0}I_\tau \: d\tau \bigg\vert & \lesssim (\varepsilon +t)^2 \varepsilon^{\beta_1} |w'''_0| + (\varepsilon + t) \varepsilon^{\mu-4}.
\end{align*}
Since $w'''_0 \sim \varepsilon^{-4}$ for $|x| \leq \varepsilon^{3/2}$ and  $\|w'''_0\|_{L^\infty} \lesssim \varepsilon^{-4}$, this bound lets us conclude that
\begin{equation*}
\eta_{xxx} \sim (\varepsilon+t) \varepsilon^{-4} \hspace{10mm} \forall \: |x| \leq \varepsilon^{3/2},
\end{equation*}
and
\begin{equation*}
|\eta_{xxx}|  \lesssim \varepsilon^{-3} \hspace{10mm} \forall \: (x,t) \in \mathbb{T} \times [-\varepsilon, T_*].
\end{equation*}

We now conclude that
$$ \bigg\vert \p^2_x\big( \eta_x q^w\cir\eta\big) \bigg\vert \lesssim |w'''_0| + \varepsilon^{\mu} |w''_0| + \varepsilon^{-2 \wedge \gamma_3 \wedge \beta_2 + \gamma_1-1 \wedge \gamma_2-1 \wedge \beta_1+\gamma_1-2}  $$

We know that for all $(x,t) \in \mathbb{T} \times [-\varepsilon, T_*)$ we have
\begin{align*}
|\p^3_x(k\cir\eta)| & \lesssim ( \varepsilon^{\gamma_3 \wedge \gamma_2 -1 \wedge \gamma_1-2} + \varepsilon^{\gamma_1} |w''_0| + \varepsilon^{\gamma_1} |w''_0 \cir\phi^{-1} \cir\eta| ) \eta_x + \varepsilon^{\gamma_2 \wedge \gamma_1-1} |\eta_{xx}| + \varepsilon^{\gamma_1} |\eta_{xxx}|. \\
|\p^3_x(a\cir\eta)| & \lesssim (\varepsilon^{-2} + |w''_0| + |w''_0 \cir\phi^{-1} \cir \eta|) \eta_x + \varepsilon^{-1} |\eta_{xx}| + |\eta_{xxx}|. \\
|\p^3_x(z\cir\eta)| & \lesssim ( \varepsilon^{\beta_3 \wedge \gamma_3 \wedge  \beta_2-1 \wedge \gamma_2-1 \wedge \beta_1-2} +\varepsilon^{\beta_1} |w''_0| + \varepsilon^{\beta_1} |w''_0 \cir \psi^{-1} \cir \eta| ) \eta_x + \varepsilon^{\beta_2 \wedge \gamma_2 \wedge \beta_1-1} |\eta_{xx}| + \varepsilon^{\beta_1} |\eta_{xxx}| \\
|\p^3_x(\varpi\cir\eta)| & \lesssim ( \eps^{-3} + \eps^{-1}|w''_0| + \eps^{-1}|w''_0\cir\phi^{-1}\cir\eta| + |w'''_0\cir\phi^{-1}\cir\eta|) \eta_x + \eps^{-2} |\eta_{xx}| + \eps^{-1}|\eta_{xxx}|.
\end{align*}

Therefore, we have the bounds
\begin{align*}
|\p^3_x(k\cir\eta)| & \lesssim \varepsilon^{\mu-3} \\
|\p^3_x(a\cir\eta)| & \lesssim \varepsilon^{-3} \\
|\p^3_x(z\cir\eta)| & \lesssim \varepsilon^{\mu-4} \\
|\p^3_x(\varpi\cir\eta)| & \lesssim \eps^{-4}.
\end{align*}
Since
\begin{align*}
\p^2_x(\eta_x q^w\cir\eta) & = \p^3_x(w\cir\eta) - \tfrac{1}{4} \p^2_x(c\cir\eta) \p_x(k\cir\eta) -\tfrac{1}{2} \p_x(c\cir\eta) \p^2_x(k\cir\eta) -\tfrac{1}{4} c\cir\eta \p^3_x (k\cir\eta).
\end{align*}
It follows that
\begin{align}
|\p^3_x(w\cir\eta)| & \lesssim \varepsilon^{-4}.
\end{align}

Lastly, it is easy to use the bounds on $\p^3_x(k\cir\eta)$ and $\p^3_x(a\cir\eta)$ to conclude that
$$ |\p^3_x I_t| \lesssim \varepsilon^{\mu-3}. $$

\subsection{Blowup time, location, and sharp bounds for $\eta_x$ and $\eta_{xx}$}
\label{sec:x_*}

\begin{lemma}[Existence and uniqueness of blowup label] 
\label{lem:x_*}
There exists a unique label $x_* \in \mathbb{T}$ such that $\eta_x(x_*, T_*) =0$. Furthermore, we have $|x_*| \lesssim \varepsilon^{\mu+2}$ and
$$ \eta_x(x_*,T_*) = \eta_{xx}(x_*, T_*)  = 0. $$
\end{lemma}

\begin{proof}[Proof of Lemma ~\ref{lem:x_*}]
Due to \eqref{eta_x:lowerbound}, we know that $\eta_x$ is bound below outside of $(x,t) \in [-\varepsilon^{3/2}, \varepsilon^{3/2}] \times [-\varepsilon, T_*]$. We know that $\eta_{xxx} > 0$ in $[-\varepsilon^{3/2}, \varepsilon^{3/2}] \times (-\varepsilon, T_*]$, so for all $t \in (-\varepsilon, T_*]$ there is at most one zero of $\eta_{xx}(\cdot, t)$ in $(-\varepsilon^{3/2}, \varepsilon^{3/2}]$.

We know from \S~\ref{sec:eta_xx} that for all $(x,t) \in \mathbb{T} \times [-\varepsilon, T_*]$ we have
\begin{align*}
\big\vert \eta_{xx}(x,t) - w''_0(x) \int^t_{-\varepsilon} \! \! e^{-\frac{1}{8}k_0(x)} I_\tau(x) \: d\tau \big\vert & \lesssim (\varepsilon+t)\big( \varepsilon^{\beta_1+1} |w''_0(x)| +  \varepsilon^{\beta_2\wedge \gamma_2 \wedge \beta_1-1}\big).
\end{align*} 
Recall that $|w''_0(x)| \lesssim \varepsilon^{-2}$ for $|x| \leq \varepsilon^2$. It follows that for $|x| \leq \varepsilon^2$ we have
\begin{align*}
\big\vert \eta_{xx} - w''_0 \int^t_{-\varepsilon} \! \! e^{-\frac{1}{8}k_0} I_\tau \: d\tau \big\vert & \lesssim (\varepsilon+t) \varepsilon^{\mu-2}.
\end{align*} 
Since $w''_0(0) = 0$ and $w'''_0 \sim \varepsilon^{-4}$ for $|x| \leq \varepsilon^{3/2}$, it follows that
$$ |w''_0(x)| \gtrsim \varepsilon^{-4} |x| \text{ and } \sgn(w''_0(x)) = \sgn(x) $$
for $|x| \leq \varepsilon^{3/2}$. Therefore, we have
$$ |\eta_{xx}| \gtrsim (\varepsilon+t) \varepsilon^{-4} [ |x| - \mathcal O(\varepsilon^{\mu+2})] \hspace{10mm} \forall \: |x| \leq \varepsilon^2. $$
It follows that there exists a constant $C$ such that for $C\varepsilon^{2+\mu} < x < \varepsilon^2$ we have $\eta_{xx}(x,t) > 0$ and for $-\varepsilon^2 < x < -C \varepsilon^{2+\mu}$ we have $\eta_{xx}(x,t) < 0$. So for all $t \in (-\varepsilon, T_*]$ there exists a unique zero of $\eta_{xx}(\cdot, t)$ in $(-\varepsilon^{3/2}, \varepsilon^{3/2})$. 

 Therefore,  we conclude that there exists a $C^2$ curve $x_*: (-\varepsilon, T_*] \rightarrow \R$ such that
$$ \{ (x,t) : |x| \leq \varepsilon^{3/2}, \eta_{xx}(x,t) = 0, -\varepsilon < t \leq  T_* \} = \{ (x_*(t),t) : -\varepsilon < t \leq T_* \}. $$
Furthermore, we know that $|x_*(t)| \leq C \varepsilon^{2+\mu}$ for all $t$. From here it is easy to conclude that $\eta_{xx}(x,t) < 0$ for $-\varepsilon^{3/2} \leq x < x_*(t)$ and $\eta_{xx}(x,t) > 0$ for $x_*(t) < x \leq \varepsilon^{3/2}$, so that $x_*(t)$ must be the minimizer of $\eta_x(\cdot, t)$ over $[-\varepsilon^{3/2}, \varepsilon^{3/2}]$.

Define $x_* : = x_*(T_*)$. We know that $\min_{\mathbb{T}} \eta_x(\cdot, t) \rightarrow 0$ as $t \rightarrow T_*$ and $\eta_x$ is bound below for $|x| \geq \varepsilon^{3/2}$, so $\eta_x(x_*(t),t) \rightarrow 0$ as $t \rightarrow T_*$. Our result now follows.
\end{proof}

We can now improve upon our lower bounds for $\eta_x$. Let $x_*(t)$ be the curve from the proof of Lemma \ref{lem:x_*}. If $t> -\varepsilon$ and $x \in (-\pi, \pi]$, there exists $\bar{x}(x,t)$ in between $x$ and $x_*$ such that 
\begin{align*}
\eta_x(x,t) & = \eta_x(x_*(t),t) + \frac{\eta_{xxx}(\bar{x}(x,t), t)}{2} (x-x_*(t))^2 \notag \\
& \geq \frac{\eta_{xxx}(\bar{x}(x,t), t)}{2} (x-x_*(t))^2.
\end{align*}
Since $|x_*| \lesssim \varepsilon^{2+\mu}$, if $\varepsilon^2 \leq |x| \leq \varepsilon^{3/2}$, then $(x-x_*(t))^2 \gtrsim \varepsilon^4$ and $\eta_{xxx}(\bar{x},t) \gtrsim (\varepsilon + t) \varepsilon^{-4}$, so we have
\begin{align*}
\eta_x(x,t) & \gtrsim (\varepsilon + t).
\end{align*}
It follows that for $\varepsilon^2 \leq |x| \leq \varepsilon^{3/2}$, $-\frac{\varepsilon}{2} \leq t \leq T_*$ we have $\eta_x \gtrsim \varepsilon$. We already know (see \S~\ref{sec:T_*}) that
$$ \eta_x \geq -\tfrac{t}{\varepsilon} + \mathcal O(\varepsilon^\mu) $$
for all $(x,t) \in \mathbb{T} \times [-\varepsilon, T_*]$, so we conclude that 
\begin{align}
\label{eta_x:lowerbound:2}
\eta_x(x,t)  \gtrsim \varepsilon  \hspace{6mm} \text{for } \varepsilon^2 \leq |x| \leq \varepsilon^{3/2}.
\end{align}

\begin{lemma}[Improved estimates for $\eta_x$ and $\eta_{xx}$]
\label{lem:eta_x:improve}
There exist constants $A,c,C$ such that for all $(x,t) \in [-\varepsilon^2, \varepsilon^2] \times [-\varepsilon, T^*)$, we have
\begin{align}
\frac{1}{2\varepsilon} (T_*-t) + c(\varepsilon+t) \varepsilon^{-4} (x-x_*)^2 & \leq \eta_x(x,t)  \leq \frac{3}{2\varepsilon} (T_*-t) + C \varepsilon^{-3} (x-x_*)^2    \\
-A \varepsilon^{-2} (T_*-t) + c(\varepsilon+t) \varepsilon^{-4} (x-x_*)&  \leq \eta_{xx}(x,t)   \leq A \varepsilon^{-2} (T_*-t) + C \varepsilon^{-3} (x-x_*)    \text{ if } x \geq x_* \\
-A \varepsilon^{-2} (T_*-t) + C \varepsilon^{-3} (x-x_*) & \leq \eta_{xx}(x,t)   \leq A \varepsilon^{-2} (T_*-t) + c(\varepsilon+t) \varepsilon^{-4} (x-x_*)   \text{ if } x \leq x_*
\end{align}
\end{lemma}

\begin{proof}[Proof of Lemma~\ref{lem:eta_x:improve}]
Fix a point $(x,t) \in [-\varepsilon^2, \varepsilon^2] \times [-\varepsilon, T^*)$. We know that $\eta_x$ is $C^1$ on $\mathbb{T} \times [-\varepsilon, T_*]$ and is $C^2$ on $\mathbb{T} \times [-\varepsilon, T_*)$. Therefore, Taylor's theorem tells us that there exists a point $(x_1,t_1)$ on the segment connecting $(x_*, T_*)$ to $(x,t)$ such that
\begin{align}
\eta_x(x,t) & = \eta_{xt}(x_*,T_*) (t-T_*) + \tfrac{1}{2} \eta_{xxx}(x_1,t_1) (x-x_*)^2  \notag \\
& + \eta_{xxt}(x_1,t_1)(t-T_*)(x-x_*) + \tfrac{1}{2} \eta_{xtt} (x_1,t_1) (t-T_*)^2.
\end{align}
Similarly, there exists a point $(x_2,t_2)$ on the segment such that
\begin{align}
\eta_{xx}(x,t) & = \eta_{xxx} (x_2,t_2) (x-x_*) + \eta_{txx}(x_2,t_2) (t-T_*).
\end{align}

We know that
$$ \eta_{xt} = w'_0 e^{-\frac{1}{8}k_0} I_t + \mathcal O(\varepsilon^{\beta_1}). $$
We also know that since $w''_0(0) = 0$, $|x_*| \lesssim \varepsilon^{2+\mu}$, and $|w'''_0| \lesssim \varepsilon^{-4}$ we have
$$ -\tfrac{1}{\varepsilon} \leq w'_0(x_*) \leq - \tfrac{1+ C\varepsilon^{1+2\mu}}{\varepsilon}. $$
$$ -\tfrac{1}{\varepsilon}-\mathcal O(\varepsilon^{\gamma_1}) \leq \eta_{xt}(x_*,T_*) \leq - \tfrac{1+ C\varepsilon^{1+2\mu}}{\varepsilon} + \mathcal O(\varepsilon^{\gamma_1}). $$
We also know that for $i=1,2$
$$ \eta_{xxx} (x_i,t_i) \sim (\varepsilon + t_i) \varepsilon^{-4}. $$
$$ \eta_{txx} = \p_x(\eta_xq^w\cir\eta) + \tfrac{1}{4} \p_x (c\cir\eta) \p_x(k\cir\eta) + \tfrac{1}{4} c\cir\eta \p^2_x(k\cir\eta) + \tfrac{1}{3} \p^2_x(z\cir\eta). $$
So for $i=1,2$
$$ |\eta_{txx}(x_i,t_i)| \lesssim \varepsilon^{-2}. $$
Also
\begin{align*}
\eta_{xtt} & = (w'_0 -\tfrac{1}{4}c_0k'_0) \p_t\big( I_t e^{-\frac{1}{8}k_0}\big) + ( \tfrac{1}{12} c\p_\theta q^z - \tfrac{8}{3} \p_\theta a w)\cir\eta \notag \\
& + \eta_{xt} ( \tfrac{1}{4} \p_\theta k + \tfrac{1}{3} \p_\theta z)\cir\eta + \eta_x ( \tfrac{1}{4} \p_t(\p_\theta k\cir\eta) + \tfrac{1}{3} \p_t(\p_\theta z\cir\eta) ). \\
\implies |\eta_{xtt}(x_1,t_1)| & \lesssim \varepsilon^{\gamma_2 \wedge \beta_1-1}.
\end{align*}
Our result now follows.
\end{proof}

Using Lemma ~\ref{lem:eta_x:improve}, we can now conclude that
\begin{align}
\frac{1}{\eta_x} & \leq  \bound{\big[\frac{1}{2\varepsilon} (T_*-t) + c(\varepsilon+t) \varepsilon^{-4} (x-x_*)^2\big]^{-1}}{\varepsilon^{-1}}, \label{1/eta_x:bound} \\
\frac{\eta_{xx}^2}{\eta_x} & \leq \bound{\varepsilon^{-3}}{\eps^{-4}}. \label{eta_xx/eta_x:bound}
\end{align}
The bound \eqref{1/eta_x:bound} will let us deduce \eqref{regularity} and \eqref{eta_xx/eta_x:bound} will be used frequently in \S~\ref{sec:eta:4} and \S~\ref{sec:eta:5}.

\subsection{Fourth derivative estimates along $\eta$ }
\label{sec:eta:4}

We know that
\begin{align}
\eta^4_x\p^3_\theta w\cir\eta = \eta_x \p^3_x(w\cir\eta) - 3\eta_{xx} \p^2_x(w\cir\eta) + \big( 3 \tfrac{\eta_{xx}^2}{\eta_x} - \eta_{xxx} \big) \p_x(w\cir\eta).
\end{align}
Therefore, 
$$ \eta^4_x |\p^3_\theta w\cir\eta| \lesssim \varepsilon^{-4} + \varepsilon^{-1} \tfrac{\eta_{xx}^2}{\eta_x} \leq \bound{ \varepsilon^{-4}}{\varepsilon^{-5}}. $$
It now follows that
\begin{align}
\eta^4_x |\p^4_\theta k\cir\eta| & \leq \bound{\varepsilon^{\gamma_2 - 5/2 \wedge \mu-4}}{ \varepsilon^{\mu-5}}, \\
\eta^4_x|\p^4_\theta a\cir\eta| & \leq \bound{ \varepsilon^{-4}}{ \varepsilon^{-5}}, \\
\eta^4_x |\p^4_\theta z\cir\eta| & \leq \bound{\varepsilon^{\mu-5}}{\varepsilon^{\mu-6}} \\
\eta^4_x|\p^4_\theta \varpi \cir \eta| & \leq \bound{\eps^{-5}}{\eps^{-6}}.
\end{align}

The usual argument for bounding derivatives of $\eta_x$ and $\eta_x q^w\cir\eta$ now gives
$$ |\p^4_x\eta| \leq \bound{ (\varepsilon + t) ( |\p^4_x w_0| + \varepsilon^{\mu-5} )} { (\varepsilon + t) ( |\p^4_x w_0| + \varepsilon^{\mu-6} )} $$
and
$$ \big\vert \p^3_x(\eta_x q^w\cir\eta) \big\vert \leq \bound{ \varepsilon^{\mu-5}}{ \varepsilon^{-\frac{11}{2}}}. $$

In the end, we get
\begin{align}
|\p^4_x(k\cir\eta)| & \leq \bound{ \varepsilon^{\gamma_2-3\wedge \mu-4}}{ \varepsilon^{\gamma_2-4 \wedge \mu-5}},   \\
|\p^4_x(a\cir\eta)| & \leq \bound{ \varepsilon^{-4}}{ \varepsilon^{-5}}, \\
|\p^4_x(z\cir\eta)| & \leq \bound{\varepsilon^{\mu-5}}{ \varepsilon^{\mu-6}}, \\
|\p^4_x(\varpi\cir\eta)| & \leq \bound{\eps^{-5}}{\eps^{-6}}, \\
|\p^4_x(w\cir\eta)| & \leq \bound{ \varepsilon^{\mu-5}}{ \varepsilon^{-\frac{11}{2}}}.
\end{align}

\subsection{Fifth Derivative Estimates}
\label{sec:eta:5}

These estimates are different from the previous sections because they require more algebra and hinge on admittedly unexpected cancellation. First note that
\begin{align*}
\eta^2_x \p^2_\theta c\cir\eta &= -\p_x(c\cir\eta) \tfrac{\eta_{xx}}{\eta_x} + \p^2_x(c\cir\eta) \\
\eta^3_x \p^3_\theta c\cir\eta & = (3\tfrac{\eta_{xx}^2}{\eta_x} - \eta_{xxx}) \p_x(c\cir\eta) \tfrac{1}{\eta_x} -3\p^2_x(c\cir\eta) \tfrac{\eta_{xx}}{\eta_x} + \p^3_x(c\cir\eta) \\
\eta_{xx} \eta^3_x (\p_\theta c\p^2_\theta c)\cir\eta & = \p_x(c\cir\eta)[ \eta_{xx} \p^2_x(c\cir\eta) - \p_x(c\cir\eta) \tfrac{\eta_{xx}^2}{\eta_x}] \\
\eta_{xx}\eta^3_x \p^3_\theta c\cir\eta & = (3\tfrac{\eta_{xx}^2}{\eta_x}-\eta_{xxx}) \p_x(c\cir\eta) \tfrac{\eta_{xx}}{\eta_x} - 3\p^2_x(c\cir\eta) \tfrac{\eta_{xx}^2}{\eta_x} + \p^3_x(c\cir\eta) \eta_{xx} \\
\eta^5_x \p^2_\theta c^2\cir\eta & = \p_x(c\cir\eta)^2 \tfrac{\eta_{xx}^2}{\eta_x} - 2 \p^2_x(c\cir\eta) \p_x(c\cir\eta) \eta_{xx} + \p^2_x(c\cir\eta)^2 \eta_x \\
\eta^5_x(\p_\theta c^2 \p^2_\theta c)\cir\eta & = \p_x(c\cir\eta)^2 [ -\p_x(c\cir\eta) \eta_{xx} + \p^2_x(c\cir\eta) \eta_x] \\
\eta^5_x(\p_\theta c\p^3_\theta c)\cir\eta & = \p_x(c\cir\eta)[ (3\tfrac{\eta_{xx}^2}{\eta_x} -\eta_{xxx}) \p_x(c\cir\eta) - 3\p^2_x(c\cir\eta) \eta_{xx} + \p^3_x(c\cir\eta) \eta_x] \\
\eta^5_x \p^4_\theta c\cir\eta & =(10\eta_{xxx} - 15\tfrac{\eta_{xx}^2}{\eta_x}) \p_x(c\cir\eta) \tfrac{\eta_{xx}}{\eta_x} - \p^4_x\eta \p_x(c\cir\eta) + (15 \tfrac{\eta_{xx}^2}{\eta_x} - 4\eta_{xxx}) \p^2_x(c\cir\eta) \\
&  - 6 \eta_{xx} \p^3_x(c\cir\eta) + \eta_x \p^4_x(c\cir\eta).
\end{align*}
Next note that
\begin{align*}
\p^2_\theta a & = 2[\p_\theta a - c + 2z] c^{-1} \p_\theta c + \mathcal O(\varepsilon^{-1}), \\
\p^3_\theta a & = 2[\p_\theta a - c + 2z] c^{-1} \p^2_\theta c + \mathcal O(\varepsilon^{-\frac{5}{2}}+ \varepsilon^{-1}|\p_\theta w| + |\p_\theta w|^2), \\
\p^4_\theta a & = 2[\p_\theta a - c + 2z] c^{-1} \p^3_\theta c + \mathcal O(\varepsilon^{-4} + \varepsilon^{-\frac{5}{2}}|\p_\theta w| + \varepsilon^{-1} |\p_\theta w|^2 + \varepsilon^{-1} |\p^2_\theta w| + |\p_\theta w| |\p^2_\theta w| ), \\
\p^5_\theta a & = 2[\p_\theta a-c-2z]c^{-1} \p^4_\theta c + \mathcal O\big( \varepsilon^{-\frac{11}{2}} + \varepsilon^{-4}|\p_\theta w| + \varepsilon^{-5/2}(|\p_\theta w|^2 + |\p^2_\theta w|) \notag \\
& \hspace{50mm} + \varepsilon^{-1}( |\p_\theta w|^3 + |\p_\theta w| |\p_\theta^2 w| + |\p^3_\theta w|) \notag \\
& \hspace{50mm} + |\p_\theta w| |\p^3_\theta w| + |\p_\theta w|^2 |\p^2_\theta w| + |\p^2_\theta w|^2 \hspace{2mm} \big).
\end{align*}
Combining these identities and our estimates gives us
\begin{align*}
10 \eta_{xxx} \tfrac{\eta_{xx}}{\eta_x} \eta_x \p^2_\theta a\cir\eta & =2[\p_\theta ac^{-1}-1-2c^{-1}z]\cir\eta (10 \eta_{xxx} + 0 ) \p_x(c\cir\eta) \tfrac{\eta_{xx}}{\eta_x} + \bound{\varepsilon^{-5}}{\varepsilon^{-\frac{11}{2}}} , \\
(10\eta_{xxx} + 15 \tfrac{\eta_{xx}^2}{\eta_x}) \eta^2_x \p^3_\theta a\cir\eta & = 2[\p_\theta ac^{-1}-1-2c^{-1}z]\cir\eta (-10 \eta_{xxx} - 15 \tfrac{\eta_{xx}^2}{\eta_x} ) \p_x(c\cir\eta) \tfrac{\eta_{xx}}{\eta_x} +  \bound{ \varepsilon^{-\frac{11}{2}}}{\varepsilon^{-\frac{13}{2}}}, \\
10\eta_{xx} \eta^3_x \p^4_\theta a\cir\eta & = 2[\p_\theta ac^{-1}-1-2c^{-1}z]\cir\eta (-10 \eta_{xxx} + 30 \tfrac{\eta_{xx}^2}{\eta_x} ) \p_x(c\cir\eta) \tfrac{\eta_{xx}}{\eta_x} + \bound{\varepsilon^{-5}}{\varepsilon^{-\frac{13}{2}}} ,\\
\eta^5_x \p^5_\theta a\cir\eta & = 2[\p_\theta ac^{-1}-1-2c^{-1}z]\cir\eta(10\eta_{xxx} - 15 \tfrac{\eta_{xx}^2}{\eta_x}) \p_x(c\cir\eta) \tfrac{\eta_{xx}}{\eta_x} + \bound{\varepsilon^{-\frac{11}{2}}}{\varepsilon^{-\frac{13}{2}}}.
\end{align*}
Therefore,
\begin{align*}
\p^5_x(a\cir\eta) & = \eta^5_x \p^5_\theta a\cir\eta + 10\eta_{xx} \eta^3_x \p^4_\theta a\cir\eta + (10\eta_{xxx} + 15 \tfrac{\eta_{xx}^2}{\eta_x}) \eta^2_x \p^3_\theta a\cir\eta \\
& + (5\p^4_x \eta + 10 \eta_{xxx} \tfrac{\eta_{xx}}{\eta_x}) \eta_x \p^2_\theta a\cir\eta + \p^5_x \eta \p_\theta a\cir\eta \\
& = 2[\p_\theta ac^{-1}-1-2c^{-1}z]\cir\eta \big( [10-10-10+10] \eta_{xxx} + [0-15+30-15]\tfrac{\eta_{xx}^2}{\eta_x}\big) \p_x(c\cir\eta) \tfrac{\eta_{xx}}{\eta_x}  \\
& + \p^5_x\eta \p_\theta a\cir\eta + \bound{\varepsilon^{-\frac{11}{2}}}{\varepsilon^{-\frac{13}{2}}} \\
& = \p^5_x\eta \p_\theta a\cir\eta + \bound{\varepsilon^{-\frac{11}{2}}}{\varepsilon^{-\frac{13}{2}}}.
\end{align*}
The exact same cancellation occurs for the other two variables to give us
\begin{align*}
\p^5_x(k\cir\eta) & = \p^5_x\eta \p_\theta k\cir\eta + \bound{\varepsilon^{\gamma_2-4\wedge \mu-\frac{11}{2}}}{\varepsilon^{\mu-\frac{13}{2}}}, \\
\p^5_x(z\cir\eta) & = \p^5_x\eta \p_\theta z\cir\eta + \bound{\varepsilon^{\mu-\frac{13}{2}}}{\varepsilon^{\mu-\frac{15}{2}}}.
\end{align*}
Similar computations prove that
\begin{align*}
\p^4_x(\eta_x(\p_\theta k \p_\theta z)\cir\eta) & = \p^5_x(k\cir\eta)\p_\theta z\cir\eta + \p_\theta k\cir\eta \p^5_x(z\cir\eta) + \bound{\varepsilon^{\gamma_2+\mu-5\wedge 2\mu -\frac{13}{2}}}{\varepsilon^{2\mu - \frac{15}{2}}} .
\end{align*} 
Now the usual method for bounding the derivatives of $\eta_x$ and $\eta_x q^w\cir\eta$ produces
\begin{align}
\label{eta_xxxxx:bound}
|\p^5_x\eta| & \leq \bound{(\varepsilon + t) (|\p^5_x w_0| + \varepsilon^{\mu-\frac{13}{2}})}{ (\varepsilon + t) (|\p^5_x w_0| + \varepsilon^{\mu-\frac{15}{2}} )}.
\end{align}
In the end, we get
\begin{align*}
|\p^5_x(k\cir\eta)| & \lesssim \bound{\eps^{\mu-6}}{\eps^{\mu-\frac{13}{2}}}, \\
|\p^5_x(a\cir\eta)| & \lesssim \bound{\eps^{-6}}{\eps^{-\frac{13}{2}}}, \\
|\p^5_x(z\cir\eta)| & \lesssim \bound{\eps^{\mu-7}}{\eps^{\mu-\frac{15}{2}}} , \\
|\p^5_x(w\cir\eta)| & \lesssim \eps^{-7}.
\end{align*}

Using similar computations to those in this section, one can compute that
\begin{align*}
\eta^7_x |\p^5_\theta w\cir\eta| & \leq \bound{ \varepsilon^{-7}}{\varepsilon^{-9}}.
\end{align*} 
This bound, together with similar bounds we proved for $\p^n_\theta z\cir\eta, \p^n_\theta k\cir\eta,$ and $\p^n_\theta a\cir\eta$ throughout this section combine with \eqref{1/eta_x:bound} to give us \eqref{regularity}.

\section{Inversion of $\eta$}

\label{sec:expansion}

In this section, we will confine our attention to labels $ x \in (-\pi, \pi]$ with $|x| \leq \varepsilon^2$.

Since $w\cir\eta(\cdot, T_*)$ is $C^{4,1}$,it has the following Taylor expansion about $x_*$:
\begin{align}
w\cir \eta & = B^w_0 + B^w_1 (x-x_*) + B^w_2(x-x_*)^2 + B^w_3 (x-x_*)^3 + R^w_0(x) (x-x_*)^4.
\end{align}
Here
\begin{equation}
|B^w_0| \lesssim 1, |B^w_1| \lesssim \varepsilon^{-1}, |B^w_2| \lesssim \varepsilon^{-2}, |B^w_3| \lesssim \varepsilon^{-4}, |R^w_0| \lesssim \varepsilon^{\mu-5}.
\end{equation}

The flow $\eta(\cdot, T_*)$ also has the Taylor expansion
\begin{align}
\eta(x, T_*)- \xi_* & = a_3(x-x_*)^3 + a_4(x) (x-x_*)^4 \notag \\
\label{expansion:eta}
& = a_3(x-x_*)^3 + a_4(x_*)(x-x_*)^4 + a_5(x)(x-x_*)^5,
\end{align}
where $\xi_* : = \eta(x_*,T_*)$, $a_3 : = \frac{1}{6} \eta_{xxx}(x_*,T_*)$, $a_4(x_*) = \frac{1}{24} \p^4_x\eta (x_*,T_*)$,
\begin{equation}
a_4(x) : = \frac{\int^x_{x_*} \p^4_x\eta(y, T_*) (x-y)^{3} \: dy}{3! (x-x_*)^{4}} \hspace{3mm} \text{ and } \hspace{3mm} a_5(x) : = \frac{\int^x_{x_*} \p^5_x\eta(y, T_*) (x-y)^{4} \: dy}{4! (x-x_*)^{5}}.
\end{equation}
Here $a_3 \sim \varepsilon^{-3}$, $|a_4(x)| \lesssim \varepsilon^{\mu-4}$, and $|a_5(x)| \lesssim \varepsilon^{-6}$. Note that $|a^{-4/3}_3 a_4| \lesssim \varepsilon^\mu$.

Let $\theta = \eta(x,T_*)$. Lemma ~\ref{quartic inversion} implies that there exists a constant $C$ such that for all $x \in [-\varepsilon^2, \varepsilon^2]$ such that $|\theta-\xi_*| \leq C \varepsilon^{-3\mu}$ we have
\begin{align}
&(x-x_*) \notag\\
& = a^{-1/3}_3 (\theta-\xi_*)^{1/3} \big[ 1+ \tfrac{1}{3}\big( -a^{-4/3}_3a_4 (\theta-\xi_*)^{1/3} \big) + \tfrac{1}{3} \big( -a^{-4/3}_3a_4 (\theta-\xi_*)^{1/3} \big)^2 + \mathcal O(\varepsilon^{3\mu} |\theta-\xi_*|) \big] \notag \\
\label{expansion:2}
 & = a^{-1/3}_3 (\theta-\xi_*)^{1/3} \bigg[ 1+ \tfrac{1}{3}\big( -a^{-4/3}_3a_4 (\theta-\xi_*)^{1/3} \big) + \mathcal O( \varepsilon^{2\mu}|\theta-\xi_*|^{2/3}) \bigg] \\
 \label{expansion:3}
& = a^{-1/3}_3 (\theta-\xi_*)^{1/3} \big[ 1 + \mathcal O(\varepsilon^\mu |\theta-\xi_*|^{1/3} ) \big].
\end{align}

A quick bootstrap argument lets us conclude that this formula holds for all $x \in [-\varepsilon^2, \varepsilon^2]$. 
Furthermore, it is easy to show that there exists two constants $0<c<C$ such that
$$ \{ \theta : |\theta-\xi_*| \leq c\varepsilon^3 \} \subset \{ \theta : |x| \leq \varepsilon^2 \} \subset \{ \theta : |\theta-\xi_*| \leq C \varepsilon^3\}. $$
So we are working is a neighborhood of radius $\sim \varepsilon^3$ around $\xi_*$.

If we define
\begin{align*}
\aa^w_0 & : = B^w_0, \\
\aa^w_1 & : = a^{-1/3}_3 B^w_1,\\
\aa^w_2 & : = a^{-2/3}_3 B^w_2 -\tfrac{1}{3} a^{-5/3}_3 a_4(x_*) B^w_1,
\end{align*}
then we have
\begin{align*}
|\aa^w_0| \lesssim 1, \qquad |\aa^w_1| \lesssim 1,\qquad  |\aa^w_2| \lesssim 1, 
\end{align*}
and
\begin{align}
w(\theta, T_*) = \aa^w_0 + \aa^w_1(\theta-\xi_*)^{1/3} + \aa^w_2(\theta-\xi_*)^{2/3} + \OO(\eps^{-1}|\theta-\xi_*|).
\end{align}

Squaring \eqref{expansion:2} and cubing \eqref{expansion:3} gives us
\begin{align*}
(x-x_*)^2 & = a^{-2/3}_3 (\theta-\xi_*)^{2/3} - \tfrac{2}{3} a^{-2}_3 a_4 (\theta-\xi_*) + \OO(\eps^{2\mu+2}|\theta-\xi_*|^{4/3}), \\
(x-x_*)^3 & = a^{-1}_3(\theta-\xi_*) + \OO(\eps^{\mu+3} |\theta-\xi_*|^{4/3}).
\end{align*}
Therefore,
\begin{align}
\eta_x(x,T_*) & = 3 a_3 (x-x_*)^2 + [ 4a_4(x) + \p_x a_4(x) (x-x_*) ] (x-x_*)^3  \notag \\
& =: 3 a_3 (x-x_*)^2 + \tilde{a}_4 (x-x_*)^3 \notag \\
& = 3a^{1/3}_3 (\theta-\xi_*)^{2/3} + a^{-1}_3 (\tilde{a}_4-2a_4) (\theta-\xi_*) + \OO(\eps^{2\mu-1}|\theta-\xi_*|^{4/3}).
\end{align}
Using this formula, one can compute that
\begin{align*}
\eta_x(x,T_*)^{-1} & = \tfrac{1}{3} a^{-1/3}_3 (\theta-\xi_*)^{-2/3} - \tfrac{1}{9} a^{-5/3}_3 (\tilde{a}_4-2a_4) (\theta-\xi_*)^{-1/3} + \OO(\eps^{2\mu+1}).
\end{align*}
Since $a_4(x) = a_4(x_*) + \OO(\eps^{-5}|\theta-\xi_*|^{1/3})$ and $\tilde{a}_4(x) = 4a_4(x_*) + \OO(\eps^{-5}|\theta-\xi_*|^{1/3})$, it follows that
\begin{align}
\label{exp:1/eta_x}
\eta_x(x,T_*)^{-1} & = \tfrac{1}{3} a^{-1/3}_3 (\theta-\xi_*)^{-2/3} - \tfrac{2}{9} a^{-5/3}_3 a_4(x_*)(\theta-\xi_*)^{-1/3} + \OO(1).
\end{align}
Since $\p_x(w\cir\eta) = B^w_1 + 2B^w_2(x-x_*) + \OO(\eps^{-4}|x-x_*|^2)$, it follows that at time $T_*$ we have
\begin{align*}
\p_\theta w\cir\eta & = \eta_x^{-1} \p_x(w\cir\eta) \\
& = \big[ \tfrac{1}{3} a^{-1/3}_3 (\theta-\xi_*)^{-2/3} - \tfrac{2}{9} a^{-5/3}_3 a_4(x_*)(\theta-\xi_*)^{-1/3} + \OO(1) \big] \notag\\
&\qquad \qquad \cdot \big[B^w_1 + 2a^{-1/3}_3B^w_2(\theta-\xi_*)^{1/3} + \OO(\eps^{-2}|\theta-\xi_*|^2)\big] \\
& = \tfrac{1}{3} \aa^w_1 (\theta-\xi_*)^{-2/3} + \tfrac{2}{3}\aa^w_2(\theta-\xi_*)^{-1/3} + \mathcal O(\varepsilon^{-1}).
\end{align*}
This is the expansion for $\p_\theta w(\cdot, T_*)$ in Theorem ~\ref{thm:w:z:k:a}.

Now consider
\begin{align*}
\p^2_\theta w\cir\eta & = \eta_x(x,T_*)^{-2} \big[ \p^2_x(w\cir\eta) - \eta_{xx}(x,T_*) \p_\theta w\cir\eta(x,T_*) \big].
\end{align*}
Differentiating \eqref{expansion:eta} twice and using our above expansions for $(x-x_*)^2$ and $(x-x_*)^3$ gives us
\begin{align}
\eta_{xx}(x,T_*) & = 6a_3(x-x_*) + 12a_4(x_*) (x-x_*)^2 + [20a_5+10\p_x a_5 (x-x_*) + \p^2_x a_5 (x-x_*)^2] (x-x_*)^3 \notag \\
& = 6a_3(x-x_*) + 12a_4(x_*) (x-x_*)^2 +\OO(\eps^{-6} |x-x_*|^3) \notag \\
\label{expansion:eta_xx}
& = 6 a^{2/3}_3 (\theta-\xi_*)^{1/3} + 10 a^{-2/3}_3 a_4(x_*) (\theta-\xi_*)^{2/3} + \OO(\eps^{-3} |\theta-\xi_*|).
\end{align}
Using the fact that $\p^2_x(w\cir\eta) = 2B^w_2 + 6B^w_3(x-x_*) + \OO( \eps^{\mu-5}|x-x_*|^2)$ along with our expansion for $\p_\theta w\cir\eta$, \eqref{exp:1/eta_x}, and \eqref{expansion:eta_xx} now gives us our expansion for $\p^2_\theta w(\cdot, T_*)$ as stated in Theorem ~\ref{thm:w:z:k:a}.

Lastly, since
\begin{align*}
\p^3_\theta w\cir\eta & = \eta^{-3}_x \big[\p^3_x(w\cir\eta) - 3\eta_{xx} \eta_x \p^2_\theta w\cir\eta - \eta_{xxx} \p_\theta w\cir\eta \big],
\end{align*}
we can do similar computations to get the expansion for $\p^3_\theta w(\cdot, T_*)$.

To get the expansions for the variables $z,k,$ and $a$, similar computations can be made, except with the constants $B^z_j, B^k_j$, or $B^a_j$ instead of $B^w_j$. The computations for these variables are nicer because $B^z_1 = B^k_1 = B^a_1 = B^z_2 = B^k_2= B^a_2= 0$, but one should use fifth order expansions of $z\cir \eta, k\cir \eta$ and $a\cir \eta$. So we have
\begin{align*}
\aa^z_0 & : = B^z_0, \\
\aa^z_3 & : = a^{-1}_3 B^z_3, \\
\aa^z_4 & : = a^{-4/3}_3 B^z_4 - a^{-7/3}_3 a_4(x_*) B^z_3,
\end{align*}
and $\aa^k_0,\aa^k_3, \aa^k_4, \aa^a_0, \aa^a_3, \aa^a_4$ are defined analogously. When one does the computations, one obtains the expansions for $z,k,$ and $a$ listed in Theorem \ref{thm:w:z:k:a}.

Unlike the functions $w\cir\eta, z\circ\eta, k\cir\eta$, and $a\cir\eta$, which are in $C^{4,1}(\TT)$ at time $T_*$, the function $\varpi\cir\eta$ has only been proven to be in $C^{3,1}(\TT)$ at time $T_*$, so the Taylor expansion can only go to fourth order. However, we still have $B^\varpi_1=B^\varpi_2=0$ which allows us to get constants in our expansion. \qed


\appendix
\section{}

\subsection{Basic identities}
\label{sec:id}

The following equations are easy to compute from \eqref{eq:w:z:k:a}:
\begin{subequations}
\begin{align}
\label{c_t:identity}
-\tfrac{3}{2}\p_t(c\cir \psi) 
&= (\p_\theta w + 4a)\cir\psi (c\cir\psi).
\\
\label{c_tx:identity}
-\tfrac{3}{2}\p_t\big( \p_\theta c\cir\psi\big) 
& = (c\p_\theta^2 w)\cir\psi + \tfrac{3}{2}(\p_\theta c \p_\theta w) + \tfrac{3}{2}(\p_\theta c\p_\theta z) \cir\psi   + 4(\p_\theta a c + a\p_\theta c) \cir\psi.
\\
\label{k_t:identity}
-\tfrac{3}{2}\p_t(k\cir \psi) 
&= (c\p_\theta k)\circ\psi.
\\
\label{k_tx:identity}
-\tfrac{3}{2}\p_t\big(\p_\theta k\cir \psi) 
&= ( c \p_\theta^2 k) \cir\psi + (\p_\theta k \p_\theta w + \p_\theta k \p_\theta z )\cir\psi.
\\
\label{z_t:identity}
-\tfrac{3}{2}\p_t(z\cir\psi)  
&= (4az- \tfrac{1}{4}c^2 \p_\theta k)\cir\psi.
\\
\label{z_tx:identity}
-\tfrac{3}{2}\p_t(\p_\theta z\cir\psi) 
& = (\tfrac{1}{2}\p_\theta w \p_\theta z + \tfrac{3}{2} \p_\theta z^2 - \tfrac{1}{2} c \p_\theta c \p_\theta k - \tfrac{1}{4} c^2 \p_\theta^2 k)\cir\psi + 4(\p_\theta a z + a \p_\theta z)\cir\psi.
\\
\label{a_t:identity}
-\tfrac{3}{2}\p_t(a\cir\psi) 
& = (\p_\theta a c + 2a^2-c^2-4cz-2z^2)\cir\psi
\\
\label{a_tx:identity}
-\tfrac{3}{2}\p_t(\p_\theta a\cir\psi) 
& = (\p^2_\theta a c + 2\p_\theta a \p_\theta c + 2\p_\theta a \p_\theta z + 4a \p_\theta a) \cir\psi   -( [2c+4z]\p_\theta c + [4c+4z]\p_\theta z)\cir\psi.
\\
\label{c_t:phidentity}
-\tfrac{3}{2}\p_t(c\cir\phi) 
& = (4 ac + c\p_\theta c + c\p_\theta z)\cir\phi.
\\
\label{c_tx:phidentity}
-\tfrac{3}{2}\p_t( \p_\theta c\cir\phi) 
&= (c\p_\theta^2 c)\cir\phi + (c\p_\theta^2 z + 3\p_\theta c^2)\cir\phi  + 3(\p_\theta c\p_\theta z)\cir\phi + 4(\p_\theta a c + a\p_\theta c)\cir\phi.
\\
\label{k_tx:phidentity}
-\tfrac{3}{2}\p_t(\p_\theta k\cir\phi) 
&= (\p_\theta k \p_\theta w + \p_\theta k \p_\theta z)\cir\phi.
\\
\label{c_t:3dentity}
-\tfrac{3}{2}\p_t(c\circ\eta) 
&= ( \p_\theta z + 4a)\circ \eta (c\cir \eta).
\\
\label{k_t:3dentity}
-\tfrac{3}{2}\p_t(k\cir\eta) 
&= -(c\p_\theta k) \circ \eta\\
\label{k_tx:3dentity}
-\tfrac{3}{2}\p_t( \p_\theta k\cir\eta)
&= ( \p_\theta w \p_\theta k + \p_\theta z \p_\theta k - c \p^2_\theta k) \cir \eta.
\end{align}
\end{subequations}

\subsection{Quartic Inversion}
\label{sec:QI}

If $K$ is a field, and $K((z))$ denotes the field of formal Laurent series \footnote{ Formal Laurent series are formal power series which allow for finitely many terms of negative degree, not to be confused with the Laurent series in complex analysis, which may have infinitely many terms of negative degree but must converge in an annulus.} in the variable $z$. The field of Puiseux series in the variable $x$ is then defined to be the union $ \bigcup_{n > 0} K((x^{1/n})) $ which is itself a field. The most important result concerning Puiseux series is the following:  

\begin{theorem} [Puiseux-Newton]
\label{Puiseux-Newton}
 If $K$ is an algebraically closed field of characteristic 0, then the field $\bigcup_{n>0} K((x^{1/n}))$ of Puiseux series with coefficients in $K$ an algebraically closed field. Furthermore, given a polynomial $P(y) = \sum^N_{i=0} a_i(x) y^i$ with $a_i \in \bigcup_{n>0} K((x^{1/n}))$, the coefficients of the roots of $P$ in $y$ can be constructed using the method of Newton polygons.
\end{theorem}

\begin{proof}[Proof of Theorem~\ref{Puiseux-Newton}]See \cite[Chapter IV, Section 3]{walker}   or \cite[Section 8.3]{brieskornknorrer}.
\end{proof}

Of particular interest to us will be the following special case of the Puiseux-Newton theorem:  

\begin{theorem}[Analytic Puiseux-Newton]
\label{Analytic-Puiseux-Netwon}
If $\C\{x\}$ denotes the ring of convergent power series in $x$, and $f(x,y) \in \C\{x\}[y]$ is a polynomial of degree $m> 0$, irreducible in $\C\{x\}[y]$, then there exists a convergent power series $y \in \C\{z\}$ such that the roots of $f$ in $\bigcup_{n > 0} \C(( x^{1/n}))$ are all given by
\begin{equation*}
y(x^{1/m}), y(e^{2\pi i/m} x^{1/m}), \hdots, y(e^{2\pi i\tfrac{m-1}{m}} x^{1/m}).
\end{equation*}
It follows that in general if $f(x,y) \in \C\{x\}[y]$ then for each Puiseux series solution $\bar{y}$ of $f(x,\bar{y}(x)) = 0$ there exists some $y \in \C\{z\}$ and $m \leq \deg f$ such that $\bar{y}(x) = y(x^{1/m})$.
\end{theorem}

\begin{proof}[Proof of Theorem~\ref{Analytic-Puiseux-Netwon}]
See~\cite[Section 8.3]{brieskornknorrer}. 
\end{proof}

\begin{lemma}[Quartic Inversion]
\label{quartic inversion}
There exists a constant $R > 0$ and a nonempty open interval $I$ containing $0$ such that for all $a_3 \in \R^\times, a_4 \in \R$ there exists a function $y(x)$ defined for $x$ satisfying $|a^3_4 x| < R^3 a^4_3$ such that
\begin{equation*}
\big\{ (x,y) \in \R^2 : |a^3_4x| < R^3 a^4_3, a_4 y \in a_3 I, -x+ a_3y^3+a_4y^4 = 0 \big\} = \big\{ (x,y(x)) : |a^3_4 x| < R^3 a^4_3 \big\}.
\end{equation*}
Furthermore, $y(x)$ is an analytic function of $x^{1/3}$ satisfying the bounds
\begin{equation*}
\big\vert y(x)-a^{1/3}_3 x^{1/3}+ \tfrac{1}{3} a^{-5/3}_3 a_4 x^{2/3} - \tfrac{1}{3} a^{-3}_3 a^2_4 x \big\vert \lesssim a^{-13/3}_3 a^3_4 x^{4/3}
\end{equation*}
for all $|a_4^3 x| < R^3 a^4_3$, with the constant in the inequality independent of $a_3, a_4$.
\end{lemma}

\begin{proof}[Proof of Lemma~\ref{quartic inversion}] The case where $a_4 = 0$ is trivial, so we will prove our result in the case $a_4 \in \R^\times$. Define the recursive sequence $c_0 : = 1$,
$$ c_n : = \hspace{-10mm}\sum_{\substack{\vec{k} \in (\Z_{\geq 0})^4 \\ k_1+ k_2+k_3+k_4 = n-1}} \hspace{-10mm} c_{k_1} c_{k_2} c_{k_3}c_{k_4} - \tfrac{1}{3} \hspace{-7mm}\sum_{\substack{\vec{m} \in (\Z_{\geq 0})^3 \\ m_1+m_2+m_3 = n \\ 0 \leq m_i \leq n-1 }} \hspace{-7mm} c_{m_1} c_{m_2} c_{m_3}, $$
and define the formal power series $\bar{y} \in \R[[x]]$,
$$ \bar{y}(x) = \sum^\infty_{n=0} (-1)^n \tfrac{c_n}{3^n} x^{n+1}. $$
It is easy to check that $y_0(x) : = \bar{y}(x^{1/3})$ is a Puiseux series solution to the algebraic equation $-x + y^3_0 + y^4_0 = 0$. It follows from~\ref{Analytic-Puiseux-Netwon} that $\bar{y}$ must be convergent with some positive (possibly infinite) radius of convergence $R$. Now pick any $a_3 \in \R^\times, a_4 \in \R^\times$. If we define
$$ y(x) : = \tfrac{a_3}{a_4}\bar{y}(a^{-4/3}_3 a_4 x^{1/3}), $$
then it is easy to check that $y$ solves $-x +a_3y^3+a_4y^4 = 0$. 

Define the interval $I$ to be the range of $\bar{y}$, thought of as a function on $(-R,R)$ and define $f(x,y) = -x + a_3y^3 + a_4y^4$. Because $\p_x f = -1$ everywhere, we know that for each $y \in \R$ the equation $f(x,y) = 0$ has exactly one solution, $x$. Therefore, if $(x,y)$ is a point such that  $|x| < a^4_3a^{-3}_4R^3$, $y \in a_3a^{-1}_4I$, and  $f(x,y) =0$, then there exists $x'$ with $|x'| < a^4_3a^{-3}_4 R^3$ such that $y(x') = y$ and since $f(x',y) = f(x',y(x')) = 0$ we conclude that $x = x'$ and $y = y(x)$.

The remaining expansion follows from the fact that $c_1=1$ and  $c_2 =3$, combined with the fact that the power series $\bar{y}$ is convergent.
\end{proof}

\begin{theorem}
\label{thm:Taylor:funstuff}
There exists universal constants $C_1,C_2$ such that the following is true: Suppose that $I \subset \R$ is an interval, $x_0 \in I$, and $\theta \in C^{3,1}(I)$ is such that $L : = \|\p^4_x \theta\|_{L^\infty}$, $a_3 \in \R^\times$, and $\theta$ has the Taylor expansion $$ \theta(x) = \theta_0 + a_3 (x-x_0)^3 + a_4(x)(x-x_0)^4$$ at $x_0$. Then for all $x \in I$ such that $|\theta(x)-\theta_0| \leq C_1 \frac{a^4_3}{L^3}$, we have
\begin{equation*}
 (x-x_0) = a^{-1/3}_3 (\theta(x)-\theta_0)^{1/3} - \tfrac{1}{3} a^{-5/3}_3 a_4(x) (\theta(x)-\theta_0)^{2/3} + \tfrac{1}{3}a^{-3}_3 a_4(x)^2(\theta(x)-\theta_0) + R(\theta-\theta_0),
 \end{equation*}
 where $R$ is a $C^{0,\frac{1}{3}}$ continuous function satisfying
 \begin{equation*}
 |R(\theta-\theta_0)| \leq C_2 a^{-13/3}_3 a_4(x)^3 (\theta(x)-\theta_0)^{4/3}.
 \end{equation*}
\end{theorem}

\begin{proof}[Proof of Theorem~\ref{thm:Taylor:funstuff}] 
Assume without loss of generality that $a_3>0$. We know that $a_4 = (x-x_0)^{-4} (\theta - \theta_0 - a_3(x-x_0)^3)$ is $C^3$ away from $x_0$ and that
$$ a_4(x) = \tfrac{\int^x_{x_0} \p_x^4 \theta (t) (x-t)^3 \: dt}{3! (x-x_0)^4} $$
for all $x \neq x_0$. It follows from this formula that
$$ |a_4(x)| \leq \tfrac{L}{4!} \hspace{4mm} \text{ and } \hspace{4mm} |\p_x a_4(x)| \leq \tfrac{L}{3} \tfrac{1}{|x-x_0|} $$
for all $x \neq x_0$.

First define the function $f: \R\times (I-x_0)\rightarrow \R$,
$$ f(x,y) : = -x + a_3 y^3 + a_4(y+x_0) y^4. $$
Using our bounds on $|a_4|$ and $|\p_xa_4|$, we see that
\begin{align*}
\p_y f(x,y) & \geq y^2(3a_3- \tfrac{L}{2}|y|), \\
f(x, \tfrac{a_3}{L}) & \geq \tfrac{23}{24}\tfrac{a^4_3}{L^3}  - |x|, \\
f(x, - \tfrac{a_3}{L})& \leq |x| - \tfrac{23}{24}\tfrac{a^4_3}{L^3}.
\end{align*}
Therefore, if we define $A: = \{ |x| < \frac{23}{24} \frac{a^4_3}{L^3} \}$ and $B : \{ |y| < \frac{6a_3}{L}\}$, then for all $x \in A$ the function $f(x,\cdot): B \rightarrow \R$ is strictly increasing and has a zero in the interior of $B$. It follows from Corollary 1.1 in~\cite{Kumagai} that there exists a unique continuous function $h: A \rightarrow B$ such that
$$ \{ (x,y) \in A \times B : f(x,y) = 0 \} = \{ (x,h(x)) : x \in A\}. $$

Now define the function $F: \R^3 \rightarrow \R$,
$$ F(x,y,a) : = -x + a_3 y^3 + ay^4. $$
It is easy to check that if $|a| \leq \frac{L}{4!}$ and $|x| < \frac{23}{24}\frac{a^4_3}{L^3}$ then
\begin{align*}
\p_y F(x,y,a) \geq y^2(3a_3-\tfrac{L}{3!}|y|),\hspace{3mm}  F(x,\tfrac{a_3}{L},a)  > 0,\hspace{2mm} \text{ and } \hspace{2mm} F(x, -\tfrac{a_3}{L}, a)  < 0.
\end{align*}
Therefore, if $\tilde{A} : = \{ (x,a) : |x| < \frac{23}{24}\frac{a^4_3}{L^3}, |a| \leq \frac{L}{4!} \}$ and $\tilde{B} : = (-18 \frac{a_3}{L}, 18\frac{a_3}{L})$ then for all $(x,a) \in \tilde{A}$ the function $F(x,\cdot, a) :\tilde{B} \rightarrow \R$ is strictly increasing and contains a 0 in the interior of $\tilde{B}$. It follows from Corollary 1.1 of~\cite{Kumagai} that there exists a unique $H: \tilde{A} \rightarrow \tilde{B}$ continuous such that
$$ \{ (x,y,a) : |x| < \tfrac{23}{24}\tfrac{a^4_3}{L^3}, |y| < 18\tfrac{a_3}{L}, |a| \leq \tfrac{L}{4!}, F(x,y,a) = 0 \} = \{ (x, H(x,a), a) : |x| < \tfrac{23}{24}\tfrac{a^4_3}{L^3}, |a| \leq \tfrac{L}{4!} \}. $$

Our previous lemma \ref{quartic inversion} tells us that there exists a constants $R, C_2 > 0$ independent of $a_3$ or $L$ such that for all $|a| \leq \frac{L}{4!},$ $|x| < R^3 (4!)^3 \frac{a^4_3}{L^3}$ we have
$$ H(x,a) = a^{-1/3}_3 x^{1/3} - \tfrac{1}{3} a^{-5/3}_3 a x^{2/3} + \tfrac{1}{3} a^{-3}_3 a^2 x + \tilde{R}(x, a), $$
where $|\tilde{R}(x,a)| \leq C_2 a^{-13/3}_3 a^3 x^{4/3}$. Now suppose that $|x| < \frac{23}{24}\frac{a^4_3}{L^3}$. Then $|h(x)| < 6\frac{a_3}{L}< 18\frac{a_3}{L}$ and
$$ F(x,h(x), a_4(h(x)+x_0)) = f(x,h(x)) = 0, $$
so $h(x) = H(x, a_4(h(x)))$. It follows that if $C_1 : = \min\big( \frac{23}{24}, (R 4!)^3\big)$ then we have
\begin{align*}
h(x) & = a^{-1/3}_3 x^{1/3} - \tfrac{1}{3} a^{-5/3}_3 a_4(h(x)+x_0) x^{2/3} + \tfrac{1}{3} a^{-3}_3 a_4(h(x)+x_0) x + \tilde{R}(x, a_4(h(x)+x_0)) \\
& =: a^{-1/3}_3 x^{1/3} - \tfrac{1}{3} a^{-5/3}_3 a_4(h(x)+x_0) x^{2/3} + \tfrac{1}{3} a^{-3}_3 a_4(h(x)+x_0) x + R(x)
\end{align*}
for all $|x| < C_1\frac{a^4_3}{L^3}$. Our result now follows.
\end{proof}

\subsection*{Acknowledgments} 
S.S.~was supported by  NSF grant DMS-2007606 and the Department of Energy Advanced Simulation and Computing (ASC) Program.  I.N.~and V.V.~were supported by the NSF CAREER grant DMS-1911413.

%
%


\end{document}